\documentclass[a4paper,11pt,fleq n]{article}
\usepackage{amsmath}
\usepackage{amssymb}
\usepackage{theorem}
\usepackage{mathrsfs} 
\usepackage[usenames,dvipsnames]{pstricks}
\usepackage{euscript}
\usepackage{graphicx}
\usepackage{charter}
\usepackage{ulem}

\newcommand{\email}[1]{\href{mailto:#1}{\nolinkurl{#1}}}
\def\endproof{\vbox{\hrule height0.6pt\hbox{\vrule height1.3ex%
width0.6pt\hskip0.8ex\vrule width0.6pt}\hrule height0.6pt}}
\newtheorem{theorem}{Theorem}[section]
\newtheorem{lemma}[theorem]{Lemma}
\newtheorem{corollary}[theorem]{Corollary}
\newtheorem{proposition}[theorem]{Proposition}
\theoremstyle{plain}{\theorembodyfont{\rmfamily}%
\newtheorem{notation}[theorem]{Notation}}
\theoremstyle{plain}{\theorembodyfont{\rmfamily}%
\newtheorem{assumption}[theorem]{Assumption}}
\theoremstyle{plain}{\theorembodyfont{\rmfamily}%
}
\theoremstyle{plain}{\theorembodyfont{\rmfamily}%
\newtheorem{problem}[theorem]{Problem}}
\theoremstyle{plain}{\theorembodyfont{\rmfamily}%
\newtheorem{example}[theorem]{Example}}
\theoremstyle{plain}{\theorembodyfont{\rmfamily}%
\newtheorem{remark}[theorem]{Remark}}
\theoremstyle{plain}{\theorembodyfont{\rmfamily}%
\newtheorem{definition}[theorem]{Definition}}
\theoremstyle{plain}{\theorembodyfont{\rmfamily}%

\definecolor{labelkey}{rgb}{0,0.08,0.45}
\definecolor{refkey}{rgb}{0,0.6,0.0}
\definecolor{Brown}{rgb}{0.45,0.0,0.05}
\definecolor{dgreen}{rgb}{0.00,0.49,0.00}
\definecolor{dblue}{rgb}{0,0.08,0.75}
\definecolor{nido}{rgb}{0.6,0.0,0.4}
\RequirePackage[colorlinks,hyperindex]{hyperref}
\hypersetup{linktocpage=true,citecolor=dblue,linkcolor=dgreen}
\tolerance 2500

\DeclareMathOperator{\essup}{\mathrm{ess-sup}}

\DeclareMathOperator{\Argmin}{\text{\rm Argmin}}
\DeclareMathOperator{\Lip}{\text{\rm Lip}}
\DeclareMathOperator{\argmin}{\mathrm{argmin}}
\DeclareMathOperator{\spann}{\text{\rm span}}

\renewcommand{\leq}{\ensuremath{\leqslant}}
\renewcommand{\geq}{\ensuremath{\geqslant}}

\newcommand{\minimize}[2]{\ensuremath{\underset{\substack{{#1}}}%
{\text{\rm minimize}}\;\;#2 }}
\newcommand{\infimize}[2]{\ensuremath{\underset{\substack{{#1}}}%
{\text{\rm infimize}}\;\;#2 }}

\newcommand{\Frac}[2]{\displaystyle{\frac{#1}{#2}}}

\newcommand{\pair}[2]{{\left\langle{{#1},{#2}}\right\rangle}}
\newcommand{\abs}[1]{{\lvert {#1}\rvert}}
\newcommand{\norm}[1]{{\lVert {#1}\rVert}}
\newcommand{\menge}[2]{\big\{{#1}~\big |~{#2}\big\}} 
\newcommand{\Menge}[2]{\Big\{{#1}~\Big|~{#2}\Big\}}

\newcommand{\genf}{\ensuremath{F}}
\newcommand{\genff}{\ensuremath{G}}

\newcommand{\ww}{\ensuremath{{\mathsf{w}}}}
\newcommand{\LL}{\ensuremath{{\mathscr L}}}

\newcommand{\Loss}{\ensuremath{{\Upsilon}}}
\newcommand{\FF}{\ensuremath{{\EuScript{F}}}}

\newcommand{\C}{\ensuremath{{\mathcal C}}}
\newcommand{\CU}{\ensuremath{U}}
\newcommand{\CV}{\ensuremath{V}}
\newcommand{\B}{\ensuremath{{\mathcal{B}}}}

\newcommand{\W}{\ensuremath{{\mathcal W}}}
\newcommand{\fmap}{\ensuremath{{\Phi}}}
\newcommand{\dictfunc}{\ensuremath{\phi}}
\newcommand{\Pj}{\ensuremath{P}}
\newcommand{\MC}{\ensuremath{{\mathcal M}}}
\newcommand{\XC}{\ensuremath{{\mathcal X}}}

\newcommand{\YC}{\ensuremath{{\mathcal Y}}}
\newcommand{\ZC}{\ensuremath{{\mathcal Z}}}
\newcommand{\YY}{\ensuremath{{\mathsf Y}}}
\newcommand{\ZZ}{\ensuremath{{\mathsf Z}}}
\newcommand{\ZZZ}{\ensuremath{{Z}}}
\newcommand{\ud}{\ensuremath{{\mathrm d}}}

\newcommand{\emp}{\ensuremath{{\varnothing}}}

\newcommand{\cart}{\ensuremath{\raisebox{-0.5mm}{\mbox{\LARGE{$\times$}}}}}

\newcommand{\RR}{\ensuremath{\mathbb{R}}}
\newcommand{\RP}{\ensuremath{{\mathbb R}_+}}
\newcommand{\RPP}{\ensuremath{{\mathbb R}_{++}}}

\newcommand{\lip}[2]{\ensuremath{{\Lip}{({#1}};{{#2}}})}
\newcommand{\RPX}{\ensuremath{\left[0,+\infty\right]}}
\newcommand{\RX}{\ensuremath{\left]-\infty,+\infty\right]}}

\newcommand{\EE}{\ensuremath{\mathsf E}}
\newcommand{\PO}{\ensuremath{\mathsf P}}
\newcommand{\PP}{\ensuremath{ P}}

\newcommand{\QQ}{\ensuremath{ P}}
\newcommand{\QQI}{\ensuremath{ \widetilde{P}}}
\newcommand{\NN}{\ensuremath{\mathbb N}}
\newcommand{\KK}{\ensuremath{\mathbb K}}
\newcommand{\ev}{\ensuremath{\operatorname{ev}}}
\newcommand{\weakly}{\ensuremath{\:\rightharpoonup\:}}
\newcommand{\exi}{\ensuremath{\exists\,}}

\newcommand{\pinf}{\ensuremath{{+\infty}}}

\newcommand{\dom}{\ensuremath{\text{\rm dom}\,}}
\newcommand{\ran}{\ensuremath{\text{\rm ran}\,}}

\newcommand{\sign}{\ensuremath{\operatorname{sign}}}

\newcommand{\inte}{\ensuremath{\text{\rm int}\,}}

\numberwithin{equation}{section}
\thispagestyle{empty}
\begin{document}

\vskip -10mm

\title{\sffamily\LARGE Regularized Learning Schemes 
in Feature Banach Spaces\thanks{The work of P. L. Combettes
was supported by the 
CNRS MASTODONS project under grant 2016TABASCO.}%
\footnote{Contact author: P. L. Combettes, 
{\email{plc@math.ncsu.edu}}, phone: +1 (919) 515 2671.}}

\author{Patrick L. Combettes$^1$,\; Saverio Salzo$^2$,\; 
and Silvia Villa$^3$\\[3mm]
\small
\small $\!^1$North Carolina State University\\
\small Department of Mathematics\\
\small Raleigh, NC 27695-8205, USA\\
\small \email{plc@math.ncsu.edu}
\\[3mm]
\small
\small $\!^2$Massachusetts Institute of Technology 
and Istituto Italiano di Tecnologia\\
\small Laboratory for Computational and Statistical Learning\\
\small Cambridge, MA 02139, USA\\
\small \email{saverio.salzo@iit.it}
\\[3mm]
\small
\small $\!^3$ Politecnico di Milano\\
\small Dipartimento di Matematica\\
\small 20133 Milano, Italy\\
\small \email{silvia.villa@polimi.it}
}

\date{~}

\maketitle

\vspace{-12mm}
\begin{abstract} 
This paper proposes a unified framework for the investigation of
constrained learning theory in reflexive Banach spaces 
of features via regularized empirical risk minimization. 
The focus is placed on 
Tikhonov-like regularization with 
totally convex functions. This 
broad class of regularizers provides a flexible model 
for various priors on the features, including in particular hard 
constraints and powers of Banach norms. In such context, the 
main results establish 
a new general form of the 
representer theorem
and the consistency of the corresponding learning schemes 
under general conditions on the loss function, the geometry 
of the feature space, and the modulus of total convexity of the
regularizer.
In addition, the proposed analysis 
gives new insight into basic tools such as 
reproducing Banach spaces, feature 
maps, and universality.
Even when 
specialized to Hilbert spaces, this framework yields new results 
that extend the state of the art.
\end{abstract}





\section{Introduction}
\normalem

A common problem arising in decision sciences is to infer a
functional relation from the observation of a finite number of 
realizations $(x_i,y_i)_{1\leq i\leq n}$
of random input/output samples from an unknown common
distribution $P$ \cite{CuZh2007,Gyorfi02,Tsyba09,Vap98}.
Given a loss function $\ell$
and a set $\C$ of functions from the input set $\XC$ to the output
set $\YC$, the problem is formalized as follows
\begin{equation}
\label{eq:minrisk}
\infimize{f \in \C}R(f),\qquad R(f) = \int_{\XC \times \YC}
\ell(x,y,f(x)) P(\ud ( x, y)).
\end{equation}
Since $P$ is not known, the goal is to devise 
a {\em consistent learning scheme}, that is, a rule that assigns to
each sample 
$(x_i,y_i)_{1\leq i\leq n}$ an estimator $f_n\in\C$ such that,
$R(f_n) \to \inf R(\C)$  as $n$ becomes arbitrarily large.

In this paper, we consider estimators defined by Tikhonov-like
regularization. Given an empirical approximation $R_n$ of the
{\em risk} $R$ and a parameter space $\FF$, we consider a
hypothesis space of functions from $\XC$ to $\YC$ described
through a linear operator $A\colon \FF \to \YC^\XC$. An
estimator is defined by the problem
\begin{equation}
\label{eq:minrisk3}
\minimize{u\in\FF}{R_n(A u)}+\lambda_n G(u),
\end{equation}
where $(\lambda_n)_{n\in\NN}$ in $\RPP$ is a vanishing 
sequence and $G\colon\FF\to [0,+\infty]$ is a {\em regularizer},
that is, a function modeling some known properties of the target.
The above approach is classical, and related to the theory of
regularized M-estimators \cite{VanGeer2000}
and regularized empirical risk minimization \cite{Vap98}.
Many popular learning algorithms are off-springs of this approach,
including support vector machines, ridge regression, and sparsity 
based methods \cite{Vap98,ZouHastie2005}, to name a few.

The goal of this paper is to study the theoretical properties of a
large family of learning schemes of the form~\eqref{eq:minrisk3}
designed for problem \eqref{eq:minrisk}.
In particular, we consider very general forms of constraint sets,
parameterizations of the hypothesis space, and regularizers. 
Flexibility in the choice of these quantities plays a crucial
role in the incorporation of the information potentially available 
on the problem at hand. More precisely, we assume 
$\C$ to be a large set of functions defined by pointwise constraints 
on the function values, e.g., 
the set of positive functions, 
the parameter ({\em feature}) space $\FF$ to be 
a reflexive Banach space, and the regularizer to be a totally
convex function. 
Moreover, we take $\YC$ as a subset of a Banach space, so as to
deal with multi-task learning \cite{Evg2005,Zhang2013}
and regression with functional response \cite{Ram2005,Ferr06}.
Within this context, our contribution is 
twofold: we analyze the variational problem~\eqref{eq:minrisk3}, 
characterizing the form of its solutions,  
and establish sufficient conditions for the consistency 
of the corresponding estimators.

Problem in~\eqref{eq:minrisk3} is usually analyzed in 
reproducing kernel Hilbert spaces. Indeed, 
in this setting, the characterization of the form of the
minimizers is well known and is typically referred to as
representer theorem \cite{Kime70, Scho01}. It provides explicit
expressions for the minimizers in terms of the corresponding
reproducing kernel \cite{DeVito04}. The case of hypothesis spaces
which are Banach spaces is much less studied; see, e.g.,
\cite{Zhang2009,Zhang2012,Zhang2013}. A first contribution of our
paper is to further develop these studies providing a refined
analysis of the reproducing property in reflexive Banach spaces,
considering also the question of universality
\cite{Capo2008,CarDevToiUma10,Mic2006} in the presence of
constraints. A crucial difference with respect to the Hilbert
space setting is that in Banach spaces, feature maps, rather than
the kernel, become the natural quantities to study the problem,
since the kernel may even not exist.  Indeed, we prove a new  form
of the representer theorem for general probability measures and
extended-valued convex regularizers that characterizes the
minimizers in terms of the feature map, the subgradient of the
loss, and the subgradient of the regularizer. Moreover, we show
that the  computation of the solution of \eqref{eq:minrisk3} can be
reduced to that of the dual optimization problem, which is finite
dimensional and convex. This fact can be quite helpful in making
Banach space problems more practical numerically, in contrast with,
for instance, the results in \cite{Zhang2013} that lead to 
solving a nonlinear system of equations. 

Regarding  the statistical analysis, our primary concern is to
provide minimal but explicit conditions on
problem~\eqref{eq:minrisk3} to ensure consistency of its minimizers
with respect to problem~\eqref{eq:minrisk}. For that purpose, a
stability approach \cite{Demo09,SteChi2008} turns out to be
natural. Indeed, while different strategies can be considered,
e.g., based on covering numbers, fat-shattering dimensions
\cite{CuZh2007,Wang11,Wu2006}, or Rademacher complexities
\cite{Bartlett02,Mend03}, these results provide conditions in terms
of the complexity measures that need to be made explicit. As we
comment later in the paper (see Remark~\ref{rmkIuh7td14}), in the
general setting considered here, this turns out to be a problem in
its own right; moreover, stronger assumptions on the probability
measure and on the loss are usually required. Finally, we note that
approaches using Rademacher complexities  do not seem to be
applicable outside of the setting of Euclidean space-valued
functions and separable losses, since no suitable comparison
principle \cite[Theorem~4.12]{LedTal91} exists.

Our stability approach allows us to bypass these difficulties and
directly obtain explicit conditions under the general assumptions
outlined above.  More precisely, our statistical analysis is based
on a sensitivity theorem characterizing the dependence of the
solution of problem 
\eqref{eq:minrisk3} on the underlying probability measure.
The analysis is conducted in terms of the feature map and
it relies on various tools of convex analysis, 
geometry of Banach spaces, and probability in Banach spaces.
The modulus of total convexity of the regularizer $G$
and the Rademacher type of the dual of $\FF$ play a key role, 
but we remark that the existence of the kernel is not required.
Overall, we establish a non trivial extension of the approach in
Hilbert spaces considered in \cite{Demo09,SteChi2008}.

The contributions of the paper are the following.
\begin{itemize}
\item We consider a constrained risk minimization problem and a
general form of learning schemes based on Tikhonov-like
regularization with totally convex regularizers and 
Banach function spaces of Banach space-valued functions.
 \item We advance the theory of reproducing Banach spaces and study the problem of universality under constraints.
 \item We analyze the variational problem defining Tikhonov regularization, and provide a novel characterization of its solutions, 
 generalizing previous forms of the representer theorem.
 \item We provide minimal explicit sufficient conditions for consistency using a stability argument.
 \end{itemize}

Notation is provided in Section~\ref{sec:notation}. 
Section~\ref{sec:LBs} is 
devoted to the study of Banach spaces of vector-valued functions
and their description by operator-valued feature maps;
universality is studied in the presence of constraints  and
the representer and 
sensitivity theorems are established.
In Section~\ref{sec:Lvr}, the regularized learning scheme
is formalized and the main consistency theorems are
presented. Finally, the Appendix contains technical results
on the Lipschitz continuity of convex functions, 
totally convex functions, Tikhonov-like regularization,
and concentration inequalities in Banach spaces.

\section{Notation and basic facts}
\label{sec:notation}

We set $\RP=[0,+\infty[$ and $\RPP=\left]0,+\infty\right[$.
Let $\B\neq\{0\}$ be a real Banach space.
The closed ball of $\B$ of radius $\rho\in\RPP$ centered at 
the origin is denoted by $B(\rho)$. Let $p\in [1,+\infty]$. The conjugate of $p$ is
\begin{equation}
p^*=\begin{cases} +\infty & \text{if}\;\;p=1\\
p/(p-1) & \text{if}\;\;1<p<+\infty\\
1&\text{if}\;\;p=+\infty.
\end{cases}
\end{equation}

\noindent
\textbf{Convex Analysis}

\noindent
Let $\genf\colon\B\to\RX$. The \emph{domain of $\genf$} is 
$\dom \genf=\menge{u\in \B}{\genf(u)<\pinf}$ and $\genf$ is 
\emph{proper} if $\dom \genf\neq\emp$. 
Suppose that $\genf$ is proper and convex.
The \emph{Moreau subdifferential} of $\genf$ is the 
set-valued operator
\begin{equation}
\partial \genf\colon\B\to 2^{\B^*}\colon u\in \B \mapsto \menge{u^*
\in \B^*}{(\forall v\in\B)\;\genf(u)+\pair{v-u}{u^*}\leq \genf(v)},
\end{equation}
and its domain is $\dom\partial \genf=\menge{u\in \B}
{\partial \genf(u)\neq\emp}$. 
Moreover, for every $(u,v)\in \dom \genf \times \B$, we set
$\genf^\prime(u;v)=\lim_{t \to 0^+} (\genf(u+t v)-\genf(u))/t$.
If $\genf$ is proper and bounded from below and 
$\C \subset \B$ is  such that $\C\cap \dom \genf \neq \emp$, we put 
$\Argmin_\C \genf=\menge{u\in \C}{
\genf(u)=\inf \genf(\C) }$, and when it is a singleton we denote 
by $\argmin_\C \genf$ its unique element. Moreover, we set
\begin{equation}
(\forall\epsilon\in\RPP)\quad
\Argmin_\C^\epsilon \genf=\menge{u\in \C}{
\genf(u)\leq\inf \genf(\C)+\epsilon}\,.
\end{equation}
We denote by $\Gamma_0(\B)$ the class of functions 
$\genf\colon\B\to\RX$ which are proper, convex, and lower 
semicontinuous. We set $\Gamma_0^+(\B)=\menge{\genf\in\Gamma_0(\B)}{\genf \geq 0 }$.

\noindent
\textbf{Geometry of Banach spaces}

\noindent
We say that $\B$ is of Rademacher 
\emph{type} $q\in [1,2]$ \cite[Definition~1.e.12]{LindTzaf1979} if 
there exists $T\in [1,+\infty[$, so that for every
$n\in\NN\smallsetminus\{0\}$ and $(u_i)_{1\leq i\leq n}$ in $\B$,
\begin{equation}
\label{eq:type_p}
\int_0^1 \Big\lVert\sum_{i=1}^n r_i(t) u_i \Big\rVert^q\ud 
t\leq T\bigg(\sum_{i=1}^n \norm{u_i}^q\bigg)^{1/q},
\end{equation}
where $(r_i)_{i\in\NN}$ denote the Rademacher functions, that is,
for every $i\in\NN$, $r_i\colon [0,1]\to\{-1,1\}\colon t
\mapsto\sign(\sin (2^i\pi t))$. The smallest $T$ for which 
\eqref{eq:type_p} holds is denoted by $T_q$.
Since every Banach space is of Rademacher type $1$, this notion
is of interest for $q\in\left]1,2\right]$. Moreover, a Banach space 
of Rademacher type $q\in\left]1,2\right]$ is also of Rademacher
type $p\in\left]1,q\right[$.

The Banach space $\B$ is called \emph{smooth} \cite{Ciora90} if,
for every $u \in \B$ there exists a unique $u^* \in \B^*$ such that 
$\norm{u^*}=1$ and $\pair{u}{u^*} = 1$. The smoothness property
is equivalent to the G\^ateaux differentiability of the norm on $\B\setminus\{0\}$.
We say that $\B$ is \emph{strictly convex} if, for every $u$ and every $v$ in $\B$
such that $\norm{u}=\norm{v}=1$ and $u\neq v$, one has $\norm{(u+v)/2} <1$.
The \emph{modulus of convexity of $\B$} is 
\begin{equation}
\begin{array}{rl}
\delta_\B\colon\left]0, 2\right]& \to \RP\\
\varepsilon&\mapsto\inf\Big\{1-\Big\Vert\Frac{u+v}{2}\Big\Vert
\,\Big\vert\,(u, v)\in \B^2, \norm{u}=\norm{v}=1, 
\norm{u-v}\geq \varepsilon \Big\},
\end{array}
\end{equation}
and the \emph{modulus of smoothness of $\B$} is 
\begin{equation}
\begin{array}{rl}
\rho_\B\colon\RP&\to\RP\\
\tau&\mapsto\sup \Big\{ \Frac{1}{2}\big(\norm{u+ v} +
\norm{u-v}\big)-1\,\Big\vert\,  (u, v)\in \B^2,
\norm{u}=1,\ \norm{v}\leq \tau \Big\}.
\end{array}
\end{equation}
We say that $\B$ is \emph{uniformly convex} if $\delta_\B$ 
vanishes only at zero, and \emph{uniformly smooth} if 
$\lim_{\tau \to 0} \rho_\B(\tau)/\tau=0$ 
\cite{Beauzamy1985,LindTzaf1979}. Now let $q\in[1,\pinf[$. Then
$\B$ \emph{has modulus of convexity of power type} $q$ if there 
exists $c\in\RPP$ such that, for every
$\varepsilon\in\left ]0, 2\right]$,
$\delta_\B(\varepsilon)\geq c \varepsilon^q$, and it 
\emph{has modulus of smoothness of power type} $q$ if there exists 
$c\in\RPP$ such that, for every $\tau\in\RPP$, 
$\rho_\B(\tau)\leq c \tau^q$ \cite{Beauzamy1985,LindTzaf1979}. 
A smooth Banach space with modulus of smoothness of power type $q$
is of Rademacher type $q$ \cite[Theorem~1.e.16]{LindTzaf1979}.
Therefore, the notion of Rademacher type is weaker than that of 
uniform smoothness of power type, in particular it does not imply
reflexivity (see the discussion after
\cite[Theorem~1.e.16]{LindTzaf1979}).

If $p\in\left]1,\pinf\right[$, the \emph{$p$-duality map} of $\B$ 
is $ J_{\B,p}=\partial (\norm{\cdot}^p/p)$ \cite{Ciora90}, and hence 
\begin{equation}
\label{e:JJJ}
(\forall u\in \B)\quad J_{\B,p}(u)
=\menge{u^*\in\B^*}{\pair{u}{u^*}=\norm{u}^p\quad
\text{and}\quad \norm{u^*}=\norm{u}^{p-1} }.
\end{equation}
For every $u \in \B$ and every $\lambda \in \RP$, $J_{\B,p}(\lambda u) = \lambda^{p-1} J_{\B,p}(u)$ and $J_{\B,p}(-u) = - J_{\B,p}(u)$.
For $p=2$ we obtain the \emph{normalized duality map} $J_\B$.
Moreover, if $\B$ is reflexive, strictly convex, and smooth, then $
J_{\B,{p}}$ is single-valued and its unique selection, 
which we denote
also by $ J_{\B,{p}}$, is a bijection from $\B$ onto 
$\B^*$ and $J_{\B^*\!\!,p^*}=J_{\B,p}^{-1}$.  

\noindent
\textbf{Totally convex functions}

\noindent
Totally convex functions, were introduced in \cite{Butn97} and
further studied in \cite{ButIus2000,ButIusZal03,Zali02}. This
notion lies between strict convexity and strong convexity.
Suppose that $\B$ is reflexive and let 
 $\genf\colon\B\to\RX$ be a proper convex function.
The \emph{modulus of total convexity of $\genf$} \cite{ButRes06} is 
\begin{equation}
\begin{aligned}
\label{def:totconvmodi}
\psi\colon\dom  \genf\times\RR&\to\RPX\colon\\
(u,t)&\mapsto\inf\big\{\genf(v)-\genf(u)-\genf^\prime(u;v-u)
\,\big\vert\, v\in \dom \genf,\ \norm{v-u}=t \big\}
\end{aligned}
\end{equation}
and $\genf$ is \emph{totally convex at $u\in\dom \genf$} if, 
for every $t\in\RPP$, $\psi(u,t)>0$. The function $\genf$
is \emph{totally convex} if it is totally convex at every point of its domain.
Let $\psi$ be the modulus of total convexity of $\genf$. For every 
$\rho\in\RPP$ such that $B(\rho)\cap\dom \genf\neq\emp$,
the \emph{modulus of total convexity of $\genf$ on $B(\rho)$} is 
\begin{equation}
\label{def:modtotconvbound}
\psi_{\rho}\colon\RR\to\RPX\colon t\mapsto
\inf_{u\in B(\rho)\cap\dom \genf} \psi(u,t),
\end{equation}
and $\genf$ is \emph{totally convex on} $B(\rho)$ if
$\psi_{\rho}>0$ on $\RPP$. Moreover, $\genf$ is \emph{totally convex
on bounded sets} if, for every $\rho\in\RPP$ such that
$B(\rho)\cap\dom \genf\neq\emp$, it is totally convex on $B(\rho)$.
Let $\phi\colon\RR\to\RPX$ be such that $\phi(0)=0$ and
$\dom\phi\subset\RP$. We set
\begin{equation}
\label{eq:hatnotation}
\widehat{\phi}\colon\RR\to\RPX\colon t\mapsto 
\begin{cases}
0 &\text{if}\;\; t=0\\
\phi(t)/|t| &\text{if}\;\;t\neq 0.
\end{cases}
\end{equation}
The upper-quasi inverse of $\phi$ is \cite{Peno05,Zali02}
\begin{equation}
\label{eq:psidag}
\phi^\natural\colon\RR\to\RPX\colon
s\mapsto
\begin{cases}
\sup\menge{t\in\RP}{\phi(t)\leq s}&\text{if}\;\;s\geq 0\\
\pinf&\text{if}\;\;s<0.
\end{cases}
\end{equation}
Note that, for every $(t,s)\in\RP^2$,
$\phi(t)\leq s\Rightarrow t\leq\phi^\natural(s)$. 
We set
\begin{multline}
\label{e:A0}
\mathcal{A}_0=\big\{ \phi\colon\RR\to\RPX \big| \dom\phi\subset\RP,\;
\phi\;\text{is increasing on $\RP$,}\\
\phi(0)=0, (\forall t\in\RPP)\;\phi(t)>0 \big\}
\end{multline}
and
\begin{equation}
\label{e:A1}
\mathcal{A}_1=\Menge{\phi\in \mathcal{A}_0}{\widehat{\phi}
\;\text{is increasing on $\RP$,}
\;\lim_{t\to 0^+}\widehat{\phi}(t)=0}.
\end{equation}
Suppose that $\genf$ is totally convex at $u \in \dom \genf$. 
Then 
$\psi(u,\cdot) \in \mathcal{A}_0$ and 
$\psi(u,\cdot)^{\!\widehat{\phantom{a}}}\in\mathcal{A}_0$.
Moreover, if additionally $\partial \genf(u) \neq \varnothing$, then
$\psi(u,\cdot)^{\!\widehat{\phantom{a}}}\in\mathcal{A}_1$.
Suppose that
$\B$ is uniformly convex with power type, then, 
for every $r \in \RPP$, $\norm{\cdot}^r$ is totally convex on bounded sets
(See Appendix~\ref{subsec:totconv}).

\noindent
\textbf{Lebesgue spaces of vector-valued and operator-valued functions}

\noindent
When a Banach space is regarded as a measurable space it is
with respect to its Borel $\sigma$-algebra.
Let $(\ZC,\mathfrak{A},\mu)$ be
a $\sigma$-finite measure space and let $\YY$ be a separable real
Banach space with norm $\abs{\cdot}$. We denote by
$\MC(\ZC, \YY)$ the set of measurable functions from $\ZC$ into 
$\YY$. If $p \neq +\infty$, 
$L^p(\ZC, \mu; \YY)$ is the Banach space of all
(equivalence classes of) measurable functions $f\in\MC(\ZC,\YY)$ 
such that $\int_{\ZC} \abs{f}^p \ud \mu <+\infty$ and
 $L^\infty(\ZC,
\mu; \YY)$ is the Banach space of all (equivalence classes of)
measurable functions $f\in \MC(\ZC, \YY)$ which are
$\mu$-essentially bounded. Let $f\in L^p(\ZC,\mu;\YY)$. Then
$\norm{f}_p=\big(\int_{\ZC} \abs{f}^p \ud \mu \big)^{1/p}$ if 
 $p\neq +\infty$, and $\norm{f}_\infty
=\mu\verb0-0\essup_{z\in\ZC} \abs{f(z)}$ otherwise.
If $p\in\left]1,+\infty\right[$, 
$L^p(\ZC,\mu;\RR)$ is uniformly convex and uniformly smooth, and it
has modulus of convexity of power type $\max\{2,p\}$, and modulus 
of smoothness of power type $\min\{2, p\}$ \cite[p. 63]{LindTzaf1979}
and hence it is of Rademacher type 
$\min\{2,p\}$.
If $\ZC$ is  countable, $\mathfrak{A}=2^\ZC$, and $\mu$ is the
counting measure, 
we set $l^p(\ZC;\YY)=L^p(\ZC, \mu; \YY)$ and $l^p(\ZC)=L^p(\ZC,
\mu; \RR)$.
Let $\YY$ and $\ZZ$ be separable real Banach spaces. 
We denote by $\LL(\YY, \ZZ)$ the
Banach space of continuous linear operators from $\YY$ into $\ZZ$
endowed with the operator norm. 
A map  $\fmap\colon \ZC \to \LL(\YY,\ZZ)$
is \emph{strongly measurable}
if, for every $y\in\YY$, the function  
$\ZC\to\ZZ\colon z \mapsto\fmap(z)y$ is measurable.  In
such a case  the function $\ZC \to \RR\colon z \mapsto
\norm{\fmap(z)}$ is measurable \cite{Dincu2000}. If
$p \neq +\infty$,  $L^p[\ZC, \mu;
\LL(\YY,\ZZ)]$ is
the Banach space of all (equivalence classes of) strongly measurable
functions  $\fmap\colon \ZC \to \LL(\YY,\ZZ)$ such that
$\int_\ZC \norm{\fmap(z)}^p \mu(\ud z)<+\infty$ 
and 
$L^\infty[\ZC, \mu; \LL(\YY,\ZZ)]$ is the Banach space of all
(equivalence classes of) strongly measurable functions
$\fmap\colon \ZC \to \LL(\YY,\ZZ)$ such that
$\mu\verb0-0\essup_{z\in\ZC} \norm{\fmap(z)}<+\infty$ 
\cite{Blasco2010}.
Let $\fmap\in L^p[\ZC, \mu; \LL(\YY,\ZZ)]$. Then
$\norm{\fmap}_p=\big( \int_\ZC \norm{\fmap(z)}^p 
\mu(\ud z) \big)^{1/p}$ if $p\neq+\infty$, and 
$\norm{\fmap}_\infty=\mu\verb0-0\essup_{z\in\ZC} \norm{\fmap(z)}$
otherwise.

\noindent
\textbf{Probability}

\noindent
Let $(\Omega, \mathfrak{A}, \PO)$ be a probability space,
let $\PO^*$ be the associated outer probability. 
For every $\xi\colon\Omega\to\RR$ and $t\in\RR$, we set
\begin{equation}
[\xi>t]=\menge{\omega\in\Omega}{\xi(\omega)>t};
\end{equation}
the sets $[\xi<t]$, $[\xi\geq t]$, and $[\xi\leq t]$
are defined analogously.
Let $(\CU_n)_{n\in\NN}$ and $U$ be functions from $\Omega$ to 
$\B$. The sequence $(\CU_n)_{n\in\NN}$ converges 
\emph{in $\PO$-outer probability} to 
$\CU$, in symbols $\CU_n\overset{\PO^*}{\to}\CU$, if
\cite{Vandervaart96}
\begin{equation}
(\forall\varepsilon\in\RPP)\quad
\PO^*\big[\norm{\CU_n-\CU}>\varepsilon \big]\to 0,
\end{equation}
and it converges
$\PO^*\!\verb0-0\text{almost surely (a.s.)}$ to $U$ if
\begin{equation}
(\exi \Omega_0 \subset \Omega)\quad\PO^*\Omega_0=0
\quad\text{and}\quad
(\forall \omega\in \Omega\smallsetminus\Omega_0)\quad
U_n(\omega)\to U(\omega).
\end{equation}
The probability space $(\Omega,\mathfrak{A}, \PO)$ is 
\emph{complete} if, for every $A\in\mathfrak{A}$ such 
that $\PO(A)=0$, and every $B\subset A$, we have $B\in\mathfrak{A}$.

\section{Learning in Banach spaces}
\label{sec:LBs}
Basic tools such as feature maps, reproducing kernel Hilbert spaces, 
and representer theorems have played an instrumental role in the
development of Hilbertian learning theory 
\cite{Hofm08,Scho01,SteChi2008}. 
In recent years, there has been a marked interest in 
extending these tools to Banach spaces; see for instance
\cite{Fass15,Zhang2009,Zhang2013} 
and references therein. The primary objective of this section 
is to further develop the theory on these topics.

\subsection{Banach spaces of vector-valued functions and 
feature map representations}
\label{subsec:fmaps}

Sampling based nonparametric estimation naturally calls for
formulations involving spaces of functions for which the pointwise
evaluation operator is continuous. In the Hilbert space setting,
this framework hinges on the notions of a reproducing kernel
Hilbert space and of a feature map, which have been extensively
investigated, e.g., in \cite{CarDevToi06,SteChi2008}.  
On the other hand, the study of reproducing kernel Banach spaces
has been developed primarily in \cite{Zhang2009,Zhang2013}. 
However,  in the Banach space
setting, the continuity of the pointwise evaluation operators, the
existence of a kernel, and the existence of a feature map may no
longer be equivalent and further investigation is in order. Towards
this goal, we start with the following proposition which extends
\cite[Proposition~2.4]{CarDevToi06}.  

\begin{proposition}
\label{prop:feature_rkbs} 
Let $\XC$ be a nonempty set, let $\YY$ and $\FF$ be separable
real Banach spaces, and let $A\colon\FF\to\YY^{\XC}$ be a linear 
operator. Then the following are equivalent:
\vspace{-2ex}
\begin{enumerate}
\setlength{\itemsep}{1pt} 
\item
\label{item2:feature_rkbs} 
$A\colon\FF\to\YY^{\XC}$ is continuous for the topology of 
pointwise convergence on $\YY^{\XC}$. 
\item
\label{item1:feature_rkbs}
There exists a map $\fmap\colon \XC\to \LL(\YY^*,\FF^*)$ such that
\begin{equation}
\label{eq:fmap}
(\forall u\in\FF)(\forall x\in \XC) \quad (Au)(x)=\fmap(x)^*u.
\end{equation}
\item
\label{item3:feature_rkbs} 
$\ran A$ can be endowed with a Banach space structure such that 
the point-evaluation operators on $\ran A$ are continuous,
$A \colon \FF \to \ran A$ is continuous, and
the quotient operator of 
$A$ is a Banach space isometry from $\FF/\ker A$ onto $\ran A$. 
\end{enumerate}
\end{proposition}
\begin{proof} 
Set $\W=\ran A$ and $\EuScript{N}=\ker A$.
Let  
$\pi_\EuScript{N}\colon\FF\to \FF/\EuScript{N} \colon u\mapsto u+\EuScript{N}$ be the 
canonical projection operator and let
$\tilde{A}\colon\FF/\EuScript{N}\to \YY^{\XC}$ be the unique linear map such
that $A=\tilde{A}\circ \pi_\EuScript{N}$. Then $\tilde{A}$ is injective 
and $\ran \tilde{A}=\ran A$. Moreover, 
for every $x\in\XC$, we define the point-evaluation operator 
$\ev_x\colon\W\to\YY\colon f\mapsto f(x)$. We recall that
$A$ is continuous for the topology of 
pointwise convergence on $\YY^{\XC}$ if and only if,
for every $x \in \XC$, $\ev_x \circ A \colon \FF \to \YY$ is continuous.

\ref{item2:feature_rkbs}$\Rightarrow$\ref{item1:feature_rkbs}: 
Set $\fmap\colon\XC\to\LL(\YY^*,\FF^*)\colon x\mapsto
(\ev_x\circ A)^*$.

\ref{item1:feature_rkbs}$\Rightarrow$\ref{item2:feature_rkbs}: 
Let $x\in\XC$. Then, by \eqref{eq:fmap}, $\ev_x\circ A=\fmap(x)^*$
is continuous.

\ref{item2:feature_rkbs}$\Rightarrow$\ref{item3:feature_rkbs}: 
Since $\EuScript{N}$ is a closed vector subspace of $\FF$, the
quotient space
$\FF/\EuScript{N}$ is a Banach space with the quotient norm 
$\pi_{\EuScript{N}}u\mapsto
\norm{\pi_{\EuScript{N}} u}_{\FF/\EuScript{N}}=\inf_{v\in \EuScript{N}}\norm{u-v}$.
Thus, we  endow $\W$ with the Banach 
space structure transported from $\FF/\EuScript{N}$ by $\tilde{A}$,
i.e., for every $u\in \FF$,
$\norm{Au}=\norm{\tilde{A}\pi_\EuScript{N} u}
=\norm{\pi_\EuScript{N} u}_{\FF/\EuScript{N}}$.
Denote by $\abs{\cdot}$ the norm of $\YY$. Let $x\in\XC$ and
$f\in\W$.  
Then there exists $u\in \FF$ such
that $f=A u$, and hence $(\forall v\in\EuScript{N})$
$\abs{f(x)}=\abs{(\ev_x\circ A)(u+v)} \leq
\norm{\ev_x\circ A}\,\norm{u+v}.$ Taking the infimum over
$\EuScript{N}$, and recalling the definition of the quotient norm, we
get $\abs{f(x)}\leq \norm{\ev_x\circ A}\norm{\pi_\EuScript{N}u}_
{\FF/\EuScript{N}}=\norm{\ev_x\circ A}\norm{f}$.
Hence, $\ev_x\colon\W\to\YY$ is continuous. 
Finally, $A\colon \FF \to \W$ is continuous since 
$A=\tilde{A}\circ \pi_\EuScript{N}$.

\ref{item3:feature_rkbs}$\Rightarrow$\ref{item2:feature_rkbs}: 
Let $x\in\XC$. Since $A\colon\FF \to \W$ is continuous and
$\ev_x\colon\W\to\YY$ is continuous,  $\ev_x \circ
A\colon \FF \to \YY$ is likewise. 
\end{proof}

\begin{definition}
\label{def:fmap}
In the setting of Proposition~\ref{prop:feature_rkbs}, if
$A$ is continuous for the topology of 
pointwise convergence on $\YY^\XC$, then the unique map $\fmap$
defined in \ref{item1:feature_rkbs} is the \emph{feature map} 
associated with $A$ and $\FF$ is the feature space.
\end{definition}

\begin{definition}\label{def:RKBS}
Let $\XC$ be a nonempty set and let $\YY$ be a separable real
Banach space. Let $\W$ be a real
Banach space of functions from $\XC$ to $\YY$. Then 
\vspace{-1ex}
\begin{enumerate}
\item\label{def:RKBSi} $\W$ is a
\emph{pre-reproducing kernel Banach space} if, for every $x \in
\XC$, the point-evaluation operator $\ev_x\colon\W\to\YY\colon
f\mapsto f(x)$ is continuous \cite{Song13}.
\item\label{def:RKBSii}  $\W$
is a \emph{reproducing kernel Banach space} if it is a reflexive, strictly convex, and smooth 
pre-reproducing kernel Banach space.
\end{enumerate}
\end{definition}

\begin{remark}\ 
\label{rmk:featuremeas}
\begin{enumerate}
\item 
Proposition~\ref{prop:feature_rkbs} establishes that 
pre-reproducing kernel Banach spaces can always be built
via feature map representations.
We note that pre-reproducing kernel Banach spaces are called
\emph{function Banach spaces} in \cite{Bur84}.

\item
\label{rmk:featuremeasii}
Equation \eqref{eq:fmap} is equivalent to
\begin{equation}\label{eq:fmap2}
(\forall u\in\FF)(\forall x\in\XC)(\forall\ww^*\in\YY^*)\quad 
\pair{u}{\fmap(x)\ww^*}=\pair{(Au)(x)}{\ww^*},
\end{equation}
which shows that $A$ is injective if and only if 
$\big\{\fmap(x)\ww^*\,\big\vert\, x\in\XC,\,
\ww^*\in\YY^*\big\}$
is dense in $\FF^*$. Note that
this last denseness condition (hence the injectivity of $A$) is usually required in the
the current literature on reproducing kernel Banach 
spaces \cite{Zhang2009,Zhang2012,Zhang2013}.
We do not need this assumption.
\end{enumerate}
\end{remark}

\begin{proposition}
\label{pZhd7rT9-03}
Let $(\XC,\mathfrak{A}_\XC,\mu)$ be a $\sigma$-finite measure
space, let $\YY$ and $\FF$ be separable real Banach spaces, 
let $A\colon\FF\to\YY^{\XC}$ be linear and continuous for 
the topology of pointwise convergence on $\YY^{\XC}$, and let
$\fmap\colon\XC\to\LL(\YY^*,\FF^*)$ be the associated feature map.
Then the following hold:
\vspace{-2ex}
\begin{enumerate}
\setlength{\itemsep}{1pt} 
\item
\label{pZhd7rT9-03i}
$\fmap\colon\XC\to\LL(\YY^*,\FF^*)$ is strongly measurable if 
and only if $\ran A\subset\MC(\XC,\YY)$. 
\item
\label{pZhd7rT9-03ii}
Let $p\in[1,+\infty]$ and suppose that 
$\fmap\in L^p[\XC,\mu;\LL(\YY^*,\FF^*)]$. Then
$\ran A\subset L^p(\XC,\mu;\YY)$ and,
for every $u\in\FF$, $\norm{A u}_p
\leq\norm{\fmap}_{p}\norm{u}$.
\end{enumerate}
\end{proposition}
\begin{proof}
\ref{pZhd7rT9-03i}:
It follows from Pettis' theorem \cite[Theorem~II.2]{DieUhl77} 
and \eqref{eq:fmap2} that $\fmap\colon\XC\to\LL(\YY^*,\FF^*)$ 
is strongly measurable if and only if, for every $u\in\FF$, 
$Au$ is measurable. 

\ref{pZhd7rT9-03ii}:
Let $u\in\FF$ and note that, by \ref{pZhd7rT9-03i}, $Au$ is 
measurable. Moreover, by \eqref{eq:fmap}, $(\forall x\in\XC)$ 
$\abs{(Au)(x)}=\abs{\fmap(x)^*u}\leq\norm{\fmap(x)}\,\norm{u}$. 
\end{proof}

We now define a notion of universality
for spaces of vector-valued functions \cite{Capo2008,CarDevToiUma10}
with respect to a constraint set.

\begin{definition}
\label{def:universal}
Let $(\XC,\mathfrak{A}_\XC)$ be a measurable space, let $\YY$ 
be a separable uniformly convex real Banach space, and let 
$\W$ be a vector space of bounded measurable functions from 
$\XC$ to $\YY$.
Let $\C\subset \mathcal{M}(\XC;\YY)$ be a convex set.
\vspace{-2ex}
\begin{enumerate}
\setlength{\itemsep}{1pt} 
\item
\label{def:infuniversal}
$\W$ is $\infty\verb0-0$\emph{universal relative to} $\C$ if, 
for every probability measure $\mu$ 
on $(\XC, \mathfrak{A}_\XC)$ and for every 
$f \in\C\cap L^\infty(\XC, \mu; \YY)$, there exists 
$(f_n)_{n \in\NN}\in(\C\cap \W)^\NN$ such that 
$\sup_{n \in\NN} \norm{f_n}_{\infty}<+\infty$ and $f_n \to f$
$\mu\verb0-0\text{a.e.}$
\item
\label{def:puniversal}
Let $p\in\left[1, +\infty\right[$. The space $\W$ is
$p\verb0-0$\emph{universal relative to} $\C$ if, 
for every probability measure $\mu$ on $(\XC,
\mathfrak{A}_\XC)$, $\C\cap\W$ 
is dense in $\C\cap L^p(\XC,\mu;\YY)$.
\end{enumerate}
\vspace{-2ex}
When $\C=\mathcal{M}(\XC;\YY)$ the
reference to the set $\C$ is omitted.
\end{definition}

\begin{definition}
\label{def:fsdH612h23a}
Let $(\YY, \abs{\cdot})$ be a real normed vector space. The 
Attouch-Wets topology \cite{Atto91,Beer93} on the class $\mathscr{C}_\YY$ of 
nonempty closed subsets of $\YY$ is that induced by the following family of 
pseudometrics
\begin{equation}
(\forall\,\rho \in \RPP)(\forall\, (\mathsf{C}_1,\mathsf{C}_2) \in \mathscr{C}^2_\YY)
\quad  \mathrm{dist}_\rho(\mathsf{C}_1, \mathsf{C}_2)
= \sup_{\abs{\ww}\leq \rho} \abs{d_{\mathsf{C}_1}(\ww) - d_{\mathsf{C}_2}(\ww)},
\end{equation}
where $d_{\mathsf{C}}(\ww) = \inf_{\mathsf{y} \in \mathsf{C}} \abs{\mathsf{y} - \ww}$
is the distance function to the set $\mathsf{C}$.
\end{definition}

The following proposition shows that Definition~\ref{def:universal}
is an extension of the standard notion of universality in the context
of reproducing kernel Hilbert spaces \cite{CarDevToiUma10,Mic2006,SteChi2008}.

\begin{theorem}
\label{prop:C-universality}
Let $(\XC, \mathfrak{A}_\XC)$ be a measurable space, let $\YY$ be a
separable uniformly convex real Banach space, and let $\W$ be a
vector space of bounded measurable functions from $\XC$ to $\YY$.
Let $(\mathsf{C}(x))_{x\in\XC}$ be a family of closed convex 
subsets of $\YY$ containing $0$, let $\C=\menge{f\in \MC(\XC, \YY)}{
(\forall x\in\XC)\;f(x)\in\mathsf{C}(x)}$,
and let $p\in\left[1,\pinf\right[$.
Consider the following properties:
\vspace{-2ex}
\begin{enumerate}
\setlength{\itemsep}{1pt} 
\item[(a)]
$\W$ is $\infty\verb0-0$universal relative to $\C$.
\item[(b)]
$\W$ is $p\verb0-0$universal relative to $\C$.
\end{enumerate}
\vspace{-1.5ex}
Then the following hold:
\vspace{-2ex}
\begin{enumerate}
\setlength{\itemsep}{1pt} 
\item
\label{prop:C-universalityi}
Suppose that $x\mapsto \mathsf{C}(x)$ is measurable \cite{Cast77}.
Then (a)$\Rightarrow$(b).
\item
\label{prop:C-universalityii} 
Suppose that $\XC$ is a locally compact Hausdorff space and let 
$\mathscr{C}_0(\XC;\YY)$ be the space of continuous functions from
$\XC$ to $\YY$ vanishing at infinity \cite{Bour65}.
Suppose that $\W\subset\mathscr{C}_0(\XC;\YY)$ and 
that $x\mapsto \mathsf{C}(x)$ is continuous with respect to the 
Attouch-Wets topology.
Consider the following property:
\begin{enumerate}
\setlength{\itemsep}{1pt} 
\item[(c)]
$\C\cap\W$ is dense in $\mathcal{C}\cap\mathscr{C}_0(\XC;\YY)$ for
the uniform topology.
\end{enumerate}
\vspace{-0.5ex}
Then (a)$\Leftrightarrow$(b)$\Leftrightarrow$(c).
\end{enumerate}
\end{theorem}
\vspace{0.5ex}
\begin{proof}
\ref{prop:C-universalityi}:
Suppose that (a) holds and let $\mu$ be a 
probability measure on $(\XC, \mathfrak{A}_\XC)$.
We have $\W \subset L^\infty(\XC, \mu; \YY)$.
We derive from (a) and the 
dominated convergence theorem that $\C\cap \W$ is dense in 
$\C\cap L^\infty(\XC, \mu; \YY)$ for the topology of $L^p(\XC,
\mu; \YY)$. Next, let $f\in\C\cap L^p(\XC, \mu; \YY)$ and
let $\epsilon\in\RPP$. Since $L^\infty(\XC, \mu; \YY)$ is dense in
$L^p(\XC, \mu; \YY)$ for the topology of $L^p(\XC, \mu; \YY)$,
there exists $g\in L^\infty(\XC, \mu; \YY)$ such that
$\norm{f-g}_p\leq\epsilon/2$. The function
\begin{equation}
\Pj_{\mathsf{C}} (g)\colon \XC \to \YY\colon x \mapsto
\Pj_{\mathsf{C}(x)}(g(x))
\end{equation}
is well defined \cite[Proposition~3.2]{Goeb84} and its 
measurability follows from the application of 
\cite[Lemma~III.39]{Cast77} 
with $\varphi\colon\XC\times \YY \to \RR\colon (x,y) \mapsto 
-|y-g(x)|$ and $\Sigma=\mathsf{C}\colon \XC \to 2^\YY$.
Then $\Pj_{\mathsf{C}}(g)\in\C$ and, for every $x \in\XC$, 
since $\{0,f(x)\}\subset\mathsf{C}(x)$, 
\begin{equation}
\begin{cases}
\abs{\Pj_{\mathsf{C}(x)}(g(x))}
\leq \abs{\Pj_{\mathsf{C}(x)}(g(x))-g(x)}+\abs{g(x)}\leq2
\abs{g(x)} \\
\abs{\Pj_{\mathsf{C}(x)}(g(x))-f(x)} 
\leq \abs{\Pj_{\mathsf{C}(x)}(g(x))-g(x)} +\abs{g(x)-f(x)}
\leq 2 \abs{g(x)-f(x)}.
\end{cases}
\end{equation}
Therefore $\Pj_{\mathsf{C}}(g)\in L^\infty(\XC, \mu; \YY)$ and
$\norm{\Pj_{\mathsf{C}}(g) -f}_p\leq2 \norm{f-g}_p \leq
\epsilon$.

\ref{prop:C-universalityii}: 
(c)$\Rightarrow$(a): Let $\mu$ be a probability measure on 
$(\XC,\mathfrak{A}_\XC)$ and let $f\in\C\cap L^\infty(\XC,\mu;\YY)$. 
We denote by $\mathscr{K}(\XC;\YY)$ the space of continuous
functions from $\XC$ to $\YY$ with compact support. Since 
$\XC$ is completely regular, we derive from Lusin's theorem 
\cite[Corollary~1 in III.\S15.8]{Dincu1967} and Urysohn's lemma, 
that there exists a sequence $(g_n)_{n\in\NN}$ in
$\mathscr{K}(\XC;\YY)$ such that $g_n \to f$
$\mu\verb0-0\text{a.e.}$ and $\sup_{n\in\NN}\norm{g_n}_\infty
\leq \norm{f}_\infty$.  Let $n\in\NN$ and define the function
$\Pj_{\mathsf{C}} (g_n)\colon \XC \to \YY\colon x \mapsto
\Pj_{\mathsf{C}(x)}(g_n(x))$.
Let us prove that $\Pj_{\mathsf{C}} (g_n)$ is continuous. 
Let $x_0\in\XC$. Since $\lim_{x \to x_0} \mathsf{C}(x)=
\mathsf{C}(x_0)$ in the Attouch-Wets topology, there
exist a neighborhood $U_1$ of $x_0$ and $t\in\RPP$ such that,
for every $x\in U_1$, $\inf \abs{\mathsf{C}(x)}<t$. Moreover
there exist a neighborhood $U_2$ of $x_0$ and $q\in\RPP$ such
that, for every $x\in U_2$, $g_n(x)\in B(q)$. Now, fix 
$r\in\left[3q+t,\pinf\right[$. Then, for every $x\in U_1\cap U_2$,
since $r\geq3q+\inf \abs{\mathsf{C}(x)}$, it follows from
\cite[Corollary~3.3 and Theorem~4.1]{Peno05} that
\begin{align}
\label{eq:PCgcont}
&\hskip-6mm
\abs{\Pj_{\mathsf{C}} (g_n)(x)-\Pj_{\mathsf{C}} (g_n)(x_0)} 
\nonumber\\
&\leq \abs{\Pj_{\mathsf{C}(x)}(g_n(x))-
\Pj_{\mathsf{C}(x)}(g_n(x_0))} +
\abs{\Pj_{\mathsf{C}(x)}(g_n(x_0))-
\Pj_{\mathsf{C}(x_0)}(g_n(x_0))} 
\nonumber\\
&\leq\phi^{\natural}\big(2r\abs{g_n(x)-g_n(x_0)}\big)+
\abs{g_n(x)-g_n(x_0)}
+\phi^\natural(2r\,\mathrm{dist}_{2q+t}(\mathsf{C}(x), \mathsf{C}(x_0))),
\end{align}
where $\phi\in\mathcal{A}_0$ is the modulus of uniform
monotonicity of the normalized duality map of $\YY$ on $B(r)$, and,
for every $\rho\in\RPP$, $\mathrm{dist}_\rho$ is as in 
Definition~\ref{def:fsdH612h23a}.  
Hence, since $\lim_{x \to x_0} \mathrm{dist}_{2q +
t} (\mathsf{C}(x), \mathsf{C}(x_0))=0$, $\lim_{x\to x_0}
\abs{g_n(x)-g_n(x_0)}=0$, and $\lim_{s\to 0^+}\phi^\natural(s)=0$ 
by Proposition~\ref{prop:psinatural}\ref{psinat_increasing},
the continuity of $\Pj_{\mathsf{C}} (g_n)$ at $x_0$ follows.
In addition, since $0\in\bigcap_{x\in\XC}\mathsf{C}(x)$, the
support of $\Pj_{\mathsf{C}} (g_n)$ is contained in that of $g_n$.
Therefore, for every $n\in\NN$, $\Pj_{\mathsf{C}} (g_n)\in\C
\cap\mathscr{K}(\XC;\YY)$, $\norm{\Pj_{\mathsf{C}} (g_n)}_\infty
\leq 2 \norm{g_n}_\infty$ and, $(\forall x\in\XC)$
$\abs{\Pj_{\mathsf{C}(x)}( g_n(x))-f(x)}\leq2 \abs{g_n(x)-
f(x)}$. Hence
$\Pj_{\mathsf{C}} (g_n) \to f$ $\mu\verb0-0\text{a.e.}$
It follows from (c) that, for every $n\in\NN$, there exists 
$f_n\in\C\cap\W$ such that 
$\norm{f_n-\Pj_{\mathsf{C}} (g_n)}_\infty\leq1/(n+1)$. Therefore 
$\sup_{n\in\NN} \norm{f_n}_{\infty}\leq\sup_{n\in\NN}
(1+\norm{\Pj_{\mathsf{C}} (g_n)}_\infty)\leq1+2 \norm{f}_\infty$
and $f_n \to f$ $\mu\verb0-0\text{a.e.}$

(b)$\Rightarrow$(c):
We follow the same reasoning as in the proof of
\cite[Theorem~4.1]{CarDevToiUma10}. By contradiction, suppose that
$\C\cap \W$ is not dense in $\C\cap \mathscr{C}_0(\XC;\YY)$.
Since $\C\cap \W$ is nonempty and convex, by the Hahn-Banach
theorem, there exists $f_0\in\C\cap \mathscr{C}_0(\XC;\YY)$ and
$\varphi\in\mathscr{C}_0(\XC;\YY)^*$, and $\alpha\in\RR$ such
that 
\begin{equation}
(\forall f \in\C\cap\W)\quad\varphi(f)<\alpha<\varphi(f_0).
\end{equation}
Now, by \cite[Corollary~2 and Theorem~5 in III.\S19.3]{Dincu1967}
there is a probability measure $\mu$ on $\XC$ and a function $h \in
L^\infty(\XC, \mu; \YY^*)$ such that 
\begin{equation}
(\forall\, f\in\mathcal{C}_0(\XC;\YY))\quad \varphi(f)=\int_\XC \pair{f(x)}{h(x)} \ud \mu(x)\,.
\end{equation}
Since $\varphi \neq 0$, we have $h \neq 0$. Moreover $h \in
L^{p^*}(\XC, \mu; \YY^*)$. Therefore
\begin{equation}
\label{eq:separation}
(\forall\, f\in\C\cap \W)\quad 
\pair{f}{h}_{p,p^*}<\alpha<\pair{f_0}{h}_{p,p^*}.
\end{equation}
Let
$H^\alpha_-=\{ f\in L^{p}(\XC, \mu; \YY) \,\vert\,
\pair{f}{h}_{p,p^*}\leq\alpha \}$.
Then $H^\alpha_-$ is a closed half-space of $L^{p}(\XC, \mu; \YY)$.
Therefore, by \eqref{eq:separation}, $\overline{\C\cap \W}
\subset H^\alpha_-$ and $f_0 \notin H^\alpha_-$. 
Hence, $\C\cap \W$ is not dense in $\C\cap L^{p}(\XC, \mu; \YY)$.
\end{proof}

\begin{remark}
The Attouch-Wets topology considered in the statement of Theorem~\ref{prop:C-universality}
is also called bounded Hausdorff topology and is in fact a generalization
of the Hausdorff topology to non-compact sets.
\end{remark}

In the next proposition we show that in the Banach space setting,
the duality map (see Section~\ref{sec:notation}) is instrumental to
properly define a kernel. This will require the involved Banach spaces
to satisfy additional geometric properties.

\begin{proposition}
\label{prop:RKBS}
Under the assumptions of Proposition~\ref{prop:feature_rkbs},
let $\fmap\colon \XC\to \LL(\YY^*,\FF^*)$ be defined by
\eqref{eq:fmap} and set $\W=\ran A$. 
Let $\mathscr{B}(\YY^*,\YY)$ be the set of operators 
mapping bounded subsets of $\YY^*$ into bounded subsets of $\YY$.
Suppose that $\FF$ is reflexive, strictly convex, and smooth,
and let $p \in \left]1,+\infty\right[$.
Then $\W$ is a reproducing kernel Banach space and
there exists a unique 
$K_p\colon\XC\times\XC\to \mathscr{B}(\YY^*,\YY)$, called
kernel, such that
\begin{equation}
\label{eq:repformula}
(\forall\, u\in\FF)(\forall\, x\in\XC)(\forall\, y^* \in
\YY^*)\quad 
\begin{cases}
K_p(x,\cdot)y^*\in\W\\
\pair{Au}{J_{\W,{p}}(K_p(x,\cdot)y^*)}=\pair{(Au)(x)}{y^*}.
\end{cases}
\end{equation}
Moreover, we have
\begin{equation}\label{eq:kernel}
(\forall\, x\in\XC)(\forall\, x^\prime\in\XC)\quad
K_p(x,x^\prime)=
\fmap(x^\prime)^*\circ  J_{\FF,{p}}^{-1} \circ \fmap(x).
\end{equation}
\end{proposition}
\begin{proof}
Let $\EuScript{N}=\ker A$.
Proposition \ref{prop:feature_rkbs} implies that $\W$ is 
isometrically isomorphic to $\FF/\EuScript{N}$.
Define 
\begin{equation}
K_{p}\colon\XC\times\XC \to \mathscr{B}(\YY^*,\YY)\colon
(x,x^\prime)\mapsto 
\fmap(x^\prime)^*\circ  J_{\FF,{p}}^{-1} \circ \fmap(x)\,.
\end{equation}
Then \eqref{eq:fmap} yields
\begin{equation}
\label{eq:repformula0}
(\forall\, x\in\XC)(\forall\, y^*\in\YY^*)\quad
K_p(x,\cdot)y^*=A J_{\FF,{p}}^{-1} (\fmap(x) y^*).
\end{equation}
Since $\FF$ is reflexive, strictly convex, and smooth, 
$\FF/\EuScript{N}$ and $\W$ are likewise.  
Defining $\tilde{A}$ and $\pi_{\EuScript{N}}$ as
in the proof of Proposition~\ref{prop:feature_rkbs}, we have 
$\tilde{A}^*\circ J_{\W,{p}}\circ\tilde{A}=J_{\FF/\EuScript{N},{p}}$ 
and $ J_{\FF,{p}}=\pi_{\EuScript{N}}^*\circ  J_{\FF/\EuScript{N},{p}}
\circ \pi_\EuScript{N}$. 
Hence, $A^* \circ  J_{\W,{p}} \circ A= J_{\FF,{p}}$.
Therefore, it follows from \eqref{eq:repformula0} and
\eqref{eq:fmap} that, for  every $(x,u)\in\XC\times\FF$,
\begin{align}
\label{eq:repformula2}(\forall\, y^*\in\YY^*)\quad 
\pair{Au}{ J_{\W,p}(K_p(x,\cdot)y^*)} 
&=\big\langle A u,J_{\W,p}(A J_{\FF,p}^{-1} (\fmap(x) y^*))\big\rangle\\
\nonumber&=\pair{u}{\fmap(x) y^*}\\
&=\pair{(Au)(x)}{y^*}.
\end{align}
Finally if a kernel satisfies \eqref{eq:repformula}, it satisfies
\eqref{eq:repformula2} and hence \eqref{eq:repformula0}, and thus 
coincides with $K_p$.
\end{proof}

\begin{remark}\ 
\vspace{-2ex}
\begin{enumerate}
\setlength{\itemsep}{1pt} 
\item
Equation \eqref{eq:repformula} is a representation formula, meaning
that the values of the functions in $\W$ can be computed in terms
of the kernel $K_{p}$, which is said to be associated with 
the feature map $\fmap$.
\item
Definition~\ref{def:RKBS}\ref{def:RKBSii} is more general than
\cite[Definition~2.2]{Zhang2013}, since the latter requires that
both $\FF$ and $\YY$ be uniformly convex and uniformly smooth.
Thus, Proposition~\ref{prop:RKBS} extends 
\cite[Theorems~2.3 and 3.1]{Zhang2013}. To this respect, we note also that
what is essential to properly define a kernel is that the 
duality map is single valued and bijective, and this is equivalent to require 
strict convexity and smoothness only.
Moreover,
in Proposition~\ref{prop:RKBS}, the kernel is built from a 
feature map, a general $p$-duality map, and without any density assumption
(see Remark~\ref{rmk:featuremeas}\ref{rmk:featuremeasii}), which results in a
more general setting than that of \cite{Zhang2009,Zhang2013}. 
Finally, we emphasize that, when dealing with kernels in Banach spaces, 
there is no reason to restrict oneself to the normalized duality 
map. Rather, allowing general $p$-duality maps usually makes the 
computation of the kernel easier, as the following two examples show.
\end{enumerate}
\end{remark}
 
\begin{remark}
\label{scalarkernel}
In the setting of Proposition \ref{prop:RKBS},
consider the scalar case $\YY=\RR$ \cite{Zhang2009}. 
Then, for every $x\in\XC$, $\fmap(x)^*\in\FF^*$ and
the kernel becomes
\begin{equation}
\label{eq:scal_kerneldef}
K_{p}\colon\XC\times\XC\to\RR\colon(x,x^\prime) 
\mapsto \big\langle J_{\FF,{p}}^{-1}(\fmap(x)^*),\fmap({x^\prime})^*\big\rangle.
\end{equation}
Moreover, for every $x\in\XC$, $K_{{p}}(x,\cdot)=A [
J_{\FF,p}^{-1}(\fmap(x)^*)]$, and formula
\eqref{eq:repformula} turns into
\begin{equation}
(\forall\,u\in\FF)(\forall\,x\in\XC) \quad\pair{Au}{
J_{\W,{p}}(K(x,\cdot))}=(Au)(x)\,.
\end{equation}
It follows from the definitions of $K_{p}$ and $ J_{\FF,p}$ that
\begin{equation}
(\forall\, (x,x^\prime)\in\XC \times \XC) \quad
K_{{p}}(x,x)=\norm{\fmap(x)}^{p^*}\quad\text{and}\quad
\abs{K_{{p}}(x,x^\prime)}\leq K_{p}(x,x)^{{1/p}} 
K_{{p}}(x^\prime,x^\prime)^{{1/p^*}}.
\end{equation}
\end{remark}
 
\begin{example}[generalized linear model]
\label{ex:dictionary}
Let $\XC$ be a nonempty set, let $\YY$ be a separable
real Banach space with norm $\abs{\cdot}$, let $\mathbb{K}$ be a
nonempty countable set, 
let $r\in\left[1,+\infty\right[$. Let 
$(\dictfunc_k)_{k\in\KK}$ be a family of functions from 
$\XC$ to $\YY$, which, in this context, is
usually called a \emph{dictionary} \cite{Demo09,Stei09}. 
Assume that for every $x\in
\XC$, $(\dictfunc_k(x))_{k\in\KK}\in
l^{r^*}(\KK; \YY)$ and denote by
$\norm{(\dictfunc_k(x))_{k\in\KK}}_{r^*}$ its norm in
$l^{r^*}(\KK; \YY)$. 
Set 
\begin{equation}
A\colon l^{r}(\KK)\to\YY^{\XC}\colon u=(\mu_k)_{k\in\KK}
\mapsto\sum_{k\in\KK}\mu_k\dictfunc_k\; \text{(pointwise)}.
\end{equation}
Let $x\in\XC$.
By H\"older's inequality we derive that, for every $u\in\FF$,
$\abs{(A u) (x)}\leq
\norm{u}_r \norm{(\dictfunc_k(x))_{k\in\KK}}_{r^*}$,
which implies that $\ev_x \circ A$ is continuous. Therefore,
Proposition~\ref{prop:feature_rkbs} ensures that
\begin{equation}
\ran A=\Menge{f\in\YY^{\XC}}{\big(\exists\, u
\in l^{r} (\KK)\big)(\forall\, x\in\XC)\quad f(x)=\sum_{k\in
\KK} \mu_k \dictfunc_k(x)}
\end{equation}
can be endowed with a Banach space structure for which the
point-evaluation operators are continuous. Moreover
\begin{equation}
\ker A=\Menge{u\in l^{r}(\KK)}{(\forall\,x\in\XC)\quad
\sum_{k\in\KK}\mu_k\dictfunc_k(x)=0}
\end{equation}
and, for every $u\in l^r(\KK)$, 
$\norm{A u}=\inf_{v\in\ker A}\norm{u-v}_{r}$. Hence, for
every $f\in\ran A$,
\begin{equation}
\norm{f}=\inf \Menge{\norm{u}_r}{ u\in
l^{r}(\KK) \;\:\text{and}\:\;(\forall\,x\in\XC)\quad 
f(x)=\sum_{k\in\KK}\mu_k\dictfunc_k(x) }.
\end{equation}
Let us compute the feature map $\fmap \colon \XC \to \LL(\YY^*,
l^{r^*}\!(\KK))$. Let $x\in\XC$, let $y^*\in\YY^*$, and denote
by $\pair{\cdot}{\cdot}_{r,r^*}$ the canonical pairing 
between $l^r(\KK)$ and $l^{r^*}\!(\KK)$. Then, 
for every $u\in l^{r}(\KK)$, 
\begin{equation}
\label{eq:featmap}
\pair{u}{\fmap(x) y^*}_{r,r^*}=\pair{\fmap(x)^* u}{y^*}
=\pair{(A u)(x)}{y^*}=\sum_{k\in\KK}\mu_k
\pair{\dictfunc_k(x)}{y^*},
\end{equation}
which gives  
$\fmap(x)y^*=(\pair{\dictfunc_k(x)}{y^*})_{k\in\KK}$. Since 
$\LL(\YY^*, l^{r^*}\!(\KK))$ and  $l^{r^*}(\KK;\YY)$ are isomorphic
Banach spaces, the feature map can be identified with
\begin{equation}
\label{e:7-1!}
\fmap\colon\XC \to  l^{r^*}(\KK;\YY)\colon x\mapsto 
(\dictfunc_k(x))_{k\in\KK}.
\end{equation}
We remark that $\ran A$ is $p$-universal if, for every
probability measure $\mu$ on $(\XC,\mathfrak{A}_\XC)$, the span of
$(\dictfunc_k)_{k\in\KK}$ is dense in $L^p(\XC,\mu;\YY)$.  Now
suppose that $r>1$.  Since $l^r(\KK)$ is reflexive, strictly
convex, and smooth, Proposition \ref{prop:RKBS} asserts that
$\ran A$ is a reproducing kernel Banach space and that the
underlying kernel $K_r\colon\XC\times\XC\to\mathscr{B}(\YY^*,\YY)$
can be computed explicitly. Indeed, \cite[Proposition~4.9]{Ciora90}
implies that the ${r}$-duality map of $l^r(\KK)$ is
\begin{equation}
\label{eq:dualitymapellr}
J_{{r}}\colon l^r(\KK)\to l^{r^*}(\KK)\colon u=(\mu_k)_{k\in\KK}
\mapsto 
(\abs{\mu_k}^{r-1} \sign(\mu_k))_{k\in\KK} 
\end{equation}
Moreover, $J_{{r}}^{-1}\colon l^{r^*}(\KK) \to l^{r}(\KK)$ is
the ${r^*}$-duality map of $l^{r^*}(\KK)$ (hence it has the
same form as \eqref{eq:dualitymapellr} with $r$ replaced by $r^*$).
Thus, for every $(x,x^\prime)\in\XC \times \XC$ and every 
$y^*\in\YY$
\begin{equation}
K_{{r}}(x,x^\prime)y^*
=\fmap(x^\prime)^* \big(  J_{{r}}^{-1} (\fmap(x) y^*) \big)
=\sum_{k\in\KK}\abs{\pair{\dictfunc_k(x)}{y^*}}^{r^*-1}
\sign(\pair{\dictfunc_k(x)}{y^*})\dictfunc_k(x^\prime).
\end{equation}
In the scalar case $\YY=\RR$, this becomes
\begin{equation}
K_{{r}}(x,x^\prime)
=\pair{J_{{r}}^{-1}(\fmap(x))}{\fmap(x^\prime)}_{r,r^*}
=\sum_{k\in\KK} \abs{\dictfunc_k(x)}^{r^*-1}
\sign(\dictfunc_k(x))\dictfunc_k(x^\prime). 
\end{equation}
\end{example}

\begin{example}[Sobolev spaces]
Let $(d,k,m)\in(\NN\smallsetminus\{0\})^3$ and let
$p\in\left]1,\pinf\right[$.   Let $\XC\subset\RR^d$ be a nonempty
open bounded set with regular boundary and consider the Sobolev
space $W^{m,p}(\XC;\RR^k)$, normed with $\|\cdot\|_{m,p}\colon
f\mapsto\big(\sum_{\alpha\in\NN^d, \abs{\alpha}\leq m}
\norm{D^\alpha f}^p_p\big)^{1/p}$.  Recall that, if $m p>d$, then
$W^{m,p}(\XC;\RR^k)$ is continuously embedded in
$\mathcal{C}(\overline{\XC};\RR^k)$ \cite{Adams03}. Therefore 
\begin{equation}
(\exi\beta\in\RPP)(\forall x\in\XC)(\forall f\in W^{m,p}(\XC;\RR^k))
\quad\abs{f(x)}\leq \norm{f}_\infty\leq \beta \norm{f}_{m,p}.
\end{equation}
Moreover $W^{m,p}(\XC;\RR^k)$  is
isometrically isomorphic to a closed vector subspace of
$[L^p(\XC;\RR^k)]^n$, for a suitable $n\in\NN$, normed with 
$\norm{\cdot}_p\colon(f_1, \ldots, f_n)\mapsto \big( \sum_{i=1}^n
\norm{f_i}_p^p \big)^{1/p}$. 
Therefore, $W^{m,p}(\XC;\RR^k)$ is uniformly convex and smooth
(with the same moduli of convexity and smoothness as $L^p$).  This
shows that $W^{m,p}(\XC;\RR^k)$ is a reproducing kernel Banach
space and also that the associated feature map $\fmap$ is bounded.
Likewise, $W_0^{m,p}(\XC; \RR^k)$ is a reproducing kernel Banach
space endowed with the norm $\norm{\nabla \cdot}_p$, where this
time $\nabla\colon W_0^{m,p}(\XC; \RR^k) \to L^p(\XC;\RR^{k\times
d})$ is an isometry. For simplicity, we address the 
computation of the kernel for the space 
$W^{1,p}_0(\XC;\RR)$. In this case, the $p$-duality map is
\begin{equation}
\frac{1}{p}\partial\norm{\nabla \cdot}_p^p=-\Delta_p\colon
W^{1,p}_0(\XC;\RR)\to\big(W^{1,p}_0(\XC;\RR)\big)^*,
\end{equation}
where $\Delta_p$ is the $p$-Laplacian operator
\cite[Section~6.6]{Atto06}. Therefore, it follows from 
\eqref{eq:scal_kerneldef} that
\begin{equation}\label{eq:Dirichlet}
(\forall\, (x,x^\prime)\in \XC^2)\quad K_{p}(x,x^\prime)=
u(x^\prime),\quad\text{where}\quad u\neq 0\quad\text{and}\quad
-\Delta_p u=\ev_x.
\end{equation}
In the case when $\XC=[0,1]$, the kernel can be computed 
explicitly as follows
\begin{equation}
(\forall\, (x,x^\prime)\in \XC^2)\quad K_p(x,x^\prime) =
\begin{cases}
\dfrac{(1-x) x^\prime}{\big( x^{p-1}+(1-x)^{p-1}
\big)^{1/(p-1)}} &\text{if}\ x^\prime \leq x\\[3ex]
\dfrac{ (1- x^\prime)x}{\big( x^{p-1}+(1-x)^{p-1}
\big)^{1/(p-1)}} &\text{if}\ x^\prime \geq x
\end{cases}
\end{equation}
Finally, using a mollifier argument \cite[Theorem 2.29]{Adams03},
$W_0^{m,p}(\XC;\RR)_+$ is dense in $\mathscr{C}_0(\XC;\RR)_+$. 
Hence, by Theorem~\ref{prop:C-universality},
$W_0^{m,p}(\XC;\RR)$ is universal relative to the cone of
$\RP\!\verb0-0$valued functions.
\end{example}

\begin{remark}
Proposition~\ref{prop:RKBS} and the results pertaining to the
computation of the kernel are of interest in their own right. Note,
however, that they will not be directly exploited subsequently
since in the main results of Section~\ref{sec:setting}  knowledge
of a kernel will turn out not to be indispensable.
\end{remark}

\subsection{Representer and sensitivity theorems in Banach spaces}
\label{subsec:repthm}

In the classical setting, a representer theorem states that a
minimizer of a Tikhonov regularized empirical risk function defined
over a reproducing kernel Hilbert space can be represented as a
finite linear combination of the feature map values on the training
points \cite{Scho01}. The investigation in Banach spaces was 
initiated in \cite{MiccPontil94} and continued in \cite{Zhang2012,Zhang2013}.
In this section representer theorems are established in the general
context of Banach spaces, totally convex regularizers,
vector-valued functions, and approximate minimization. These
contributions capture and extend existing results. Moreover, we
study the sensitivity of such representations with respect to
perturbations of the probability distribution on $\XC\times\YC$. 

\begin{definition}
\label{def:lossandrisk}
Let $\XC$ and $\YC$ be nonempty sets, let $(\XC\times\YC,
\mathfrak{A},\PP)$ be a complete probability space, and let
$\PP_\XC$ be the marginal probability measure of $\PP$ on $\XC$. 
Let $\YY$ be a separable reflexive real Banach
space with norm $\abs{\cdot}$ and Borel $\sigma$-algebra
$\mathfrak{B}_\YY$. $\Loss(\XC\times\YC\times\YY)$ is the set of
functions $\ell\colon\XC\times\YC\times\YY\to\RP$ such that $\ell$
is measurable with respect to the tensor product $\sigma$-algebra
$\mathfrak{A}\otimes\mathfrak{B}_\YY$
and, for every $(x,y)\in\XC\times\YC$, $\ell(x,y,\cdot)\colon\YY
\to\RR$ is continuous and convex. A function in 
$\Loss(\XC\times\YC\times\YY)$ is a
\emph{loss}. The \emph{risk} associated 
with $\ell\in\Loss(\XC\times\YC\times\YY)$ and $\PP$ is
\begin{equation}
\label{eq:riskP}
R\colon\MC(\XC, \YY) \to \RPX\colon
f\mapsto\int_{\XC\times\YC} \ell\big(x,y,f(x)\big)\PP(\ud (x,y)).
\end{equation}
In addition, 
\vspace{-2ex}
\begin{enumerate}
\setlength{\itemsep}{1pt} 
\item\label{def:Gpi}
given $p\in [1,+\infty[$, $\Loss_p(\XC\times\YC\times\YY,\PP)$ is
the set of functions  $\ell\in\Loss(\XC\times\YC\times\YY)$ such
that
\begin{equation}
\label{2cwrey6T24a}
(\exi b\in L^1(\XC\times\YC, \PP;\RR))(\exi c\in\RP)(\forall
(x,y,\ww)\in\XC\times\YC \times \YY)\quad 
\ell(x,y,\ww)\leq b(x,y)+c\abs{\ww}^p;
\end{equation}
\item 
$\Loss_\infty(\XC\times\YC\times\YY,\PP)$ is the set of functions
$\ell\in\Loss(\XC\times\YC\times\YY)$ such that 
\begin{multline}
\label{2cwrey6T24b}
(\forall\rho\in\RPP)(\exi g_\rho\in L^1(\XC\times\YC,\PP;\RR))\\
(\forall\, (x,y)\in\XC\times\YC)(\forall\ww\in B(\rho))
\quad \ell(x,y, \ww)\leq g_\rho(x,y);
\end{multline}
\item\label{def:locLiploss}
$\Loss_{\YY,\text{loc}}(\XC\times\YC\times\YY)$ is the set of
functions $\ell\in\Loss(\XC\times\YC\times\YY)$ such that 
\begin{multline}
\label{e:loclip}
(\forall\rho\in\RPP)(\exi \lip{\ell}{\rho}\in\RPP)
(\forall (x,y)\in\XC\times\YC)(\forall (\ww,\ww^\prime)\in
B(\rho)^2)\\ \abs{\ell(x,y,\ww)-\ell(x,y,\ww^\prime)}\leq
\lip{\ell}{\rho}\abs{\ww-\ww^\prime}.
\end{multline}
\end{enumerate}
\end{definition}

\begin{remark}\
\label{r:contrisk} 
\vspace{-2ex}
\begin{enumerate}
\setlength{\itemsep}{1pt} 
\item
\label{r:contriskii} 
The properties defining the classes of losses 
introduced in Definition~\ref{def:lossandrisk} arise in 
the calculus of variations \cite{FonLeo07}. 
Let $p\in [1,+\infty]$ and suppose that 
$\ell\in\Loss_p(\XC\times\YC\times\YY,\PP)$. 
Then the risk \eqref{eq:riskP} is real-valued on 
$L^p(\XC,\PP_{\XC};\YY)$. Moreover,
since for every $(x,y)\in\XC\times\YC$, $\ell(x,y,\cdot)$ is
convex and continuous, $R\colon L^p(\XC,\PP_{\XC};\YY)\to\RP$ is 
convex and continuous \cite[Corollaries~6.51 and 6.53]{FonLeo07}.
\item
\label{r:contriski} 
If $\ell\in\Loss_p(\XC\times\YC\times\YY,\PP)$ 
then $\ell(x,y,\cdot)$ is bounded on bounded sets. Hence, by
Proposition~\ref{prop:Lip2Subrad}\ref{propii:Bound2Lip}, 
$\ell(x,y,\cdot)$ is Lipschitz continuous relative to bounded sets.
\item
\label{r:contriskiii} 
If $q\in\left[p,\pinf\right]$, then 
$\Loss_p(\XC\times\YC\times\YY,\PP)\subset\Loss_q
(\XC\times\YC\times\YY,\PP)$. 
\item
\label{r:contriskiv} 
Suppose that $\ell\in\Loss_{\YY,\text{loc}}(\XC\times\YC\times\YY)$ 
and that there exists $f\in L^\infty(\XC,\PP_{\XC};\YY)$ such that
$R(f)<+\infty$. 
Then $\ell\in\Loss_\infty(\XC\times\YC\times\YY,\PP)$ and 
\ref{r:contriskii} implies that 
$R\colon L^\infty(\XC,\PP_{\XC};\YY)\to\RP$ is 
convex and continuous.
\item
The following are consequences of
Propositions~\ref{prop:Lip2Subrad}\ref{propii:Bound2Lip} and
\ref{prop:LipGrowthp}\ref{prop:LipGrowthpi}:
\begin{enumerate}
\setlength{\itemsep}{1pt} 
\item
\label{rmk:lossloclip_vs_lpi}
Suppose that $\ell\in\Loss_1(\XC\times\YC\times\YY,\PP)$ and
let $c\in \RP$ be as in Definition~\ref{def:lossandrisk}\ref{def:Gpi}.
Then $\ell\in\Loss_{\YY,\text{loc}}(\XC\times\YC\times\YY)$ and 
$\sup_{\rho\in\RPP}\lip{\ell}{\rho}\leq c$. 
Hence $\ell$ is Lipschitz continuous in the third variable, 
uniformly with respect to the first two. Moreover, in this case, 
the inequality in \eqref{2cwrey6T24a} is true with 
$b=\ell(\cdot,\cdot,0)$.
\item
\label{rmk:lossloclip_vs_lpii}
Let $p\in\left]1,+\infty\right[$, let
$\ell\in\Loss_p(\XC\times\YC\times\YY,\PP)$, and 
suppose that the inequality in \eqref{2cwrey6T24a} 
holds with $b$ bounded and some $c\in\RP$. 
Then $\ell\in\Loss_{\YY,\text{loc}}(\XC\times\YC\times\YY)$ 
and $\ell(\cdot,\cdot,0)$ is bounded. Moreover,
for every $\rho\in\RPP$, $\lip{\ell}{\rho}\leq
(p-1)\norm{b}_\infty+3 c p\max\{1,\rho^{\,p-1}\}$.
\item 
\label{rmk:lossloclip_vs_lpiii}
Let $\ell\in\Loss_\infty(\XC\times\YC\times\YY,\PP)$. Then the
functions $(g_\rho)_{\rho\in\RPP}$ in \eqref{2cwrey6T24b} belong
to $L^\infty(P)$ if and only if
$\ell\in\Loss_{\YY,\text{loc}}(\XC\times\YC\times\YY)$ and
$\ell(\cdot,\cdot,0)$ is bounded. In this case, for every
$\rho\in\RPP$, $\lip{\ell}{\rho}\leq 2\norm{g_{\rho+1}}_{\infty}$.
\end{enumerate}
\end{enumerate}
\end{remark}

\begin{example}[$L^p$-loss]
\label{ex:distbased}
Consider the setting of Definition~\ref{def:lossandrisk} and 
let $p\in\left[1, \pinf\right[$.
Suppose that $\YC\subset\YY$, that 
$\int_{\XC \times \YC} \abs{y}^p \PP(\ud(x,y))<\pinf$,
and that
\begin{equation}
\label{dbloss}
(\forall\,(x,y,\ww)\in\XC\times\YC
\times\YY)\quad\ell(x,y,\ww)=\abs{y-\ww}^p.
\end{equation}
Then $\ell\in\Loss_p(\XC\times\YC\times\YY,\PP)$.
Moreover, suppose that $\YC$ is
bounded and set $\beta=\sup_{y\in\YC}\abs{y}$. Then
$\ell\in\Loss_{\YY,\text{loc}}(\XC\times\YC\times\YY)$ and
$(\forall \rho\in\RPP)$ $\lip{\ell}{\rho}\leq p
(\rho+\beta)^{p-1}$.  Indeed, the case $p=1$ is straightforward. If
$p>1$, it follows from \eqref{eq:Lipforploss} that, for every
$y\in\mathcal{Y}$ and every $(\ww, \ww^\prime)\in\YY^2$,
$\big\lvert\abs{\ww-y}^p-\abs{\ww^\prime-y}^p \big\rvert\leq p\,
\text{\rm max}\{\abs{y-\ww}^{p-1}, \abs{y-\ww^\prime}^{p-1}\} 
\abs{\ww-\ww^\prime}$.
Therefore, for every $(\ww, \ww^\prime)\in B(\rho)^2$ and every
$y\in\mathcal{Y}$, 
$\big\lvert\abs{\ww-y}^p-\abs{\ww^\prime-y}^p \big\rvert\leq
p (\rho+\beta)^{p-1} \abs{\ww-\ww^\prime}$. 
\end{example}

Now we propose a general representer theorem which involves the 
feature map from Definition~\ref{def:fmap}.

\begin{theorem}[Representer]
\label{thm:genrep} 
Let $\XC$ and $\YC$ be nonempty sets, let $(\XC\times\YC,
\mathfrak{A},\PP)$ be a complete probability space, and let
$\PP_\XC$ be the marginal probability measure of $\PP$ on $\XC$. 
Let $\YY$ be a separable reflexive 
real Banach space with norm $\abs{\cdot}$, 
let $\FF$ be a separable reflexive real Banach space, let
$A\colon\FF\to\MC(\XC,\YY)$ be linear and continuous with respect 
to pointwise convergence on $\YY^\XC$, and let $\fmap$ be the
associated feature map. Let $p\in[1,\pinf]$, let 
$\ell\in\Loss_p(\XC\times\YC\times\YY,\PP)$, let $R$ be the
risk associated with $\ell$ and $\PP$, and suppose that 
$\fmap\in L^p[\XC,\PP_\XC;\LL(\YY^*,\FF^*)]$. Set $F=R\circ A$, let 
$G\in\Gamma_0^+(\FF)$, let $\lambda\in\RPP$, 
let $\epsilon\in\RP$, and suppose that 
$u_{\lambda}\in\FF$ satisfies
\begin{equation}
\label{e:defula}
\inf\|\partial (F+\lambda  G )(u_{\lambda})\|
\leq\epsilon.
\end{equation}
Then there exists $h_{\lambda}\in L^{p^*}(\XC\times\YC, \QQ ;
\YY^*)$ such that
\begin{equation}
\label{eq:h}
(\forall\, (x,y)\in\XC\times\YC)\quad
h_{\lambda}(x,y)\in\partial_{\,\YY}\ell\big(x,y,
(Au_{\lambda})(x)\big)
\end{equation}
and
\begin{equation}
\label{eq:genrep}
(\exi e^*\in\FF^*)\quad\|e^*\|\leq\epsilon\quad\text{and}\quad
e^*-\EE_\QQ (\fmap h_{\lambda})\in 
\lambda \partial  G(u_{\lambda}),
\end{equation}
where $\fmap h_{\lambda}\colon \XC\times\YC \to \FF^*\colon (x,y)
\mapsto\fmap(x) h_{\lambda}(x,y)$ and, for every
$(x,y,\ww)\in\XC\times \YC \times \YY$,
$\partial_{\,\YY}\ell(x,y,\ww)=\partial\ell (x,y,\cdot)(\ww)$.
Moreover, the following hold:
\vspace{-2ex}
\begin{enumerate}
\setlength{\itemsep}{1pt} 
\item
\label{thm:genrepii}
Suppose that $p\neq+\infty$.
Let $(b,c)$ be as in Definition~\ref{def:lossandrisk}\ref{def:Gpi}.
If $p=1$, then $\norm{h_{\lambda}}_{\infty}
\leq c$; if $p>1$, then $\norm{h_{\lambda}}_{1}\leq ( p-1)
\norm{b}_1+3 p c (1+\norm{\fmap}_{p}^{p-1}
\norm{u_{\lambda}}^{p-1})$. 
\item
\label{thm:genrepiii}
Suppose that $p=+\infty$, that
$\ell\in\Loss_{\YY,\text{loc}}(\XC\times\YC\times\YY)$
and let $\rho\in\left]\norm{u_\lambda},\pinf\right[$.
Then  $h_{\lambda}\in  L^\infty(\XC\times\YC,\PP;\YY^*)$ and
$\norm{h_{\lambda}}_{\infty}
\leq \lip{\ell}{\rho\norm{\fmap}_\infty}$.
\end{enumerate}
\end{theorem}
\begin{proof}
Set
\begin{equation}\label{eq:R}
\Psi\colon L^p(\XC\times\YC,\PP;\YY) \to \RPX\colon g\mapsto
\int_{\XC\times\YC} \ell(z, g(z))  \QQ (\ud z).
\end{equation}
Since $\ell\in \Loss_p(\XC\times\YC\times\YY,\PP)$,  $\Psi$ is
real-valued and convex. Place
$L^{p}(\XC\times\YC, \QQ; \YY)$ and
$L^{p^*}(\XC\times\YC, \QQ; \YY^*)$ in duality by means of
the pairing 
\begin{equation}
\label{eq:lebduality}
\pair{\cdot}{\cdot}_{p, p^*}\colon(g,h)\mapsto\int_{\XC\times\YC}
\pair{g(z)}{h(z)}\,\QQ (\ud z).
\end{equation}
From now on, we denote by $L^p$ and $L^{p^*}$  the above cited
Lebesgue spaces, endowed with the weak topologies 
$\sigma(L^p,L^{p^*})$ and $\sigma(L^{p^*}, L^p)$, derived from the
duality \eqref{eq:lebduality}. Moreover, since $\ell\geq 0$, 
it follows from \cite[Theorem 21(c)-(d)]{Rock74} that
$\Psi\colon L^p\to\RR$ is lower
semicontinuous and 
\begin{equation}
\label{eq:sub}
(\forall g\in L^p)\quad
\partial \Psi(g)=\menge{h\in L^{p^*}}{h(z)\in\partial_{\,\YY}
\ell(z,g(z))\; \text{for}\ \QQ\text{-a.a.}~z\in\XC\times\YC}.
\end{equation}
Next, since $\fmap\in L^p[\XC,\PP_\XC;\LL(\YY^*,\FF^*)]$, it 
follows from Proposition~\ref{pZhd7rT9-03}\ref{pZhd7rT9-03ii}, 
that $A\colon\FF\to L^p(\XC,\PP_\XC;\YY)$ is continuous.
Therefore the map
$\widehat{A}\colon\FF \to L^p$ defined by 
\begin{equation}
(\forall\,u \in\FF)\quad\widehat{A}
u\colon \XC\times\YC\to\YY\colon (x,y)\mapsto (Au)(x)
\end{equation}
is linear and continuous.
Moreover, 
\begin{align}\label{eq:adjointiota2}
(\forall\,u\in\FF)(\forall\,h\in
L^{p^*})\quad \langle\widehat{A} u, h\rangle_{p, p^*} 
=\int_{\XC\times\YC} \pair{ u}{\fmap(x)
h(x,y)}\, \QQ(\ud(x,y))
=\pair{u}{\EE_\QQ (\fmap h)}.
\end{align}
Note that, in \eqref{eq:adjointiota2}, $\EE_\QQ (\fmap h)$ is
well defined, since $\fmap h$ is measurable \cite[Proposition
1.7]{Dincu2000}, and, for every $(x,y)\in\XC\times\YC$, 
$\norm{\fmap(x) h(x,y)}\leq \norm{\fmap(x)}
\abs{h(x,y)}$. Hence, by H\"older's inequality
$\int_{\XC\times\YC} \norm{\fmap(x) h(x,y)}
\QQ(\ud(x,y))<+\infty$,
and \eqref{eq:adjointiota2} implies that
$\widehat{A}^* : L^{p^*} \to \FF^*\colon h \mapsto\EE_\QQ (\fmap h)$.
Now, since $F=\Psi \circ \widehat{A}$, applying
\cite[Theorem~2.8.3(vi)]{Zali02} to
$\Psi\colon L^p\to\RR$ 
and  $\widehat{A}\colon \FF \to  L^p$ 
and, taking into account \eqref{eq:sub}, we get
\begin{multline}
\label{eq:subdiffPsiiota}
\partial F (u_{\lambda}) 
=\widehat{A}^* (\partial
\Psi(\widehat{A}
u_{\lambda})) \\
=\menge{\EE_\QQ (\fmap h)}{h\in
L^{p^*}\!, h(x,y)\in\partial_{\,\YY}
\ell(x,y, (Au_{\lambda})(x))
\text{ for } \QQ\text{-a.a.}~(x,y)\in\XC\times\YC}.
\end{multline}
Using \eqref{e:defula} and \cite[Theorem~2.8.3(vii)]{Zali02}, 
there exists $e^*\in B(\varepsilon)$
such that $e^*\in\partial(F+\lambda G)(u_\lambda)
=\partial F(u_{\lambda})+\lambda \partial
G(u_{\lambda})$. Hence, in view of
\eqref{eq:subdiffPsiiota}, there exists $h_{\lambda}\in
L^{p^*}$ satisfying
$h_{\lambda}(x,y)\in\partial_{\,\YY}\ell(x,y, (Au_{\lambda})(x))$
for $\QQ$-a.a.\ $(x,y)\in\XC\times\YC$ and $e^*-\EE_\QQ [\fmap
h_{\lambda}]\in\lambda\partial  G(u_{\lambda})$.
Since $\PP$ is complete, and  for every $(x,y)\in\XC\times\YC$,
$\dom\partial_{\,\YY}\ell(x,y,\cdot) \neq \emp$, we can modify 
$h_\lambda$ so that
$h_{\lambda}(x,y)\in\partial_{\,\YY}\ell(x,y,(Au_{\lambda})(x))$
holds for every $(x,y)\in\XC\times\YC$.
 
\ref{thm:genrepii}:
Let $(x,y)\in\XC\times\YC$. Since
$h_{\lambda}(x,y)\in\partial_{\,\YY} \ell(x,y, (Au_{\lambda})(x))$, 
\begin{equation}\label{eq:ineq_embedding}
\abs{(Au_{\lambda})(x)}=\abs{
\fmap(x)^*u_{\lambda}}\leq  \norm{\fmap(x)}
\norm{u_{\lambda}}.
\end{equation}
By Definition~\ref{def:lossandrisk}\ref{def:Gpi}, there exists
$b\in L^1(\XC\times\YC, \QQ;\RR)_+$
and $c\in\RPP$ such that, for every $\ww\in\YY$, 
$\ell(x,y,\ww)\leq b(x,y)+c \abs{\ww}^p$.
Therefore, it follows from Proposition~\ref{prop:LipGrowthp} and
\eqref{eq:ineq_embedding} that, if $p=1$, we have
$\abs{h_{\lambda}(x,y)}\leq c$ and, if $p>1$, we have
$
\abs{h_{\lambda}(x,y)}\leq  (p-1) b(x,y)+
3 p c (\norm{\fmap(x)}^{p-1} \norm{u_{\lambda}}^{p-1}+1)
$.
Hence, using Jensen's inequality, $\norm{h_{\lambda}}_{1}
\leq (p-1) \norm{b}_{1}+3 c p (1+\norm{\fmap}_p^{p-1}
\norm{u_{\lambda}}^{p-1})$. 

\ref{thm:genrepiii}:
Let $(x,y)\in\XC \times \YC$ be such that 
$\norm{\fmap(x)}\leq \norm{\fmap}_{\infty}$,
and set $\tau=\rho\norm{\fmap}_{\infty}$. We assume $\tau>0$.
Then \eqref{eq:ineq_embedding} yields
$\abs{(Au_{\lambda})(x)}<\tau$. Thus, since $B(\tau)$ is a
neighborhood of $(Au_{\lambda})(x)$ in $\YY$,
$\ell(x,y,\cdot)$  is Lipschitz continuous relative to
$B(\tau)$, with Lipschitz constant $\lip{\ell}{\tau}$ and
$h_{\lambda}(x,y)\in\partial_{\,\YY} \ell(x,y, (Au_{\lambda})(x))$,
Proposition~\ref{prop:Lip2Subrad}\ref{propi:Lip2Subrad} gives
$\abs{h_{\lambda}(x,y)}\leq\lip{\ell}{\tau}$.
\end{proof}

\begin{remark}\label{rmkdfhjhl3d05a}\  
\vspace{-2ex}
\begin{enumerate}
\setlength{\itemsep}{1pt} 
\item 
Condition~\eqref{e:defula} is a relaxation of the characterization
of $u_\lambda$ as an exact minimizer of $F+\lambda  G$ via Fermat's
rule, namely  $0\in\partial(F+\lambda  G)(u_\lambda)$.  
\item 
Using different methods, \cite[Theorem~5.7]{Zhang2013} gives a
representer theorem which 
holds only for reproducing kernel 
Banach spaces of vector-valued functions, discrete probabilities,
and $\epsilon=0$ (see the following Remark~\ref{rem:repthm1}). 
By contrast, Theorem~\ref{thm:genrep} is formulated
for general probability measures and in terms of 
the feature map. This underlines the fact that 
the kernel plays no role in the representation and does not even
need to exist.
\item
Theorem~\ref{thm:genrep} is sufficiently general to deal with an
offset space \cite{DeVito04}. To see this, 
let $\FF_1$ and $\FF_2$ 
be separable reflexive real Banach spaces, let 
$A_1\colon\FF_1\to\MC(\XC,\YY)$ and
$A_2\colon\FF_2\to\MC(\XC,\YY)$ be linear operators which are 
continuous with respect to pointwise convergence on $\YY^\XC$,  
let $\fmap_1\colon\XC \to \LL(\YY^*,\FF_1^*)$ 
and $\fmap_2\colon\XC \to \LL(\YY^*,\FF_2^*)$  be 
the feature maps associated with $A_1$ and $A_2$ respectively, and
let $G_1 \in \Gamma_0^+(\FF_1)$.
Suppose that, in Theorem~\ref{thm:genrep}, $\FF=\FF_1\times\FF_2$, 
$\epsilon=0$, and 
\begin{equation}
(\forall u=(u_1, u_2) \in \FF_1 \times \FF_2)\quad A u=A_1 u_1 +
A_2 u_2\quad\text{and}\quad G(u)=G_1(u_1).
\end{equation}
Then, setting $u_\lambda=(u_{1,\lambda},u_{2,\lambda})$,
\eqref{eq:h} and \eqref{eq:genrep} yield
\begin{equation}
\label{e:ar1}
(\forall\, (x,y)\in\XC\times\YC)\quad
h_{\lambda}(x,y)\in\partial_{\,\YY}\ell\big(x,y, (A_1
u_{1,\lambda})(x)+(A_2 u_{2,\lambda})(x)\big)
\end{equation}
and
\begin{equation}
\label{e:ar2}
-\EE_\QQ (\fmap_1 h_{\lambda})\in 
\lambda \partial  G_1(u_{1,\lambda})\quad\text{and}\quad \EE_\QQ
(\fmap_2 h_{\lambda})=0.
\end{equation}
This gives a representer theorem with offset space $\FF_2$. 
If we assume further that $\FF_1$ and $\FF_2$ are reproducing 
kernel Hilbert spaces of scalar functions, that $G_1=\|\cdot\|^2$, 
and that $p<+\infty$, the resulting special case of 
\eqref{e:ar1} and \eqref{e:ar2} appears in 
\cite[Theorem~2]{DeVito04}.
\end{enumerate} 
\end{remark}

\begin{corollary} 
\label{cor:repthm}
In Theorem~\ref{thm:genrep}, make the additional assumption that
$\FF$ is strictly convex and smooth, that there exists a convex even
function $\varphi\colon\RR\to\RP$ vanishing only at $0$ such that
\begin{equation}
\label{eq:regPhinorm}
G=\varphi\circ \norm{\cdot},
\end{equation}
and that $u_{\lambda}\neq 0$. Let $r \in \left]1, +\infty\right[$. Then there exist $e^*\in\FF^*$,
$h_{\lambda}\in L^{p^*}(\XC\times\YC,\QQ;\YY^*)$,
and $\xi({u_\lambda})\in\partial\varphi(\norm{u_\lambda})$ 
such that $\norm{e^*}\leq \epsilon$, \eqref{eq:h} holds, and
\begin{equation}
\label{eq:genrep0Bq}
J_{\FF.r}(u_\lambda)=
\frac{\norm{u_\lambda}^{r-1}}{\lambda\xi({u_\lambda})}
(e^*-\EE_\QQ [\fmap h_{\lambda}]).
\end{equation}
\end{corollary}
\begin{proof} 
Note $\partial\varphi(\RPP)\subset\RPP$ since $\varphi$ is strictly 
increasing on $\RPP$. It follows from Theorem~\ref{thm:genrep}
that there exist $h_{\lambda}\in L^{p^*}(\XC\times\YC,\QQ;\YY^*)$ 
and $e^*\in\FF^*$ such that \eqref{eq:h} and \eqref{eq:genrep} 
hold. Next, we prove that 
\begin{equation}
\label{eq:subdiffphinorm}
(\forall\, u\in\FF)\quad \partial G(u)=
\menge{u^*\in\FF^*}{\pair{u}{u^*}=\norm{u}\,
\norm{u^*}\ \text{and}\ \norm{u^*}\in\partial\varphi(\norm{u})}.
\end{equation}
It follows from \cite[Example~13.7]{Livre1} that, 
for every $u^*\in\FF^*$, $G^*(u^*)=\varphi^*(\norm{u^*})$.
Moreover, the Fenchel-Young identity entails that,
for every $(u,u^*)\in\FF\times \FF^*$, we have
\begin{align}
u^*\in\partial G(u) 
&\Leftrightarrow \varphi(\norm{u})+\varphi^*(\norm{u^*}) 
=\pair{u}{u^*}\nonumber\\
&\Leftrightarrow \pair{u}{u^*}=\norm{u} \norm{u^*}
\ \text{and}\ \norm{u^*}\in\partial\varphi(\norm{u})\,.
\end{align}
Set $u^*_\lambda=\big(e^*-\EE_\QQ (\fmap h_{\lambda})\big)/\lambda$.
Since $u_\lambda\not\in\{0\}=\Argmin_\FF
G=\menge{u\in\FF}{0\in\partial G(u)}$ and
$u^*_\lambda\in\partial G(u_\lambda)$, then $u_\lambda^*\neq 0$.
Now put
$v^*_\lambda=\norm{u_\lambda}^{r-1} u^*_\lambda/\norm{u_\lambda^*}$, then 
\eqref{eq:subdiffphinorm} yields
$\pair{u_\lambda}{v^*_\lambda}=
\norm{u_\lambda}^r$ and $\norm{u^*_\lambda}\in\partial
\varphi(\norm{u_\lambda})$. Moreover, $\norm{v^*_\lambda}=
\norm{u_\lambda}^{r-1}$. Hence, \eqref{e:JJJ} yields
$v^*_\lambda=J_{\FF,r}(u_\lambda)$ and \eqref{eq:genrep0Bq} follows.
\end{proof}

\begin{remark}
\label{rem:repthm1-}
In Corollary~\ref{cor:repthm}  let
$\varphi=\abs{\cdot}^r$.
Then \eqref{eq:genrep0Bq} specializes to 
\begin{equation}
\label{eq:genrep0Bq1}
J_{\FF,r}(u_{\lambda})=\frac{1}{r \lambda}\big( e^*-\EE_\QQ (\fmap h_{\lambda}) \big).
\end{equation}
If $\FF$ is a Hilbert space, $r=2$, and $\epsilon=0$, we obtain the
representation $u_{\lambda}= - (2 \lambda)^{-1}
\EE_\QQ (\fmap h_{\lambda})$, 
which was first obtained in \cite[Corollary~3]{DeVito04}. 
\end{remark}

\begin{remark}
\label{rem:repthm1}
Let $\epsilon=0$ and let 
$\QQ=n^{-1}\sum_{i=1}^n\delta_{(x_i,y_i)}$ be the empirical 
probability measure associated with the sample 
$(x_i,y_i)_{1\leq i\leq n}\in(\XC\times\YC)^n$. In this context,
we obtain a representation for the solution $u_\lambda$ to 
the regularized empirical risk minimization problem
\begin{equation}\label{eqdfhjhl3d05b}
\minimize{u\in\FF}{\frac 1 n \sum_{i=1}^n \ell (x_i, y_i, A u
(x_i))+\lambda G(u)}.
\end{equation}
Indeed \eqref{eq:genrep} implies that 
 there exists $(\ww^*_i)_{1\leq i\leq n}\in (\YY^*)^n$ such that 
\begin{equation}\label{eqdfhjhl3d05c}
\frac 1 \lambda \sum_{i=1}^n \fmap(x_i)\ww^*_i \in 
\partial G(u_\lambda)
\end{equation}
We observe that the coefficients $(\ww^*_i)_{1 \leq i \leq n}$, solve the dual 
problem
\begin{equation}\label{eqiuYh73n26a}
\min_{(\ww^*_i)_{1\leq i\leq n}\in (\YY^*)^n} 
\lambda G^*\Big(\frac 1 \lambda \sum_{i=1}^n \fmap(x_i)\ww^*_i \Big) 
+ \frac 1 n \sum_{i=1}^n \ell^*(x_i,y_i, - n \ww^*_i),
\end{equation}
of \eqref{eqdfhjhl3d05b}, where $\ell^*(x_i,y_i, \cdot)$ is the
conjugate of $\ell(x_i,y_i, \cdot)$.  Thus, if $G^*$ is
differentiable and $\YY$ is finite dimensional,
\eqref{eqdfhjhl3d05b} can be solved via the finite dimensional
convex problem \eqref{eqiuYh73n26a}, by inverting
\eqref{eqdfhjhl3d05c}, which yields
\begin{equation}
\label{eqdfhjhl3d05d}
u_\lambda = \nabla G^*\bigg(\frac 1 \lambda \sum_{i=1}^n 
\fmap(x_i)\ww^*_i \bigg).
\end{equation}
If $G$ is as in Corollary~\ref{cor:repthm}, then 
\eqref{eqdfhjhl3d05d} gives
$u_{\lambda}=J_{\FF,r}^{-1}\big(\sum_{i=1}^n\fmap(x_i)\ww^*_i\big)$.
Thus, $u_\lambda$, can be expressed in terms of the
feature vectors $(\fmap(x_i))_{1\leq i\leq n}$, for some vector
coefficients $(\ww^*_i)_{1\leq i\leq n}\in (\YY^*)^n$. This covers
the classical setting of representer theorems in scalar-valued
Banach spaces of functions \cite[Theorem~3]{Zhang2012} and improves
the vector-valued case of \cite[Theorem 5.7]{Zhang2013}. 
The dual variational framework \eqref{eqiuYh73n26a}
requires less restrictions and offers more flexibility in terms of 
solution methods than the fixed point approach proposed
in \cite{Fass15}, \cite[Theorem 23]{Zhang2009}, and 
\cite[Section~5.3]{Zhang2013}.
\end{remark}

\begin{example}
We recover a case-study of \cite{MiccPontil94}.
Let $\phi\colon\RP\to\RP$ be strictly increasing, continuous, and 
such that $\phi(0)=0$ and $\lim_{t\to+\infty}\phi(t)=+\infty$.
Define $\varphi\colon\RR\to\RP\colon t\mapsto\int_0^\abs{t}
\phi(s)\ud s$, which is strictly convex, even, and 
vanishes only at $0$.
Assume that $\varlimsup_{t\to 0}\varphi(2 t)/\varphi(t)<+\infty$, 
let $(\Omega,\mathfrak{S},\mu)$ be a measure
space, and let
$\FF=L_{\varphi}(\Omega,\mu;\RR)$ be the associated Orlicz 
space endowed with the Luxemburg norm induced by $\varphi$.
We recall that $\FF^*=L_{\varphi^*}(\Omega,\mu;\RR)$,
the Orlicz space endowed with
the Orlicz norm associated to $\varphi^*$ \cite{RaoRen91}.
Moreover, in this case the normalized duality map $J_{\FF^*}=
J_\FF^{-1}\colon \FF^*\to\FF$
can be computed. Indeed, by
\cite[Theorem~7.2.5]{RaoRen91}, 
we obtain that, for every $g\in\FF^*$, there exists 
$\kappa_g\in\RPP$ such that
$J_{\FF^*}(g)=\norm{g}\phi^{-1}(\kappa_g\abs{g})\sign(g)$. 
Given $(g_i)_{1\leq i\leq n}\in(\FF^*)^n$, 
$(y_i)_{1\leq i\leq n}\in\RR^n$, and
$\lambda\in\RPP$, the problem considered in 
\cite{MiccPontil94} is to solve
\begin{equation}
\minimize{u\in\FF}{ \frac 1 n \sum_{i=1}^n \ell (y_i,
\pair{u}{g_i})+\lambda \varphi(\norm{u})}.
\end{equation}
This corresponds to the framework considered in 
Corollary~\ref{cor:repthm} and Remark~\ref{rem:repthm1},
with $\XC=\FF^*$, $\YC=\YY=\RR$, 
$\QQ=n^{-1}\sum_{i=1}^n\delta_{(g_i,y_i)}$, and
$(\forall g\in\XC)(\forall u\in\FF)$
$(Au)(g)=\pair{u}{g}$.
Since, in this case, for every $g\in \XC$, $\fmap(g)=g$, we
derive from \eqref{eqdfhjhl3d05d} that there exist 
$\kappa\in\RPP$ and $(\alpha_i)_{1 \leq i \leq n} \in \RR^n$ 
such that
\begin{equation}
u_{\lambda}=\|u_{\lambda}\| \phi^{-1} 
\bigg(\kappa\Big\lvert\sum_{i=1}^n\alpha_i g_i
\Big\rvert\bigg)\sign\Big(\sum_{i=1}^n\alpha_i g_i\Big)
\;\text{{and}}\;
-{n\lambda\phi(\norm{u_\lambda})}\alpha_i\in
{\norm{u_\lambda}}\partial\ell(y_i,\cdot)(\pair{u_{\lambda}}{g_i}).
\end{equation}
\end{example}

We conclude this section with a sensitivity result in terms of a
perturbation on the underlying probability measure. 

\begin{theorem}[Sensitivity]
\label{thm:sensitivity2}
In Theorem~\ref{thm:genrep}, make the additional assumption that
$G$ is totally convex at every point of $\dom G$ and let $\psi$ 
be its modulus of total convexity.
Take $h_{\lambda}\in L^{p^*}(\XC\times\YC,\QQ;\YY^*)$ such that
conditions \eqref{eq:h}-\eqref{eq:genrep} hold. Let $\QQI$ be a 
probability measure on $(\XC\times\YC, \mathfrak{A})$ such that
$\ell\in\Loss_\infty(\XC\times\YC\times\YY,\QQI)$ and $\fmap$ is
$\QQI_\XC$-essentially bounded. Define 
\begin{equation}
\label{perturbrisk}
\widetilde{R}\colon\MC(\XC,\YC)\to[0,+\infty]: f\mapsto\int_{\XC
\times \YC} \ell(x,y, f(x)) \widetilde{P}(\ud (x,y)) 
\quad\text{and}\quad \widetilde{F}=\widetilde{R}\circ A. 
\end{equation}
Let $\tilde{\epsilon}\in\RPP$ and let $\tilde{u}_{\lambda}\in\FF$ be
such that $\inf\|\partial(\widetilde{F}+\lambda G)
(\tilde{u}_{\lambda}))\|\leq \tilde{\epsilon}$. Then the following
hold:
\vspace{-2ex}
\begin{enumerate}\setlength{\itemsep}{1pt} 
\item
\label{e:sensitivity0}
$h_{\lambda}\in L^1(\XC\times\YC, \QQI; \YY^*)$.
\item
\label{thm:sensitivity2ii}
$\psi(u_{\lambda},\cdot)^{\!\widehat{\phantom{a}}}
(\norm{\tilde{u}_{\lambda}-u_{\lambda}})\leq 
\big(\norm{\EE_{\QQI} (\fmap h_{\lambda})-\EE_{\QQ} (\fmap h_{\lambda})
}+ \epsilon+\tilde{\epsilon}\big)/\lambda.
$
\end{enumerate}
\end{theorem}
\begin{proof}
\ref{e:sensitivity0}:
Let $\gamma$ be the norm of $\fmap$ in $L^\infty[\XC, \QQI_{\XC};
\LL(\YY,\ZZ)]$
and let  $\rho\in\left]\gamma \norm{u_{\lambda}},+\infty\right[$.  
Since  $\ell\in\Loss_\infty(\XC\times\YC\times\YY,\QQI)$, there
exists $g\in L^1(\XC\times\YC, \QQI;\RR)$ such that
\begin{equation}\label{eq:boundloss}
(\forall\, (x,y)\in\XC \times \YC)(\forall\, \ww\in B(\rho+1)
)\quad \ell(x,y,\ww)\leq g(x,y).
\end{equation}
Let $(x,y)\in\XC\times\YC$ be such that $\norm{\fmap(x)}\leq
\gamma$. Then
$\abs{(Au_\lambda)(x)}\leq
\norm{\fmap(x)} \norm{u_{\lambda}}\leq
\gamma \norm{u_{\lambda}}<\rho
$.
Therefore, since $h_{\lambda}(x,y)\in\partial_{\,\YY}
\ell(x,y,(Au_{\lambda})(x))$, it follows from
Proposition~\ref{prop:Lip2Subrad}%
\ref{propi:Lip2Subrad}-\ref{propii:Bound2Lip} 
and \eqref{eq:boundloss} that 
$\abs{h_\lambda(x,y)}
\leq 2 \sup \ell(x,y,  B(\rho+1))\leq 2 g(x,y)$.
Hence $h_{\lambda}\in L^1(\XC\times\YC, \QQI;\YY^*)$.

\ref{thm:sensitivity2ii}:
Let $(x,y)\in\XC\times\YC$. Since
$h_{\lambda}(x,y)\in\partial_{\,\YY}
\ell(x,y, (Au_{\lambda})(x))$, we have
\begin{align}\label{eq:sensitivity3}
\nonumber \pair{ \tilde{u}_{\lambda}-u_{\lambda}}{\fmap(x)
h_{\lambda}(x,y)} &=\pair{(A \tilde{u}_{\lambda})(x)-(A
u_{\lambda})(x)}{h_\lambda(x,y)}\\
&\leq \ell(x,y, (A \tilde{u}_{\lambda})(x))-\ell(x,y,(A
u_{\lambda})(x)).
\end{align}
Since $\fmap$ is $\QQI_\XC$-essentially bounded and 
$h_{\lambda}\in L^1(\XC\times\YC, \QQI; \YY^*) $, 
$\fmap h_\lambda$ is $\QQI$-integrable. 
Integrating 
\eqref{eq:sensitivity3} with respect to $\QQI$ yields
\begin{equation}
\pair{ \tilde{u}_{\lambda}-u_{\lambda}}{\EE_{\QQI} (\fmap
h_\lambda)}\leq \widetilde{R}(A
\tilde{u}_{\lambda})-\widetilde{R}(A u_{\lambda}).
\end{equation}
Moreover, \eqref{eq:genrep} and \eqref{eq:totconvineq2} yield
\begin{equation}
\pair{\tilde{u}_{\lambda}-u_{\lambda}}{e^*-\EE_\QQ (\fmap
h_\lambda)}+\lambda \psi(u_{\lambda},
\norm{\tilde{u}_{\lambda}-u_{\lambda}})\leq \lambda
 G(\tilde{u}_{\lambda})-\lambda  G(u_{\lambda}).
\end{equation}
Summing the last two inequalities we obtain
\begin{align}\label{eq:sensitivity}
\nonumber \pair{\tilde{u}_{\lambda}-u_{\lambda}}{\EE_{\QQI} (\fmap
h_\lambda)-\EE_\QQ (\fmap h_\lambda)+e^*} &+ \lambda
\psi(u_{\lambda}, \norm{\tilde{u}_{\lambda}-u_{\lambda}}) \\
&\leq (\widetilde{F}+\lambda
 G)(\tilde{u}_{\lambda})-(\widetilde{F}+\lambda
 G)(u_{\lambda}).
\end{align}
Since there exists $\tilde{e}^*\in\FF^*$ such that 
$\norm{\tilde{e}^*}\leq\tilde{\epsilon}$ and 
$\pair{u_{\lambda}-\tilde{u}_{\lambda}}{\tilde{e}^*}
\leq (\widetilde{F}+\lambda  G)(u_{\lambda})-
(\widetilde{F}+\lambda  G)(\tilde{u}_{\lambda})$,
we have $(\widetilde{F}+\lambda  G)
(\tilde{u}_{\lambda})-(\widetilde{F}+\lambda  G)
(u_{\lambda})\leq \tilde{\epsilon}
\norm{u_{\lambda}-\tilde{u}_{\lambda}}$.
This, together with \eqref{eq:sensitivity}, yields
\begin{equation}
\lambda \psi(u_{\lambda},
\norm{\tilde{u}_{\lambda}-u_{\lambda}})\leq
(\epsilon+\tilde{\epsilon}) \norm{
\tilde{u}_{\lambda}-u_{\lambda}}+\norm{\EE_{\QQI} (\fmap
h_\lambda)-\EE_\QQ (\fmap h_\lambda)}
\norm{\tilde{u}_{\lambda}-u_{\lambda}}
\end{equation}
and the statement follows.
\end{proof}

\section{Learning via regularization}
\label{sec:Lvr}

We study statistical learning in Banach spaces 
and present the main results of the paper. 

\subsection{Consistency theorems}
\label{sec:setting}

We first formulate our assumptions. They involve the feature map 
from Definition~\ref{def:fmap}, as well as the loss and the risk
introduced in Definition~\ref{def:lossandrisk}.

\begin{assumption}\
\label{asssec4}
\vspace{-2ex}
\begin{enumerate}\setlength{\itemsep}{1pt} 
\item
\label{asssec4i}
$(\Omega,\mathfrak{S},\PO)$ is a complete probability space,
$\XC$ and $\YC$ are two nonempty sets, $\mathfrak{A}$ is 
a sigma algebra on $\XC\times\YC$ containing the singletons,
$(X,Y)\colon(\Omega,\mathfrak{S},\PO)\to
(\XC\times\YC,\mathfrak{A})$
is a random variable with distribution $\PP$ on
$\XC\times\YC$, and $\PP$ has marginal $\PP_\XC$ on $\XC$.
\item
\label{asssec4ii}
$\YY$ is a separable reflexive real Banach space, 
$\ell\in\Loss_{\YY,\text{loc}}(\XC\times\YC\times\YY)$, 
$R\colon\MC(\XC,\YY)\to\RPX$ is the risk associated with 
$\ell$ and $\PP$, and there exists 
$f\in L^\infty(\XC,\PP_{\XC};\YY)$ such that $R(f)<+\infty$. 
For every $\rho\in\RPP$, $\lip{\ell}{\rho}$ is as in
\eqref{e:loclip}.
\item
\label{ass:constr}
$\C$ is a nonempty convex subset of $\MC(\XC,\YY)$.
\item
\label{asssec4iii}
$\FF$ is a separable reflexive real Banach space, 
$q\in [2,\pinf[$, $\FF^*$ is of Rademacher type 
$q^*$ with Rademacher type constant $T_{q^*}$.
\item
$A\colon\FF\to \MC(\XC,\YY)$ is linear and continuous with respect 
to pointwise convergence on $\YY^\XC$, $\fmap$ is the feature 
map associated with $A$, 
$\fmap\in L^\infty[\XC,\PP_\XC;\LL(\YY^*,\FF^*)]$.
\item
\label{asssec4III}
$G\in\Gamma_0^+(\FF)$, $ G(0)=0$, the modulus of total convexity 
of $G$ is $\psi$, $\psi_0=\psi(0,\cdot)$, and $G$ is totally 
convex on bounded sets.
\item
$(\lambda_n)_{n\in\NN}$ is a sequence in $\RPP$ such that
$\lambda_n\to 0$. 
\item
\label{asssec5i}
$(X_i,Y_i)_{i\in\NN}$ is a sequence of independent copies of
$(X,Y)$. For every $n\in\NN\smallsetminus\!\{0\}$, 
$Z_n=(X_i,Y_i)_{1\leq i\leq n}$ and
\begin{equation}
\label{eq:emp-risk}
R_n\colon\MC(\XC,\YY)\times(\XC\times\YC)^n 
\to\RP\colon 
(f,(x_1,y_1),\ldots, (x_n,y_n) )\mapsto
\frac{1}{n}\sum_{i=1}^n\ell(x_i,y_i,f(x_i)).
\end{equation}
The function $\varepsilon \colon \RPP\to [0,1]$ satisfies
$\lim_{\lambda\to 0^+}\varepsilon(\lambda)=0$.
For every $n\in\NN\smallsetminus\{0\}$ and every $\lambda\in\RPP$, 
the function $u_{n,\lambda}\colon(\XC \times \YC)^n\to \FF$ 
satisfies
\begin{equation}
\label{eq:learnalgo}
(\forall z\in (\XC \times \YC)^n)
\quad u_{n,\lambda}(z)\in 
\Argmin_\FF^{\varepsilon(\lambda)}(R_n(A\cdot,z)+\lambda G).
\end{equation}
\end{enumerate}
\end{assumption}

In the context of learning theory, $\XC$ is the input space and $\YC$
is the output space, which can be considered to be embedded in the
ambient space $\YY$. The probability distribution $\PP$ describes a
functional relation from $\XC$ into $\YC$ and $R$ quantifies the
expected loss of a function $f\colon\XC\to\YY$ with respect to
the underlying distribution $\PP$. 
The set $\C$ models a priori constraints. 
Since $\mathcal{M}(\XC,\YY)$ is poorly structured, 
measurable functions are handled via the Banach feature space
$\FF$ and the feature map $\fmap$.
Note that the resulting space of functions is only a pre-reproducing
kernel Banach space in the sense of \cite{Song13}, since a kernel is not
required. 
Under the provision that the range of $A$ is 
universal relative to $\C$ (see Definition~\ref{def:universal}) 
every function $f\in\C$ can be approximately represented by a 
feature $u\in\FF$ via $f\approx Au$. Since the true risk $R$ depends
on $P$, which is unknown, the empirical risk $R_n$ is constructed 
from the available data, namely a realization of $Z_n$.
In \eqref{eq:learnalgo}, $u_{n,\lambda}$ is obtained 
by approximately minimizing a regularized empirical risk. 
Regularization is achieved by the addition of the convex
function $G$, which will be asked to fulfill certain compatibility 
conditions with the constraint set $\C$, 
e.g., $\overline{\dom  G}=A^{-1}(\C)$. 
The objective of our analysis can be stated as follows.

\begin{problem}[consistency]
Consider the setting of Assumption~\ref{asssec4}. 
The problem is to approach the infimum of 
the risk $R$ on $\C$ by means of approximate solutions 
\begin{equation}
\label{eZhd7rT9-30a}
u_{n,\lambda_n}(Z_n)\in 
\Argmin_\FF^{\varepsilon(\lambda_n)}(R_n(A\cdot,Z_n)+\lambda_n G)
\end{equation}
to the empirical regularized problems
\begin{equation}
\label{eq:empminriskbis}
\minimize{u\in\FF}{R_n(A u, Z_n)}+\lambda_n  G(u),
\end{equation}
in the sense that $R(A u_{n,\lambda_n}(Z_n)) \to\inf R(\C)$  in
probability (\emph{weak consistency}) or almost surely 
(\emph{strong consistency}), under suitable conditions on 
$(\lambda_n)_{n\in\NN}$.
\end{problem}

\begin{definition}
\label{def:goodconstraints}
Let $p\in [1,+\infty]$. Then $\C$ in Assumption~\ref{asssec4} is 
\emph{$p$-admissible} if $\C\subset L^p(\XC,\PP_{\XC};\YY)$, or if
$\C\cap L^p(\XC, \PP_{\XC};\YY) \neq \emp$ and there exists a 
family $(\mathsf{C}(x))_{x\in\XC}$ of closed convex subsets of 
$\YY$ such that 
$\C=\menge{f\in\MC(\XC, \YY)}{(\forall x\in\XC)\;f(x)\in
\mathsf{C}(x)}$.
\end{definition}

We are now ready to state the two main results of the paper
(see Section~\ref{sec:proofs} for proofs).

\begin{theorem}
\label{thm:main}
Suppose that Assumption~\ref{asssec4} holds, set 
$\varsigma=\norm{\fmap}_{\infty}$, and write
$\varepsilon=\varepsilon_1\varepsilon_2$, where
$\varepsilon_1$ and 
$\varepsilon_2$ are functions from $\RPP$ to $[0,1]$.
Let $p \in [1,+\infty]$ and 
suppose that $\ell \in \Loss_p(\XC\times\YC\times\YY,\PP)$, 
that $\C$ is $p$-admissible, that $\ran A$ is 
$p\verb0-0$universal relative to $\C$, and that
$A(\dom G)\subset\C\cap\ran A\subset\overline{A(\dom G)}$,
where the closure is in $L^p(\XC,\PP_\XC;\YY)$. 
Then the following hold:
\vspace{-2ex}
\begin{enumerate}\setlength{\itemsep}{1pt} 
\item
\label{thm:maini}
Assume that $\ell(\cdot,\cdot,0)$ is bounded and let 
$(\forall n\in\NN)$ $\rho_n \in \big[\psi_0^\natural\big(
(\norm{\ell(\cdot,\cdot,0)}_\infty +1)/\lambda_n\big), +\infty\big[\,$.
Suppose that 
\begin{equation}
\label{eq:conscond}
\lip{\ell}{\varsigma\rho_n}\varepsilon_1(\lambda_n)\to 0
\quad\text{and}\quad\varepsilon_2(\lambda_n)=
O\bigg(\frac{\lip{\ell}{\varsigma\rho_n}}{n^{1/q}}\bigg),
\end{equation}
and that
\begin{equation}
\label{eq:conscond2}
(\forall{\tau}\in\RPP)\quad \lip{\ell}{\varsigma\rho_n}
(\widehat{\psi}_{\rho_n})^\natural \bigg( 
\frac{{\tau}\lip{\ell}{\varsigma\rho_n}}{\lambda_n n^{1/q}}\bigg)
\to 0.
\end{equation}
Then $R(A u_{n,\lambda_n}(Z_n))\overset{\PO^*}{\to}\inf R(\C)$.
Moreover, if 
\begin{equation}
\label{eq:conscond2s}
(\forall{\tau}\in\RPP)\quad\lip{\ell}{\varsigma\rho_n}
(\widehat{\psi}_{\rho_n})^\natural\bigg( 
\frac{{\tau}\lip{\ell}{\varsigma \rho_n} \log n}{\lambda_n n^{1/q}} 
\bigg)\to 0,
\end{equation}
then $R(A u_{n,\lambda_n}(Z_n))\to\inf R(\C)$
$\PO^*\!\verb0-0\text{a.s.}$
\item
\label{thm:maini+}
Assume that $p\in\left]1,+\infty\right[$ and that the function $b$ 
associated with $\ell$ in 
Definition~\ref{def:lossandrisk}\ref{def:Gpi} is bounded, and let 
$(\forall n\in\NN)$ $\rho_n \in \big[\psi_0^\natural\big(
(\norm{\ell(\cdot,\cdot,0)}_\infty +1)/\lambda_n\big), +\infty\big[\,$.
Suppose that 
\begin{equation}
\label{eZhd7rT9-06r}
\rho^{\,p-1}_{n}\varepsilon_1(\lambda_n)\to 0, \;
\varepsilon_2(\lambda_n)=O\Big(\frac{
\rho^{\,p-1}_{n}}{n^{1/q}}\Big),
\;\text{and}\;
(\forall{\tau}\in\RPP)\; \rho^{\,p-1}_{n}
(\widehat{\psi}_{\rho_n})^\natural\bigg(
\frac{{\tau}\rho^{\,p-1}_{n}}{\lambda_n n^{1/q}} \bigg) \to 0.
\end{equation}
Then $R(Au_{n,\lambda_n}(Z_n))\overset{\PO^*}{\to}\inf R(\C)$.
Moreover, if 
\begin{equation}
\label{eZhd7rT9-06t}
(\forall{\tau}\in\RPP)\quad\rho^{\,p-1}_{n}
(\widehat{\psi}_{\rho_n})^\natural\bigg(
\frac{{\tau}\rho^{\,p-1}_{n} \log n}{\lambda_n n^{1/q}} \bigg) \to 0,
\end{equation}
then $R(Au_{n,\lambda_n}(Z_n))\to\inf R(\C)$
$\PO^*\!\verb0-0\text{a.s.}$
\item 
\label{thm:mainiii}
Assume that $p=1$ and let
$(\forall n\in\NN)$ $\rho_n \in \big[\psi_0^\natural((R(0)+1)/\lambda_n),+\infty\big[\,$.
Suppose that 
\begin{equation}
\label{eq:conscond_l1}
\varepsilon_1(\lambda_n) \to 0, \ \ 
\varepsilon_2(\lambda_n)=O\Big(\frac{1}{  n^{1/q}} \Big),
\quad\text{and}\quad
(\forall{\tau} \in\RPP)\ \ 
(\widehat{\psi}_{\rho_n})^\natural\bigg(
\frac{{\tau}}{\lambda_n n^{1/q}} \bigg) \to 0.
\end{equation}
Then $R(A u_{n,\lambda_n}(Z_n)) \overset{\PO^*}{\to}\inf R(\C)$.
Moreover, if 
\begin{equation}
\label{eq:conscond_l13}
(\forall{\tau}\in\RPP)\quad
(\widehat{\psi}_{\rho_n})^\natural\bigg(
\frac{{\tau} \log n}{\lambda_n n^{1/q}} \bigg) \to 0,
\end{equation}
then $R(A u_{n,\lambda_n}(Z_n)) \to\inf R(\C)$
$\PO^*\!\verb0-0\text{a.s.}$
\item
\label{thm:mainii}
Suppose that $S=\Argmin_{\dom  G}(R\circ A)\neq\emp$. Then 
there exists a unique $u^\dag\in S$ which minimizes $ G$ on 
$S$; moreover, $Au^\dag\in\C$ and $R(Au^\dag)=\inf R(\C)$. 
Furthermore, suppose that the following conditions are satisfied:
\setcounter{enumi}{2}
\begin{equation}
\label{str_conscond}
\varepsilon_1(\lambda_n) \to 0, \quad
\frac{\varepsilon_2(\lambda_n)}{\lambda_n}\to 0,
\quad\text{and}\quad \frac{1}{\lambda_n n^{1/q}} \to 0\,.
\end{equation}
Then $\norm{u_{n,\lambda_n}(\ZZZ_n)-u^\dag}\overset{\PO^*}{\to}0$ 
and $R(Au_{n,\lambda_n}(Z_n))\overset{\PO^*}{\to} R(A u^\dag)$. 
Finally, suppose in addition that 
\begin{equation}
\label{e:34789v624389}
(\log n)/(n^{1/q}\lambda_n) \to 0\,.
\end{equation}
Then $\norm{u_{n,\lambda_n}(\ZZZ_n)-u^\dag} \to 0$ 
$\PO^* \verb0-0\text{a.s.}$ and $R(A u_{n, \lambda_n}(Z_n)) 
\to R(A u^\dag)$ $\PO^* \verb0-0\text{a.s.}$
\end{enumerate}
\end{theorem}

\begin{remark}\ 
\vspace{-2ex}
\begin{enumerate}\setlength{\itemsep}{1pt} 
\item
In the setting of Example~\ref{ex:distbased}, 
$\ell(\cdot,\cdot,0)$ is bounded if $\YC$ is a bounded subset of $\YY$.
\item
$A(\dom G)\subset\C\cap \ran A \subset \overline{A (\dom  G)}$
is a compatibility condition between $G$ and $\C$. It is
satisfied in particular when $\overline{\dom G}=A^{-1}(\C)$, since 
$A(A^{-1}(\C))=\C\cap \ran{A}$.
On the other hand, $\ran A$ is trivially
$\infty\verb0-0$universal relative to $\C$ when $\C \subset \ran
A$, or $\ran A\subset\C$ and $\ran A$ is 
$\infty\verb0-0$universal.
\item
For every $\rho \in \RPP$, $\dom (\widehat{\psi}_\rho)^{\natural}$ is 
an interval containing $0$ with nonempty interior.
Indeed,
it follows from Assumption~\ref{asssec4}\ref{asssec4III} that
$\Argmin_{\FF} G\neq\emp$, hence $0\in\dom\partial G$. 
Therefore, Proposition~\ref{p:8}\ref{p:8viii} ensures that, 
for every $\rho\in\RP$,
${\psi}_\rho\in\mathcal{A}_1$. Thus, 
Proposition~\ref{prop:psinatural}\ref{psinat_ina0}
yields
$(\widehat{\psi}_\rho)^{\natural} \in \mathcal{A}_0$ 
and the statement follows from 
Proposition~\ref{prop:psinatural}\ref{psinat_dom}.
\item\label{rmk:forcingiv}
Let $(s_n)_{n\in\NN}$ and $(\rho_n)_{n \in \NN}$ be sequences 
in $\RPP$ and suppose that $\rho=\inf_{n\in\NN}\rho_n >0$. Then 
$(\widehat{\psi}_{\rho_n})^\natural(s_n) \to 0
\Rightarrow s_n \to 0$. Indeed, for every $n\in\NN$,
$\rho\leq \rho_n \Rightarrow
\psi_{\rho_n}\leq\psi_{\rho}\Rightarrow
\widehat{\psi}_{\rho_n}\leq
\widehat{\psi}_{\rho} \Rightarrow
(\widehat{\psi}_{\rho})^\natural\leq
(\widehat{\psi}_{\rho_n})^\natural$.
Therefore $(\widehat{\psi}_{\rho_n})^\natural(s_n) \to 0
\Rightarrow (\widehat{\psi}_{\rho})^\natural(s_n) \to 0
\Rightarrow s_n \to 0$
by Proposition~\ref{prop:psinatural}\ref{psi_forcing}.
\end{enumerate}
\end{remark}

Next we consider an important special case, in which
the consistency conditions can be made explicit.

\begin{corollary}
\label{c:consJnorm} 
Suppose that Assumption~\ref{asssec4} holds, set 
$\varsigma=\norm{\fmap}_{\infty}$, and write
$\varepsilon=\varepsilon_1\varepsilon_2$, where
$\varepsilon_1$ and 
$\varepsilon_2$ are functions from $\RPP$ to $[0,1]$.
Let $p\in [1,+\infty]$ and 
suppose that $\ell\in\Loss_p(\XC\times\YC\times\YY,\PP)$, 
that $\C$ is $p$-admissible, that $\ran A$ is 
$p\verb0-0$universal relative to $\C$, that
$A(\dom G)\subset\C\cap\ran A\subset\overline{A(\dom G)}$,
where the closure is in $L^p(\XC,\PP_\XC;\YY)$. In addition,
assume that 
\begin{equation}
\label{eZhd7rT8-29a}
\begin{cases}
\FF\;\text{is uniformly convex with modulus of
convexity of power type}\;q\\
G=\eta\norm{\cdot}^r+H,
\quad\text{where}\;\eta\in\RPP,\; r\in\left]1,+\infty\right[, 
\;\text{and}\; H\in\Gamma^+_0(\FF).
\end{cases}
\end{equation}
Let $\beta$ be the constant
defined in Proposition~\ref{prop:modnormr}, and set
${m=\max\{r,q\}}$. Then the following holds:
\vspace{-2ex}
\begin{enumerate}\setlength{\itemsep}{1pt} 
\item
\label{c:consJnormi}
Assume that $\ell(\cdot,\cdot,0)$ is bounded and
set $(\forall n\in\NN)$ 
$\rho_n=\big((\norm{\ell(\cdot,\cdot, 0)}_\infty+1)/(\eta
\beta\lambda_n)\big)^{1/r}$.
Suppose that 
\begin{equation}
\label{eZhd7rT9-06x}
\lip{\ell}{\varsigma \rho_n}
\varepsilon_1(\lambda_n) \to 0,\quad
\quad \varepsilon_2(\lambda_n)=
O\bigg(\frac{\lip{\ell}{\varsigma\rho_n}}{n^{1/q}}\bigg),
\quad\text{and}\quad
\dfrac{\lip{\ell}{\varsigma \rho_n}^{m}}
{\lambda_n^{{m/r}} n^{{1}/q}}\to 0.
\end{equation}
Then $R(Au_{n,\lambda_n}(Z_n))\overset{\PO^*}{\to}\inf R(\C)$. 
Moreover, if 
\begin{equation}
\label{eZhd7rT9-06y}
\dfrac{\lip{\ell}{\varsigma \rho_n}^{m} \log n}{\lambda_n^{{m/r}}
n^{{1}/q}}\to 0,
\end{equation}
then $R(A u_{n,\lambda_n}(Z_n)) \to\inf R(\C)$
$\PO^*\!\verb0-0\text{a.s.}$
\item
\label{c:consJnormiii}
Assume that $p\in\left]1,+\infty\right[$, that the function $b$ 
associated with $\ell$ in 
Definition~\ref{def:lossandrisk}\ref{def:Gpi} is bounded, and that
\begin{equation}
\label{eZhd7rT8-29b}
\frac{\varepsilon_1(\lambda_n)}{\lambda_n^{(p-1)/r}} \to 0,\quad  
\varepsilon_2(\lambda_n)=O\bigg(\frac{1}{n^{1/q}
\lambda_n^{(p-1)/r}}\bigg),
\quad\text{and}\quad\frac{1}{\lambda^{p {m/r}}_n n^{{1}/q}}\to 0.
\end{equation}
Then $R(Au_{n,\lambda_n}(Z_n))\overset{\PO^*}{\to}\inf R(\C)$.
Moreover, if 
$(\log n)/(\lambda^{p{m/r}}_n n^{{1}/q})\to 0$,
then $R(Au_{n,\lambda_n}(Z_n))\to\inf R(\C)$
$\PO^*\!\verb0-0\text{a.s.}$
\item
\label{c:consJnormii}
Assume that $p=1$ and that
\begin{equation}
\label{e:nonso1}
\varepsilon_1(\lambda_n) \to 0, \ \ 
\varepsilon_2(\lambda_n)=O\Big(\frac{1}{  n^{1/q}} \Big),
\quad\text{and}\quad
\frac{1}{\lambda_n^{{m/r}} n^{{1}/q}} \to 0.
\end{equation}
Then $R(A u_{n,\lambda_n}(Z_n)) \overset{\PO^*}{\to}\inf R(\C)$.
Moreover, if $(\log n)/(\lambda_n^{{m/r}} n^{{1}/q})  \to 0$,
then $R(A u_{n,\lambda_n}(Z_n)) \to\inf R(\C)$
$\PO^*\!\verb0-0\text{a.s.}$
\end{enumerate}
\end{corollary}

\begin{remark}
Corollary~\ref{c:consJnorm} shows that consistency is achieved
when the sequence of regularization parameters
$(\lambda_n)_{n\in\NN}$ converges to zero not too fast.
The upper bound depends on the 
power type of the modulus of convexity of the feature space, the
exponent of the norm in the regularizer, and the Lipschitz behavior
of the loss. Note that a faster decay of $(\lambda_n)_{n\in\NN}$
is allowed when $q=2$.
\end{remark}

\begin{remark} The class of regularizers considered in Corollary~\ref{c:consJnorm} 
includes the elastic-net penalty both in the setting of generalized linear models \cite{Demo09}
and multiple kernel learning \cite{Suzu2013}.
The proofs of Theorem~\ref{thm:main} 
and Corollary~\ref{c:consJnorm} are based on a stability analysis,
which combines the sensitivity Theorem~\ref{thm:sensitivity2} with a 
Banach space-valued Hoeffding's inequality (Theorem~\ref{thm:Hoeffding}).  
The strength of such a method is that it can be applied in very general
situations, since it does not require any structure on the input space and no
hypotheses on the probability measure. We highlight that the analysis 
can be applied even to unbounded output spaces if Hoeffding's inequality 
is replaced by Markov's inequality.
\end{remark}

\begin{remark}
\label{rmkIuh7td14}
In the setting of general regularizers and/or Banach feature spaces,
the literature on consistency of regularized empirical risk 
minimizers is scarce.
\begin{enumerate}
\item  
In \cite[Theorem 7.20]{SteChi2008} only continuous, real-valued regularizers 
are considered and consistency is established under the provision that 
local Rademacher complexities can be suitably bounded and 
an appropriate variance bound holds \cite{Bart2005}. 
However, it is not clear whether this result is useful for 
other regularizers apart from the squared norm.
\item 
A well-studied method to prove consistency of regularized empirical 
risk minimization
is based on covering numbers \cite{CuZh2007,Wang11,Wu2006}.
However, it should be stressed that the application of such method 
in the vector-valued setting would require the following additional
assumptions:
(a) the input space is a locally compact topological space and the
feature
map is continuous with respect to the uniform operator topology 
and takes compact operators as values 
(this implies the finiteness of the related covering numbers); 
(b) the covering numbers decay polynomially 
(this usually requires smooth kernels); and (c)
an appropriate variance bound for the loss is available.
\item 
In \cite{Song2011}, the consistency of an $\ell^1$-regularized 
empirical risk minimization scheme is studied in a particular 
type of Banach spaces of functions, in which a linear representer 
theorem is shown to hold. 
Note that, in general reproducing kernel Banach 
spaces, the representation is not linear; see 
Corollary~\ref{cor:repthm} and \cite{Zhang2012,Zhang2013}.
In \cite{Stei09}, consistency and learning rates are 
provided for classification problems and $G=\|\cdot\|$,
under appropriate growth assumptions on 
the average empirical entropy numbers. 
\item
In \cite{MauPon12} a class of regularizers inducing structured 
sparsity is considered and associated statistical bounds are provided.
\end{enumerate}
\end{remark}

We complete this section by providing an illustration of the above 
consistency theorems to learning with dictionaries in the context
of Example~\ref{ex:dictionary}. The setting will be a specialization
of Assumption~\ref{asssec4} to specific types of feature maps and 
regularizers. Our analysis extends in several directions that of 
\cite{Demo09}.

\begin{example}[Generalized linear model]
\label{exa:2014}
Suppose that 
Assumption~\ref{asssec4}\ref{asssec4i}-\ref{ass:constr} hold.
Let $\mathbb{K}$ be a nonempty at most countable set,
let $r\in\left]1,\pinf\right[$, and let $\FF=l^{r}(\KK)$.
Let $\varsigma\in\RPP$, let 
$(\dictfunc_k)_{k\in\KK}$ be a dictionary of functions in
$\mathcal{M}(\XC,\YY)$ such that, for 
$\PP_\XC\verb0-0\text{a.a.}~x\in\XC$,
$\sum_{k\in\KK}\abs{\dictfunc_k(x)}^{r^*}\leq
\varsigma^{r^*}$, and set 
\begin{equation}
\label{eq:fmap_dictionary0}
A\colon\FF\to
\YY^{\XC}\colon u=(\mu_k)_{k\in\KK}
\mapsto\sum_{k\in\KK}\mu_k \dictfunc_k\quad\text{(pointwise)}.
\end{equation}
Let $\fmap\colon\XC\to l^{r^*}(\KK;\YY)\colon x\mapsto 
(\dictfunc_k(x))_{k\in\KK}$
be the associated feature map.
For every $k\in\KK$, let $\eta_k\in\RP$ and let 
$h_k\in\Gamma^+_0(\RR)$ be
such that $h_k(0)=0$. Define
\begin{equation}
\label{e:genova2012-10-05h}
G\colon \FF \to [0,+\infty]\colon u=(\mu_k)_{k\in\KK}
\mapsto\sum_{k\in\KK}g_k(\mu_k),
\quad\text{where}\quad(\forall k\in\KK)\quad
g_k=h_k+\eta_k |\cdot|^r.
\end{equation}
Let $(\lambda_n)_{n\in\NN}$ be a sequence in $\RPP$ such that
$\lambda_n\to 0$ and let
$(X_i,Y_i)_{i\in\NN}$ be a sequence of independent copies of
$(X,Y)$. For every $n\in\NN\smallsetminus\!\{0\}$, let
$Z_n=(X_i,Y_i)_{1\leq i\leq n}$, and let
$u_{n,\lambda_n}(Z_n)$ be defined according to 
\eqref{eq:learnalgo} as an approximate minimizer of the 
regularized empirical risk
\begin{equation}
\frac{1}{n}\sum_{i=1}^n
\ell\big(X_i,Y_i,(Au)(X_i)\big)+\lambda_n G(u).
\end{equation}
The above model covers several classical regularization 
schemes, such as the Tikhonov (ridge regression) model 
\cite{Gyorfi02}, the $\ell_1$ or lasso model \cite{Tib96}, the 
elastic net model \cite{Demo09,ZouHastie2005}, the bridge 
regression model \cite{Fu1998,Kol2009}, as well as generalized 
Gaussian models \cite{Anto02}. Furthermore the following hold:
\vspace{-2ex}
\begin{enumerate}\setlength{\itemsep}{1pt} 
\item
$\FF$ is uniformly convex with modulus of convexity of power type 
$\text{max}\{2,r\}$ (see Section~\ref{sec:notation}). 
Moreover, $\ran A\subset\MC(\XC,\YY)$,
\begin{equation}
(\forall x\in\XC)(\forall u\in \FF)\quad\abs{\fmap(x)^* u}
=\abs{(A u)(x)}\leq\norm{u}_r 
\norm{(\dictfunc_k(x))_{k\in\KK}}_{r^*}
\leq\varsigma\norm{u}_r,
\end{equation}
and therefore $\norm{\fmap}_\infty\leq\varsigma$. Now suppose that 
$\inf_{k\in\KK}\eta_k>0$. Then, in view of Proposition
\ref{prop:modnormr}, $ G$ is totally convex on bounded sets.
Altogether, Assumption~\ref{asssec4} holds with 
$q=\text{max}\{2,r\}$.
\item
Let $p\in[1,+\infty]$ and suppose that one of the following holds:
\vspace{-1ex}
\begin{enumerate}\setlength{\itemsep}{1pt} 
\item
\label{rmk:admconstri} 
$\C=A\big(l^r(\KK)\cap\cart_{\!k\in\KK}\dom h_k\big)$. 
\item
\label{rmk:admconstrii}
$\C=\mathcal{M}(\XC,\YY)$ and $\spann\{\dictfunc_k\}_{k\in\KK}$ 
is $p\verb0-0$universal (Definition~\ref{def:universal}). 
\end{enumerate}
Then $\C$ is $p\verb0-0$admissible
(Definition~\ref{def:goodconstraints}), 
$A(\dom G)\subset\C\cap\ran A\subset\overline{A(\dom G)}$ 
(where the closure is in $L^p(\XC,\PP_\XC;\YY)$), 
and $\ran A$ is $p$-universal relative to $\C$.
Indeed, as for \ref{rmk:admconstri}, $\C\subset\ran A\subset
L^p(\XC,\PP_\XC;\YY)$, hence $\C$ is $p$-admissible and 
$\ran A$ is $p$-universal relative to $\C$. 
Moreover, $A(\dom G) \subset\C\subset\overline{A(\dom G)}$ since, 
for every $u\in l^r(\KK)\cap\cart_{k\in\KK}\dom h_k$ and every
$\epsilon\in\RPP$, there exists
$\bar{u}\in\RR^\KK$ with finite support, such that
$\norm{u-\bar{u}}_r\leq\epsilon$ and $\norm{Au-A\bar{u}}_p\leq
\varsigma\norm{u-\bar{u}}_r\leq\varsigma\epsilon$. On the other
hand, it follows from 
Theorem~\ref{prop:C-universality}\ref{prop:C-universalityii} that, 
if $\C=\mathcal{M}(\XC,\YY)$, \ref{rmk:admconstrii} is 
satisfied when $\XC$ is a locally compact topological space and 
$\spann\{\dictfunc_k\}_{k\in\KK}$ is dense in 
$\mathscr{C}_0(\XC,\YY)$ endowed with the uniform topology.
\item
Let $\C$ be as in item~\ref{rmk:admconstri} or 
\ref{rmk:admconstrii}, let $\eta\in\RPP$, and suppose that 
$(\forall k\in\KK)$ $\eta_k\geq\eta$. Then consistency can be
obtained in the setting of Corollary~\ref{c:consJnorm}, 
where $q=\max\{2,r\}=m$.
In particular, in items \ref{c:consJnormiii} and
\ref{c:consJnormii} of Corollary~\ref{c:consJnorm}, we have
$\lambda_n^{p m/r} n^{1/q}=\lambda_n^{p} n^{1/r}$, if $r \geq 2$;
and $\lambda_n^{p m/r} n^{1/q}=\lambda_n^{2p/r} n^{1/2}$, if $r
\leq 2$. 
Moreover, by Theorem~\ref{thm:main}\ref{thm:mainii}, weak consistency
holds if $1/(\lambda_n n^{1/\max\{2,r\}})\to 0$, and strong
consistency holds if 
$(\log n)/(\lambda_n n^{1/\max\{2,r\}}) \to 0$.
\item 
Suppose that $r\in\left]1,2\right]$ and that the loss function is 
differentiable with respect to the third variable. Then, by
exploiting the separability of $G$, for a given sample
size $n$, an estimate $u_{n,\lambda_n}(z_n)$ can be constructed
in $l^2(\KK)$ using proximal splitting algorithms such as those
described in \cite{Siop07,Villa13}.
\end{enumerate}
\end{example}

\begin{remark}
Let us compare the results of Example~\ref{exa:2014} to the 
existing literature on generalized linear models.
\vspace{-2ex}
\begin{enumerate}\setlength{\itemsep}{1pt} 
\item
In the special case when $\KK$ is finite, $r>1$, and 
$G=\|\cdot\|_r^r$, \cite{Kol2009} provides 
an excess risk bound which  depends 
on the dimension of the dictionary (the cardinality of $\KK$)
and the level of sparsity of the regularized risk minimizer;
see \cite{Bulh11} for a recent account of the role of sparsity
in regression.  
\item
In the special case when $r=2$ and, for every $k\in\KK$,
$h_k=w_k\abs{\cdot}$ with $w_k\in\RPP$ in 
\eqref{e:genova2012-10-05h}, we recover the elastic net
framework of \cite{Demo09}. This special case yields a strongly 
convex problem in a Hilbert space.  
In our general setting, the 
exponent $r$ may take any value in $\left]1,+\infty\right[$.
Note also that our
framework allows for the enforcement of hard constraints
on the coefficients since the functions $(h_k)_{k\in\KK}$ are
not required to be real-valued. We highlight that, when 
specialized to the elastic net regularizer, 
Theorem~\ref{thm:main}\ref{thm:mainii} 
guarantees consistency under the same conditions
as in \cite[Theorem 2]{Demo09}.
\end{enumerate}
\end{remark}

\subsection{Proofs of the main results}
\label{sec:proofs}

We start with a few properties of the functions underlying our 
construct. To this end, throughout this subsection, the following
notation will used.

\begin{notation}
\label{dZhd7rT9-02}
In the setting of Assumption~\ref{asssec4}, 
\begin{equation}
F=R\circ A\quad\text{and}\quad
(\forall n\in\NN\smallsetminus\{0\})\quad
F_n\colon\FF\times(\XC\times\YC)^n\to\RP\colon
(u,z)\mapsto R_n(A u,z).
\end{equation}
In addition, $\varsigma=\norm{\fmap}_{\infty}$, and, for every
$n\in\NN\smallsetminus\{0\}$ and $\lambda\in\RPP$, 
\begin{align}
\label{eq:parameters}
\nonumber\alpha_{n,\lambda}\colon\RPP\times\RPP&\to\RP\\
(\tau,\rho)
&\mapsto\frac{\varsigma \lip{\ell}{\varsigma\rho}}{\lambda } 
\bigg(\frac{4 T_{q^*}}{ n^{1/q}}+2\sqrt{\frac{2\tau}{n}}
+\frac{4\tau}{3n}\bigg).
\end{align}
Now let $\tau\in [1,+\infty[$ and $n\in\NN\smallsetminus\{0\}$. 
Then, since $2\sqrt{2 \tau}\leq 1+2\tau \leq 3 \tau$
and $n^{1/q} \leq n^{1/2} \leq n$, we have
\begin{equation}
\label{e:afY34nd93nl10c}
(\forall\, \rho \in \RPP)\quad \alpha_{n,\lambda}(\tau,\rho)
\leq \frac{\tau\varsigma (4 T_{q^*}+5) 
\lip{\ell}{\varsigma\rho}}{\lambda n^{1/q}}
\end{equation}
\end{notation}

\begin{proposition}
\label{p:5}
Suppose that Assumption~\ref{asssec4} is satisfied.
Then the following hold:
\vspace{-2ex}
\begin{enumerate}\setlength{\itemsep}{1pt} 
\item
\label{p:5i}
$F\colon\FF\to\RP$ is convex and continuous.
\item
\label{p:5ii}
Let $n\in\NN\smallsetminus\{0\}$ and 
$z\in(\XC\times\YC)^n$. Then $F_n(\cdot,z)\colon\FF\to\RP$ is 
convex and continuous.
\item
\label{p:5iii}
$G$ is coercive and strictly convex.
\item
\label{p:5iv}
For every $\lambda\in\RPP$, $F+\lambda G$ admits a unique minimizer.
\end{enumerate}
\end{proposition}
\begin{proof}
\ref{p:5i}:
Remark~\ref{r:contrisk}\ref{r:contriskiv} ensures that
$R\colon L^\infty(\XC,\PP_{\XC};\YY)\to\RP$ is 
convex and continuous. In turn, 
Proposition~\ref{pZhd7rT9-03}\ref{pZhd7rT9-03ii} implies that
$A\colon\FF\to L^\infty(\XC,\PP_{\XC};\YY)$ is continuous.

\ref{p:5ii}: The argument is the same as above, except that
$\PP$ is replaced by the empirical measure 
$(1/n)\sum_{i=1}^n\delta_{(x_i,y_i)}$, where
$z=(x_i,y_i)_{1\leq i\leq n}$.

\ref{p:5iii}: 
It follows from Assumption~\ref{asssec4}\ref{asssec4III} and
Proposition~\ref{p:8}\ref{p:8ix} that $G$ is coercive; 
its strict convexity follows from  the
definition in \eqref{def:totconvmodi}.

\ref{p:5iv}: By \ref{p:5i} and \ref{p:5iii}, $F+\lambda G$ is 
a strictly convex coercive function in $\Gamma_{0}^+(\FF)$.
It therefore admits a unique minimizer 
\cite[Theorem~2.5.1(ii) and Proposition~2.5.6]{Zali02}.
\end{proof}

The strategy of the proof of Theorem~\ref{thm:main} is to
split the error in three parts, i.e.,
\begin{multline}
\label{spliterror}
R(A u_{n,\lambda}(Z_n))-\inf R (\C)\\
=(F(u_{n,\lambda}(Z_n))-F(u_\lambda))+(F(u_\lambda)-\inf F(\dom G)) 
+(\inf F(\dom G)-\inf R (\C)),\\
\quad\text{where}\quad u_\lambda=\argmin_{\FF}(F+\lambda G).
\end{multline}
Note that Proposition~\ref{p:5}\ref{p:5iv} ensures that
$u_\lambda$ is uniquely defined. 
The first term on the right-hand side of \eqref{spliterror} is
known as the sample error and the second term as the approximation 
error. Proposition~\ref{p:4}\ref{p:4i} 
ensures that the approximation error goes to zero as
$\lambda\to 0$.
Below, we start by showing that $\inf R(\C)-\inf F(\dom G)=0$, 
if $\ran A$ is universal with respect
to $\C$ and some compatibility conditions
between $G$  and $\C$ hold. Next, we study the sample error. 
Note that $F(u_{n,\lambda}(Z_n))-F(u_\lambda)$ 
may not be measurable, hence the convergence
results are given with respect to the outer probability $\PO^*$.

\begin{proposition}
\label{prop:infriskonLp} 
Let $\XC$ and $\YC$ be nonempty sets, let $(\XC\times\YC,
\mathfrak{A}, \PP)$ be a probability space, let $\PP_\XC$ be the
marginal of $\PP$ on $\XC$, and
let $\YY$ be a separable reflexive real Banach space. Let
$\ell\in\Loss(\XC\times\YC,\YY)$, and let
$R\colon\MC(\XC,\YY)\to\RPX$ be the risk associated with $\ell$ and
$\PP$. Let $\C \subset \MC(\XC,\YC)$ be  nonempty and convex.
Let $p\in [1,+\infty]$ and assume that $\C$ is 
$p\verb0-0$admissible and that there exists 
$g\in\C\cap L^p(\XC, \PP_{\XC}; \YY)$ such that $R(g)<+\infty$. Then
$\inf R (\C)=\inf R (\C\cap L^p(\XC, \PP_{\XC}; \YY))$.
\end{proposition}
\begin{proof}
Suppose that
$\C=\menge{f\in\MC(\XC,\YY)}{(\forall x\in\XC)\;f(x)
\in\mathsf{C}(x)}$. 
Let $f\in\C$ be such that $R(f)<+\infty$. For every $n\in\NN$,
set $A_n=\menge{x\in\XC}{\abs{f(x)}\leq n}$, let 
$A^c_n $ be its complement, and define $f_n\colon\XC\to\YY$, 
$f_n=\boldsymbol{1}_{A_n} f+\boldsymbol{1}_{A^c_n} g$. 
For every $n\in\NN$ and $x\in\XC$, $f_n(x)\in\mathsf{C}(x)$
and $\abs{f_n(x)}\leq \max\{n, \abs{g(x)}\}$, 
hence $f_n\in\C\cap L^p(\XC, \PP_{\XC}; \YY)$. Moreover,
\begin{equation}
\label{eq:RLp}
(\forall n\in\NN)\quad
\abs{R(f_n)-R(f)}\leq\int_{A^c_n \times \YC} \abs{\ell(x,y,
g(x))-\ell(x,y, f(x))} P(\ud(x,y)).
\end{equation}
Set $h\colon(x,y)\mapsto\abs{\ell(x,y,g(x))-\ell(x,y,f(x))}$.
Since $R(f)<\pinf$ and $R(g)<\pinf$, we have 
$h\in L^1(\XC\times\YC,\PP)$.
Since $\boldsymbol{1}_{A_n^c \times \YC} h \to 0$ pointwise and  
$\boldsymbol{1}_{A_n^c \times \YC} h\leq h$, it follows from
the dominated convergence theorem that the right-hand side of
\eqref{eq:RLp} tends to zero, and hence $R(f_n)\to R(f)$. 
This implies that
$\inf R (\C\cap L^p(\XC, \PP_{\XC}; \YY))\leq R(f)$. 
\end{proof}

\begin{proposition}
\label{prop:constraints} 
Let $\XC$ and $\YC$ be nonempty sets, let $(\XC\times\YC,
\mathfrak{A}, \PP)$ be a probability space, let $\PP_\XC$ be the
marginal of $\PP$ on $\XC$, and
let $\YY$ be a separable reflexive real Banach space.
Let $\C \subset \MC(\XC,\YC)$ be  nonempty and convex and
let $p\in[1,\pinf]$. Suppose
that $\ell\in\Loss_p(\XC\times\YC,\YY,\PP)$, that
$\fmap\in L^p[\XC,\PP_\XC;\LL(\YY^*,\FF^*)]$, and that 
$A(\dom  G) \subset \C\cap \ran A \subset \overline{A
(\dom G)}$, where the closure is in $L^p(\XC,\PP_\XC;\YY)$. 
Let $R\colon\MC(\XC,\YY)\to\RPX$ be the risk associated 
with $\ell$ and $\PP$. Then the following hold:
\vspace{-2ex}
\begin{enumerate}\setlength{\itemsep}{1pt} 
\item
\label{ite:constr1}
$\inf F(\dom G)=\inf R(\C\cap\ran{A})$.
\item
\label{ite:constr2}
Suppose that $\C$ is $p\verb0-0$admissible and $\ran A$ is
$p$\verb0-0universal relative to $\C$. 
Then $\inf F(\dom G)=\inf R (\C)$.
\end{enumerate}
\end{proposition}
\begin{proof}
\ref{ite:constr1}:
By Remark~\ref{r:contrisk}\ref{r:contriskii}, $R$ is
continuous on $L^p(\XC,\PP_\XC;\YY)$ and hence 
$\inf R(A(\dom G))=\inf R( \overline{A (\dom  G)})$.  Therefore,  
since $A(\dom G)\subset\C\cap\ran A\subset\overline{A(\dom G)}$,
the assertion follows.

\ref{ite:constr2}:
Suppose first that $p<+\infty$.  Since $R$ is continuous on
$L^p(\XC, \PP_\XC; \YY)$ and $\C\cap \ran A$ is dense in $\C
\cap L^p(\XC,\PP_\XC;\YY)$, $\inf R(\C\cap\ran{A})=\inf R(\C
\cap L^p(\XC,\PP_\XC;\YY))$. Thus, since $\C$ is $p$-admissible,
Proposition \ref{prop:infriskonLp} gives $\inf R(\C\cap L^p(\XC,
\PP_\XC; \YY))=\inf R(\C)$ and hence 
$\inf R (\C\cap \ran{A})=\inf R(\C)$. 
The statement follows from \ref{ite:constr1}. Now suppose 
that $p=+\infty$. Let $f\in\C\cap L^\infty(\XC,\PP_\XC;\YY)$.
By Definition~\ref{def:universal}\ref{def:infuniversal}, 
there exists $(f_n)_{n\in\NN}\in(\C\cap \ran{A})^\NN$ and 
$\rho\in\RPP$
such that $\sup_{n\in\NN} \norm{f_n}_{\infty}\leq\rho$ and 
$f_n\to f$ $\PP_\XC\verb0-0\text{a.s.}$ It follows from 
\eqref{2cwrey6T24b} that 
$(\exi g_\rho\in L^1(\XC\times\YC,\PP;\RR))
(\forall (x,y)\in\XC\times\YC)$
$\abs{\ell(x,y,f_n(x))-\ell(x,y,f(x))}\leq 2g_\rho(x,y)$.
By the dominated convergence theorem, $R(f_n)\to R(f)$.
Thus, $\inf R(\C\cap\ran{A})=
\inf R(\C\cap L^\infty(\XC,\PP_\XC;\YY))$ and we conclude as above.
\end{proof}

\begin{proposition}
\label{p:consistency}
Suppose that Assumption~\ref{asssec4} holds and that 
Notation~\ref{dZhd7rT9-02} is in use.
Write $\varepsilon=\varepsilon_1\varepsilon_2$, where
$\varepsilon_1$ and 
$\varepsilon_2$ are functions from $\RPP$ to $[0,1]$,
let $\lambda \in\RPP$, and define 
$u_{\lambda}=\argmin_{\FF}(F+\lambda G)$. 
Let $\tau\in\RPP$, let $n\in\NN\smallsetminus\{0\}$, 
and let $\rho\in[\norm{u_{\lambda}},+\infty[$. Then
the following hold:
\vspace{-2ex}
\begin{enumerate}\setlength{\itemsep}{1pt} 
\item
\label{eq:sampleerr1}
$\!\PO^*\bigg[
\norm{u_{n,\lambda}(\ZZZ_n)-u_\lambda}
\geq\varepsilon_1(\lambda)+
(\widehat{\psi}_{\rho})^\natural\bigg(
\alpha_{n,\lambda}(\tau,\rho)
+\dfrac{\varepsilon_2(\lambda)}{\lambda} \bigg) \bigg] 
\leq e^{-\tau}$.
\item
\label{eq:sampleerr1.2}
$\!\PO^*\bigg(\!\big[\norm{u_{n,\lambda}(\ZZZ_n)}\leq\rho\big] 
\!\cap\!\bigg[F(u_{n,\lambda}(\ZZZ_n))\!-
\!F(u_\lambda)\geq
\varsigma
\lip{\ell}{\varsigma\rho}
\bigg(\!\varepsilon_1(\lambda)\!+\!(\widehat{\psi}_{\rho})^\natural
\bigg(\alpha_{n,\lambda}(\tau,\rho)\!+\!
\Frac{\varepsilon_2(\lambda)}{\lambda}\! 
\bigg)\!\bigg)\!\bigg]\!\bigg)$ $\leq\!e^{-\tau}$.
\item
\label{p:consistencyiii}
Suppose that $\ell \in \Loss_1(\XC\times\YC\times\YY,\PP)$ and let 
$c\in\RP$ be as in Definition~\ref{def:lossandrisk}\ref{def:Gpi}. 
Then 
\begin{equation}
\!\PO^*\bigg[F(u_{n,\lambda}(\ZZZ_n))-F(u_\lambda)
\geq\varsigma c
\bigg(\!\varepsilon_1(\lambda)+(\widehat{\psi}_{\rho})^\natural
\bigg(\alpha_{n,\lambda}(\tau,\rho)\!+\!
\Frac{\varepsilon_2(\lambda)}{\lambda}
\bigg)\!\bigg)\!\bigg]\leq e^{-\tau}.
\end{equation}
\end{enumerate}
\end{proposition}
\begin{proof}
\ref{eq:sampleerr1}:
Let $z=(x_i,y_i)_{1\leq i\leq n}\in(\XC\times\YC)^n$. Since
\begin{equation}
u_{n,\lambda}(z)\in\Argmin_\FF^{\varepsilon_1(\lambda)
\varepsilon_2(\lambda)}(F_n(\cdot,z )+\lambda  G),
\end{equation}
it follows from Proposition~\ref{p:5}\ref{p:5ii} and
Ekeland's variational principle 
\cite[Corollary~4.2.12]{Lucche2006} that there exists 
$v_{n,\lambda}\in\FF$ such that $\norm{u_{n,
\lambda}(z)-v_{n,\lambda}}\leq\varepsilon_1(\lambda)$ and
$\inf \norm{ \partial (F_n(\cdot, z)+\lambda  G)(v_{n,
\lambda})}\leq\varepsilon_2(\lambda)$.
We note that $\ell \in \Loss_\infty(\XC\times\YC\times\YY)$ by
Remark~\ref{r:contrisk}\ref{r:contriskiv}.
Hence, setting $\QQI=(1/n) \sum_{i=1}^n
\delta_{(x_i, y_i )}$, we derive from 
Theorems~\ref{thm:genrep}\ref{thm:genrepiii} and
\ref{thm:sensitivity2}\ref{thm:sensitivity2ii}
that there exists a measurable and $\PP$-a.s.~bounded function
$h_\lambda \colon\XC\times\YC \to \YY^*$ such that 
$\norm{h_\lambda}_{\infty}\leq \lip{\ell}{\varsigma \rho}$
and
\begin{equation}
\norm{v_{n,\lambda}-u_\lambda}\leq
(\widehat{\psi}_{\rho})^\natural\bigg( \frac{1}{\lambda}
\big\lVert \EE_\PP [\fmap h_\lambda]-\frac{1}{n} \sum_{i=1}^n
\fmap (x_i) h_\lambda(x_i,
y_i)\big\rVert+\frac{\varepsilon_2(\lambda)}{\lambda}
\bigg).
\end{equation}  
Thus, for every $z\in (\XC \times \YC)^n$
\begin{equation}\label{eq:stab}
\norm{u_{n,\lambda}(z)-u_\lambda}\leq
\varepsilon_1(\lambda)+(\widehat{\psi}_{\rho})^\natural\bigg(
\frac{1}{\lambda} \big\lVert \EE_\PP [\fmap h_\lambda]
-\frac{1}{n} \sum_{i=1}^n \fmap (x_i) h_\lambda(x_i,
y_i)\big\rVert+\frac{\varepsilon_2(\lambda)}{\lambda}
\bigg).
\end{equation}
Now consider the family of i.i.d.~random vectors
$(\fmap(X_i)h_\lambda(X_i,Y_i) )_{1\leq i\leq n}$, 
from $\Omega$ to $\FF^*$. Since
$\max_{1\leq i\leq n}\norm{\fmap(X_i)
h_\lambda(X_i,Y_i)}\leq \varsigma
\lip{\ell}{\varsigma\rho}$ $\PO$-a.s.,  Theorem
\ref{thm:Hoeffding} gives
\begin{equation}
\label{eq:12345}
\PO \Big[ \Big\lVert \EE_\PO [\fmap(X)h_\lambda(X,Y)]-\frac{1}{n}
\sum_{i=1}^n \fmap (X_i) h_\lambda(X_i, Y_i)\Big\rVert \geq
\lambda\alpha_{n,\lambda}(\tau,\rho) \Big]{\leq e^{-\tau}}.
\end{equation}
Hence, since $(\widehat{\psi}_{\rho})^\natural$ is increasing by
Proposition~\ref{prop:psinatural}\ref{psinat_ina0}, a fortiori 
we have
\begin{align}
\label{e:afY34nd93nl10a}
\nonumber\PO  \bigg[  \varepsilon_1(\lambda) 
&+(\widehat{\psi}_{\rho})^\natural \bigg( \frac 1 \lambda 
\bigg\lVert
\EE_\PP[\fmap h_\lambda]-\frac{1}{n}\sum_{i=1}^n \fmap (X_i)
h_\lambda(X_i,Y_i)\bigg\rVert+
\frac{\varepsilon_2(\lambda)}{\lambda}\bigg) \\
&\geq\varepsilon_1(\lambda)+(\widehat{\psi}_{\rho})^\natural \bigg(
\alpha_{n,\lambda}(\tau,\rho)+\frac{\varepsilon_2(\lambda)}
{\lambda}\bigg)\bigg]\leq e^{-\tau}.
\end{align}
Thus \ref{eq:sampleerr1} follows from \eqref{eq:stab} and 
\eqref{e:afY34nd93nl10a}.

\ref{eq:sampleerr1.2}:
Let $\omega\in\big[\norm{u_{n,\lambda}(\ZZZ_n)}\leq\rho \big]$.
Since $\norm{u_\lambda}\leq\rho$ and 
$\norm{u_{n,\lambda}(\ZZZ_n(\omega))}\leq\rho$, 
we have $\norm{A u_\lambda}_\infty\leq\varsigma\rho$ and 
$\norm{A u_{n,\lambda}(\ZZZ_n(\omega))}_\infty\leq\varsigma\rho$.
Hence, we derive from Assumption~\ref{asssec4}\ref{asssec4ii} 
that
$F(u_{n,\lambda}(\ZZZ_n(\omega)))-
F(u_\lambda)\leq \lip{\ell}{\varsigma \rho}
\norm{Au_{n,\lambda}(\ZZZ_n(\omega))-A u_\lambda}_\infty\leq
\varsigma \lip{\ell}{\varsigma
\rho} \norm{u_{n,\lambda}(\ZZZ_n(\omega))-u_\lambda}
$.
Thus, \ref{eq:sampleerr1.2} follows from \ref{eq:sampleerr1}.

\ref{p:consistencyiii}:
It follows from Remark~\ref{r:contrisk}\ref{rmk:lossloclip_vs_lpi}
that $\ell$ is globally Lipschitz continuous in the third
variable uniformly with respect to the first two and that 
$\sup_{\rho^\prime\in\RPP}\lip{\ell}{\rho^\prime}\leq c$. 
Hence, we derive from \eqref{eq:riskP} that 
$R$ is Lipschitz continuous on $L^1(\XC,\PP_\XC;\YY)$
with Lipschitz constant $c$. As a result, 
\begin{equation}
(\forall\omega\in\Omega)\quad
\label{eZhd7rT9-06a}
F(u_{n,\lambda}(\ZZZ_n(\omega)))-F(u_\lambda)\leq c
\norm{Au_{n,\lambda}(\ZZZ_n(\omega))-A u_\lambda}_\infty\leq
\varsigma c\norm{u_{n,\lambda}(\ZZZ_n(\omega))-u_\lambda}.
\end{equation}
Thus, the statement follows from \ref{eq:sampleerr1}.
\end{proof}

The following technical result will be required subsequently.
\begin{lemma}
\label{l:invert_alpha+}
Let $\alpha\colon\RP \to \RP$ and let $\gamma \in \RPP$ be such that,
for every $\tau \in \left]1,+\infty\right[$, 
$\alpha(\tau) \leq \gamma \tau$.
Let $\phi \in \mathcal{A}_0$, 
let $(\eta,\epsilon)\in \RPP\times\RP$, and suppose that 
$\phi^\natural(2 \gamma)<\eta$ and $\phi^\natural(2\epsilon)<\eta$.
Set $\tau_0=\phi(\eta^-)/(2 \gamma)$. 
Then $\phi^\natural(\alpha(\tau_0)+\epsilon)<\eta$.
\end{lemma}
\begin{proof}
Recalling Proposition~\ref{prop:psinatural}\ref{psi_psinat_eq}, we
derive from the inequalities $\phi^\natural(2\gamma)<\eta$ and
$\phi^\natural(2\epsilon)<\eta$ that $\tau_0>1$ and
$\phi(\eta^-)>2\epsilon$, respectively.  
Therefore, since $\gamma \tau_0 =
\phi(\eta^-)/2$, we have $\alpha(\tau_0)+\epsilon \leq \tau_0
\gamma+\epsilon=\phi(\eta^-)/2+\epsilon<\phi(\eta^-)$.
Again, by Proposition~\ref{prop:psinatural}\ref{psi_psinat_eq}, 
we obtain that $\phi^\natural(\alpha(\tau_0)+\epsilon)<\eta$.
\end{proof}

\begin{proposition}
\label{p:wconsistency} 
Suppose that Assumption~\ref{asssec4} holds, that 
Notation~\ref{dZhd7rT9-02} is in use, and
that $\ell(\cdot,\cdot,0)$ is bounded. Write
$\varepsilon=\varepsilon_1\varepsilon_2$, where
$\varepsilon_1$ and $\varepsilon_2$ are functions from $\RPP$ 
to $[0,1]$. Let $(\forall n\in\NN)$ $\rho_n \in \big[\psi_0^\natural\big(
(\norm{\ell(\cdot,\cdot,0)}_\infty +1)/\lambda_n\big), +\infty\big[\,$.
Then the following hold:
\vspace{-2ex}
\begin{enumerate}\setlength{\itemsep}{1pt} 
\item 
\label{p:wconsistencyi}
Let $\lambda \in\RPP$, and set 
$u_{\lambda}=\argmin_{\FF}(F+\lambda G)$ and let
$\rho \in \big[\psi_0^\natural \big((\norm{\ell(\cdot,\cdot,0)}_\infty+1)
/\lambda \big),+\infty\big[\,$. 
Let $\tau\in\RPP$ and let $n\in\NN\smallsetminus\!\{0\}$. Then
\begin{align}
\label{eq:wrate}
\nonumber\PO^*\Big[ F(u_{n,\lambda}(\ZZZ_n))-\inf F(\dom G) \geq
\varsigma&\lip{\ell}{\varsigma \rho} \big(
\varepsilon_1(\lambda)+(\widehat{\psi}_\rho)^\natural\big(
\alpha_{n,\lambda}(\tau,\rho)+\varepsilon_2(\lambda)/\lambda
\big)\big)\\
&+F(u_\lambda)-\inf F(\dom G)\Big]\leq e^{-\tau}.
\end{align}
\item
\label{p:wconsistencyii}
Suppose that \eqref{eq:conscond} and 
\eqref{eq:conscond2} hold. Then $F(u_{n,\lambda_n}(\ZZZ_n))
\overset{\PO^*}{\to}\inf F (\dom G)$.
\item
\label{p:wconsistencyiii}
Suppose that \eqref{eq:conscond} and 
\eqref{eq:conscond2s} hold.
Then $F(u_{n,\lambda_n}(\ZZZ_n))\to\inf F(\dom G)$
$\PO^*\verb0-0\text{a.s.}$
\end{enumerate}
\end{proposition}
\begin{proof}
\ref{p:wconsistencyi}: 
Since for every $z_n=(x_i,y_i)_{1\leq i\leq n}\in(\XC\times\YC)^n$,
$F_n(0, z_n)\leq \norm{\ell(\cdot, \cdot, 0)}_\infty$
and $F(0)\leq \norm{\ell(\cdot, \cdot, 0)}_\infty$, 
it follows from 
Proposition~\ref{prop:flambda2} that
$\norm{u_{n,\lambda}(\ZZZ_n)}\leq \rho$ 
and $ \norm{u_{\lambda}}\leq \rho$. 
Thus, Proposition~\ref{p:consistency}\ref{eq:sampleerr1.2} yields
$\PO^*\big[F(u_{n,\lambda}(\ZZZ_n))-F(u_\lambda)\geq\varsigma
\lip{\ell}{\varsigma\rho} \big(\varepsilon_1(\lambda)+
(\widehat{\psi}_{\rho})^\natural\big(
\alpha_{n,\lambda}(\tau,\rho)
+\varepsilon_2(\lambda)/\lambda\big)\big)\big]\leq e^{-\tau}$,
and \eqref{eq:wrate} follows.

\ref{p:wconsistencyii}:
Because of \eqref{e:afY34nd93nl10c}, conditions
\eqref{eq:conscond}-\eqref{eq:conscond2} imply that
\begin{equation}
(\forall\tau\in{[1,+\infty[})
\quad\varsigma\lip{\ell}{\varsigma \rho_n}
\big(\varepsilon_1(\lambda_n)+(\widehat{\psi}_{\rho_n})^\natural
\big(\alpha_{n,\lambda_n}(\tau,\rho_n)+\varepsilon_2
(\lambda_{{n}})/\lambda_{{n}} \big) \big) \to 0.
\end{equation}
Therefore, it follows from \eqref{eq:wrate} and
Proposition~\ref{p:4}\ref{p:4i} that for every
$(\eta,\tau)\in\RPP\times[1+\infty[$,
there  exists $\bar{n}\in\NN$ such that, for every integer
$n\geq\bar{n}$, $\PO^{*}\big[
F(u_{n,\lambda_n}(\ZZZ_n))-\inf F(\dom G)\geq
\eta\big]\leq e^{-\tau}$.
Hence, for every $(\eta,\tau)\in\RPP{\times[1,+\infty[}$,
$\varlimsup_{n\to+\infty}\PO^{*}
\big[ F(u_{n,\lambda_n}(\ZZZ_n))-
\inf F (\dom  G)\geq \eta \big]\leq e^{-\tau}$.
The convergence in outer probability follows.

\ref{p:wconsistencyiii}: 
Let $\eta\in\RPP$ and let $\xi\in\left] 1, +\infty \right[$.
It follows from \eqref{eq:conscond} and \eqref{eq:conscond2s} that
there exists an integer $\bar{n}\geq3$ such that,
for every integer $n\geq\bar{n}$, we have
\begin{equation}
\label{eq:eta_varepsilon2log}
\lip{\ell}{\varsigma\rho_n}
(\widehat{\psi}_{\rho_n})^\natural
\Big(\frac{2 \varsigma \xi (4T_{q^*}+5)\lip{\ell}{\varsigma\rho_n}
\log n}{\lambda_n n^{1/q}} \Big)<\eta\quad\text{and}\quad
\lip{\ell}{\varsigma \rho_n} (\widehat{\psi}_{\rho_n})^\natural
\Big(2 \frac{\varepsilon_2(\lambda_n)}{\lambda_n} \Big)< \eta\,.
\end{equation}
Let $n\in\NN$ be such that $n\geq\bar{n}$ and set 
$\gamma=\varsigma (4 T_{q^*}+5)\lip{\ell}{\varsigma
\rho_n}/(\lambda_n n^{1/q})$.
We derive from \eqref{e:afY34nd93nl10c} that
$(\forall\, \tau \in [1,+\infty)\ \alpha_{n,\lambda_n}(\tau,
\rho_n) \leq \tau \gamma$.
Then, since 
$1\leq\xi\log n$, it follows from Lemma~\ref{l:invert_alpha+} 
that
\begin{multline}
\tau_{0}=\widehat{\psi}_{\rho_n}
\bigg(\bigg(\frac{\eta}{\lip{\ell}{\varsigma \rho_n}}
\bigg)^{\!\!-} \bigg)
\frac{\lambda_n n^{1/q}}{2 \varsigma (4T_{q^*}+5)
\lip{\ell}{\varsigma
\rho_n}}\\ \Rightarrow\quad  \lip{\ell}{\varsigma \rho_n}
(\widehat{\psi}_{\rho_n})^\natural
\bigg(\alpha_{n,\lambda_n}(\tau_{0},\rho_n)
+\frac{\varepsilon_2(\lambda_n)}{\lambda_n}\bigg)<\eta.
\end{multline}
Now set
\begin{equation}
\Omega_{n,\eta}=\big[F(u_{n,\lambda_n}(\ZZZ_n))-\inf F(\dom G)>
\varsigma\lip{\ell}{\varsigma\rho_n}\varepsilon_1(\lambda_n)
+\varsigma\eta+F(u_{\lambda_n})-\inf F(\dom G)\big].
\end{equation}
Item \ref{p:wconsistencyi} yields
\begin{align}
\label{eq:wrate2.1}
\PO^*  \Omega_{n,\eta}\leq\exp\bigg(-\widehat{\psi}_{\rho_n} 
\bigg(\bigg(\frac{\eta}{\lip{\ell}{\varsigma \rho_n}} \bigg)^{\!\!-}
\bigg)\frac{\lambda_n n^{1/q}}{2 \varsigma (4T_{q^*}+5)
\lip{\ell}{\varsigma \rho_n}} \bigg).
\end{align}
We remark that, by 
Proposition~\ref{prop:psinatural}\ref{psi_psinat_eq}%
-\ref{psinat_ina0}, the first condition in 
\eqref{eq:eta_varepsilon2log} is equivalent to
\begin{equation}\label{eq:taugeqalogn}
\widehat{\psi}_{\rho_n} 
\bigg(\bigg(\frac{\eta}{\lip{\ell}{\varsigma \rho_n}}\bigg)^{\!\!-} 
\bigg) 
\frac{\lambda_n n^{1/q}}{2 \varsigma (4T_{q^*}+5)
\lip{\ell}{\varsigma \rho_n}}>\xi\log n\,.
\end{equation}
Thus it follows from \eqref{eq:wrate2.1} and 
\eqref{eq:taugeqalogn} that $\sum_{n=\bar{n}}^{+\infty} 
\PO^*\Omega_{n,\eta}\leq\sum_{n=\bar{n}}^{+\infty}
1/n^\xi<+\infty$.
Hence, using the Borel-Cantelli lemma 
(which holds for outer measures too) 
we conclude that $F(u_{n,\lambda_n}(\ZZZ_n))\to\inf F(\dom G)$
$\PO^*\verb0-0\text{a.s.}$
\end{proof}

The next proposition considers the case of a globally Lipschitz continuous 
loss $\ell$, and does not require the boundedness of $\ell(\cdot,\cdot,0)$.

\begin{proposition}
\label{p:wconsistencybis} 
Suppose that Assumption~\ref{asssec4} holds, that 
Notation~\ref{dZhd7rT9-02} is in use, and that 
$\ell\in\Loss_1(\XC\times\YC\times\YY;\PP)$. 
Let $c\in\RP$ be as in Definition~\ref{def:lossandrisk}\ref{def:Gpi}
and write $\varepsilon=\varepsilon_1\varepsilon_2$, where
$\varepsilon_1$ and $\varepsilon_2$ are functions from $\RPP$ 
to $[0,1]$. 
Let $(\forall n\in\NN)$ $\rho_n \in \big[\psi_0^\natural
((R(0)+1)/\lambda_n),+\infty\big[\,$.
Then the following hold:
\vspace{-2ex}
\begin{enumerate}\setlength{\itemsep}{1pt} 
\item 
\label{p:wconsistencybisi}
Let $\lambda \in\RPP$, set 
$u_{\lambda}=\argmin_{\FF}(F+\lambda G)$ and let
$\rho \in \big[\psi_0^\natural\big((F(0)+1)/\lambda \big),+\infty\big[\,$. 
Let $\tau\in\RPP$ and let $n\in\NN\smallsetminus\!\{0\}$. Then
\begin{align}
\label{eq:wratebis}
\nonumber\PO^*\Big[ F(u_{n,\lambda}(\ZZZ_n))-\inf F(\dom G) \geq
\varsigma c \big(
\varepsilon_1(\lambda)&+(\widehat{\psi}_\rho)^\natural\big(
\alpha_{n,\lambda}(\tau,\rho)+\varepsilon_2(\lambda)/\lambda
\big)\big)\\
&+F(u_\lambda)-\inf F(\dom G)\Big]\leq e^{-\tau}.
\end{align}
\item
\label{p:wconsistencybisii}
Suppose that \eqref{eq:conscond_l1} holds.
Then $F(u_{n,\lambda_n}(\ZZZ_n))
\overset{\PO^*}{\to}\inf F(\dom G)$.
\item
\label{p:wconsistencybisiii}
Suppose that \eqref{eq:conscond_l1} and \eqref{eq:conscond_l13} hold.
Then $F(u_{n,\lambda_n}(\ZZZ_n))\to\inf F(\dom G)$
$\PO^*\verb0-0\text{a.s.}$
\end{enumerate}
\end{proposition}
\begin{proof}
\ref{p:wconsistencybisi}:
By Proposition~\ref{prop:flambda2}, 
$\norm{u_{\lambda}} \leq \rho$. Thus, \eqref{eq:wratebis} follows 
from Proposition~\ref{p:consistency}\ref{p:consistencyiii}.

\ref{p:wconsistencybisii}-\ref{p:wconsistencybisiii}: 
Using \ref{p:wconsistencybisi}, these can be established as in 
the proof of 
Proposition~\ref{p:wconsistency}\ref{p:wconsistencyii}%
--\ref{p:wconsistencyiii}.
\end{proof}

\begin{proposition}
\label{p:sconsistency}
Suppose that Assumption~\ref{asssec4} holds, that
Notation~\ref{dZhd7rT9-02} is in use, and that
$S=\Argmin_{\dom G}F\neq\emp$.
Let $u^\dag=\argmin_{u\in S}G(u)$ and write
$\varepsilon=\varepsilon_1\varepsilon_2$, where
$\varepsilon_1$ and $\varepsilon_2$ are functions from $\RPP$ 
to $[0,1]$. For every 
$\lambda\in\RPP$, set $u_\lambda=\argmin_{\FF}(F+\lambda G)$. Let 
$\rho\in\left]\sup_{\lambda\in\RPP}\norm{u_\lambda},\pinf\right[$
and let $\tau\in\RPP$.
Then, for every sufficiently small $\lambda\in\RPP$ and every
$n\in\NN\smallsetminus\!\{0\}$,
\begin{equation}
\label{eq:strcons}
\PO^* \bigg[
\norm{u_{n,\lambda}(\ZZZ_n)-u^\dag}\geq \varepsilon_1(\lambda)
+(\widehat{\psi}_\rho)^\natural\bigg(\alpha_{n,\lambda}(\tau,\rho) 
+\dfrac{\varepsilon_2(\lambda)}{\lambda}\bigg)+
\norm{u_{\lambda}-u^\dagger}\bigg]\leq e^{-\tau}.
\end{equation}
Moreover, assume that \eqref{str_conscond} is satisfied. Then the
following hold:
\vspace{-2ex}
\begin{enumerate}\setlength{\itemsep}{1pt} 
\item
\label{p:sconsistencyi}
For every sufficiently large $n\in\NN$,
\begin{equation}
\label{eq:wstrcons}
\PO^*\bigg[F(u_{n,\lambda_n}(\ZZZ_n))-F(u^\dag)\geq\varsigma 
\lip{\ell}{\varsigma\rho} \bigg(
\varepsilon_1(\lambda_n)+ (\widehat{\psi}_\rho)^\natural\bigg( 
\alpha_{n,\lambda_n}(\tau,\rho)
+\dfrac{\varepsilon_2(\lambda_{n})}{\lambda_{n}}\bigg)\bigg)
+\lambda_n\bigg]\leq 2e^{-\tau}.
\end{equation}
\item 
\label{p:sconsistencyii}
$u_{n,\lambda}(\ZZZ_n)\overset{\PO^*}{\to} u^\dag$ and 
$F(u_{n,\lambda_n}(Z_n))\overset{\PO^*}{\to}
\inf F(\dom G)$. 
\item 
\label{p:sconsistencyiii}
Suppose that \eqref{e:34789v624389} holds. Then 
$F(u_{n,\lambda_n}(Z_n))\to\inf F(\dom G)$ $\PO^*\verb0-0
\text{a.s.}$ and $u_{n,\lambda_n}(\ZZZ_n)\to u^\dag$ 
$\PO^*\verb0-0\text{a.s.}$
\end{enumerate}
\end{proposition}
\begin{proof}
First note that items \ref{q34c87b} and \ref{p:9v} in 
Proposition~\ref{p:9} imply that $u^\dag$ is well defined and
that $\sup_{\lambda\in\RPP}\norm{u_\lambda}<+\infty$.
Now, let $\lambda \in\RPP$ and let $n\in\NN$.
Since $\norm{u_\lambda}\leq \rho$, it follows 
from Proposition~\ref{p:consistency}\ref{eq:sampleerr1} that
\begin{equation}
\label{eq:consreg}
\PO^*\Big[\norm{u_{n,\lambda}(\ZZZ_n)-u_{\lambda}} \geq
\varepsilon_1(\lambda)+(\widehat{\psi}_\rho)^\natural\big( 
\alpha_{n,\lambda}(\tau,\rho)+
\varepsilon_2(\lambda_n)/\lambda_n\big)\Big]\leq e^{-\tau}
\end{equation}
and, since
$\norm{u_{n,\lambda}(\ZZZ_n)-u^\dagger}\leq
\norm{u_{n,\lambda}(\ZZZ_n)
-u_{\lambda}}+\norm{u_{\lambda}-u^\dag}$,
\eqref{eq:strcons} follows. 
Note also that Proposition~\ref{p:8}\ref{p:8viii} implies that
$\widehat{\psi}_{\rho}\in\mathcal{A}_0$.

\ref{p:sconsistencyi}:
Let $\eta \in\RPP$ be such that
$\sup_{\lambda \in\RPP} \norm{u_\lambda}+\eta\leq\rho$.
It follows from \eqref{str_conscond}, \eqref{e:afY34nd93nl10c}, and 
Proposition~\ref{prop:psinatural}\ref{psinat_increasing}, that
$\varepsilon_1(\lambda_n)+(\widehat{\psi}_\rho)^\natural
\big(\alpha_{n,\lambda_n}(\tau,\rho)+
\varepsilon_2(\lambda_n)/\lambda_n\big)\to 0$. Hence, there 
exists $\bar{n}\in\NN$
such that for every integer $n\geq\bar{n}$,
$\varepsilon_1(\lambda_n)
+(\widehat{\psi}_\rho)^\natural\big(\alpha_{n,\lambda_n}(\tau,\rho)+
\varepsilon_2(\lambda_n)/\lambda_n\big)\leq\eta$.
Now, take an integer $n\geq\bar{n}$ and set 
$\Omega_n=\big[
\norm{u_{n,\lambda_n}(\ZZZ_n)-u_{\lambda_n}}\leq\eta \big]$. Then 
$\Omega_n\subset\big[\norm{u_{n,\lambda_n}(\ZZZ_n)}\leq\rho\big]$ 
and it follows from \eqref{eq:consreg} that 
$\PO^*(\Omega\setminus \Omega_n)\leq e^{-\tau}$.
Hence, we deduce from 
Proposition~\ref{p:consistency}\ref{eq:sampleerr1.2} that
\begin{equation}
\label{eq:sampleerr2}
\PO^* \Big[
F(u_{n,\lambda_n}(\ZZZ_n))-F(u_{\lambda_n})\geq\varsigma
\lip{\ell}{\varsigma \rho} \Big(
\varepsilon_1(\lambda_n)+(\widehat{\psi}_\rho)^\natural\big(
\alpha_{n,\lambda}(\tau,\rho)+\varepsilon_2(\lambda_n)/\lambda_n
\big)\Big)\Big]\leq 2 e^{-\tau}.
\end{equation}
On the other hand, Proposition~\ref{p:9}\ref{orderofinfI}
implies that, for $n$ sufficiently large,
$F(u_{\lambda_n})-F(u^\dag)\leq\lambda_n$, which combined with
\eqref{eq:sampleerr2} gives \eqref{eq:wstrcons}.

\ref{p:sconsistencyii}:
This follows from \eqref{eq:strcons} and \eqref{eq:wstrcons},
as in the proof of Proposition~\ref{p:wconsistency}\ref{p:wconsistencyii}.

\ref{p:sconsistencyiii}: 
Let $\eta \in \RPP$ and $\xi \in \left]1,\pinf\right[$.
Using \eqref{str_conscond} and \eqref{e:34789v624389} we 
obtain a version of \eqref{eq:eta_varepsilon2log} in which 
$\rho_n \equiv \rho$.
The proof of the fact that $F(u_{n,\lambda_n}(\ZZZ_n))\to F(u^\dag)$ 
$\PO^*\verb0-0\text{a.s.}$ then
follows the same line as that of 
Proposition~\ref{p:wconsistency}\ref{p:wconsistencyiii}.
Next, let $n \in\NN$ be sufficiently large so that
\begin{equation}
\label{eq:eta_varepsilon2.3}
(\widehat{\psi}_{\rho})^\natural \bigg(\frac{2\varsigma\xi 
(4T_{q^*}+5)
\lip{\ell}{\varsigma\rho}\log n}{\lambda_n n^{1/q}}\bigg)<\eta\quad
\text{and}\quad(\widehat{\psi}_{\rho})^\natural 
\bigg(\frac{2\varepsilon_2(\lambda_n)}{\lambda_n}\bigg)< \eta\,.
\end{equation}
Using Lemma~\ref{l:invert_alpha+}, upon setting
$\tau_{0}=({\widehat{\psi}_{\rho}(\eta^-\!)
\lambda_n n^{1/q}}/
(2 \varsigma(4T_{q^*}+5)\lip{\ell}{\varsigma\rho})$,
we obtain $(\widehat{\psi}_{\rho})^\natural 
\big(\alpha_{n,\lambda_n}(\tau_{0},\rho)+
{\varepsilon_2(\lambda_n)}/{\lambda_n} \big)<\eta$.
It then follows from \eqref{eq:strcons} and 
\eqref{eq:eta_varepsilon2.3}
that, for $n$ sufficiently large,
\begin{equation}
\PO^* \Big[
\norm{u_{n,\lambda_n}(\ZZZ_n)-u^\dag}>\varepsilon_1(\lambda_n)
+\eta+\norm{u_{\lambda_n}-u^\dagger}\Big]\leq\exp\bigg(-
\frac{\widehat{\psi}_{\rho} (\eta^-\!)\lambda_n n^{1/q}}{2
\varsigma(4T_{q^*}+5)\lip{\ell}{\varsigma\rho}}\bigg) 
<\frac{1}{n^\xi}.
\end{equation}
The conclusion follows by the Borel-Cantelli lemma.
\end{proof}\\

\noindent
{\bfseries Proof of Theorem~\ref{thm:main}.}
We first note that 
Proposition~\ref{prop:constraints}\ref{ite:constr2}
asserts that $\inf F(\dom G)=\inf R(\C)$.

\ref{thm:maini}:
This follows from
Proposition~\ref{p:wconsistency}\ref{p:wconsistencyii}--%
\ref{p:wconsistencyiii}.

\ref{thm:maini+}:
Remark~\ref{r:contrisk}\ref{rmk:lossloclip_vs_lpii} implies that, 
for every $\rho \in \RPP$, $\lip{\ell}{\rho}
\leq (p-1)\norm{b}_\infty+3 c p \max\{ 1, \rho^{\,p-1}\}$ and
$\ell(\cdot,\cdot,0)$ is bounded. Hence conditions
\eqref{eZhd7rT9-06r} and \eqref{eZhd7rT9-06t} imply 
\eqref{eq:conscond}-\eqref{eq:conscond2} 
and \eqref{eq:conscond2s} respectively. Therefore, the statement 
follows from \ref{thm:maini}.

\ref{thm:mainiii}:
This follows from
Proposition~\ref{p:wconsistencybis}\ref{p:wconsistencybisii}--%
\ref{p:wconsistencybisiii}.

\ref{thm:mainii}:
This follows from Proposition~\ref{p:sconsistency}%
\ref{p:sconsistencyii}--\ref{p:sconsistencyiii}.
\endproof

\noindent
{\bfseries Proof of Corollary~\ref{c:consJnorm}.}
Since $\FF$ is uniformly convex
of power type $q$, $\FF^*$ is uniformly smooth with modulus of
smoothness of power type $q^*$ \cite[p.~63]{LindTzaf1979} and hence
of Rademacher type $q^*$ (see Section~\ref{sec:notation}) in 
conformity with Assumption~\ref{asssec4}\ref{asssec4iii}.
Moreover, by \eqref{eZhd7rT8-29a}, the modulus of total convexity
$\psi_{\rho}$ of $ G$ on $B(\rho)$ is greater then that
of $\eta \norm{\cdot}^r$. Hence, by
Proposition~\ref{prop:modnormr}, 
\begin{equation}
\label{eq:ucoJrhob}
(\forall\rho \in\RP)(\forall t\in\RP)\qquad \psi_\rho(t)\geq 
\begin{cases}
\eta\beta t^r & \text{if}\;\; r\geq q \\[2ex] 
\dfrac{\eta\beta{t}^{q}}{(\rho+t)^{q-r}} & \text{if}\;\;r<q
\end{cases}
\end{equation}
and, for every
$\rho\in\RP$ and every $s\in\RP$,
\begin{equation}
\label{eq:invpsib}
(\widehat{\psi}_\rho)^{\natural}(s)\leq
\begin{cases}
\bigg(\dfrac{s}{\eta\beta}\bigg)^{1/(r-1)} & 
\text{if}\;\; r\geq q \\[2ex]
2^{q} \rho \max\bigg\{ \bigg( \dfrac{s}{\eta\beta \rho^{r-1}}
\bigg)^{1/(q-1)}, \bigg( \dfrac{s}{\eta\beta \rho^{r-1}} 
\bigg)^{1/(r-1)} \bigg\}& \text{if } r<q.
 \end{cases}
\end{equation}

\ref{c:consJnormi}:
It follows from \eqref{eq:ucoJrhob} that
\begin{equation}\label{eq:rhol}
(\forall\, n \in \NN)\qquad\psi_0^\natural
\left(\frac{\norm{\ell(\cdot,\cdot,0)}_\infty+1}
{\lambda_n}\right)\leq
\bigg(\frac{\norm{\ell(\cdot,\cdot, 0)}_\infty+1}
{\eta\beta\lambda_n}\bigg)^{1/r}=\rho_n.
\end{equation}
Now fix $\tau\in\RPP$ and assume that 
$\sup_{n\in\NN}\lip{\ell}{\varsigma\rho_n}>0$.
Since $\lip{\ell}{\varsigma \rho_n}^m/(\lambda_n^{m/r} n^{1/q})\to 0$
and $m \geq 2$, we have 
$\lip{\ell}{\varsigma\rho_n}/(\lambda_n^{m/r} n^{1/q})\to 0$.
Moreover, since $m/r\geq 1$, we have 
$\lip{\ell}{\varsigma \rho_n}/(\lambda_n n^{1/q})\to 0$
and, therefore, since $\rho_n\to\pinf$, there exists 
$\bar{n}\in\NN\smallsetminus\{0\}$ such that,
for every integer $n\geq \bar{n}$, 
$\tau \lip{\ell}{\varsigma \rho_n} / (\lambda_n n^{1/q})
\leq \eta\beta \rho_n^{r-1}$. Suppose that
$q>r$ and take an integer $n\geq\bar{n}$. 
Evaluating  the maximum in \eqref{eq:invpsib}, we obtain
\begin{equation}\label{eq:psinat_estimate}
(\widehat{\psi}_{\rho_n})^\natural
\left(\frac{\tau\lip{\ell}{\varsigma \rho_n}}{\lambda_n
n^{1/q}}\right)\leq 2^{q}
\left(\frac{\tau\rho_n^{\,{q}-r}}{\eta\beta}
\frac{\lip{\ell}{\varsigma \rho_n}}{\lambda_n
n^{1/q}}\right)^{1/(q-1)}.
\end{equation}
Therefore, substituting the expression of $\rho_n$ yields
\begin{equation}
\label{eq:majorizeconscond}
\lip{\ell}{\varsigma \rho_n}(\widehat{\psi}_{\rho_n})^\natural
\left(\frac{\tau\lip{\ell}{\varsigma \rho_n}}{\lambda_n n^{1/q}}\right)
\leq 2^{q}\tau^{\frac{1}{q-1}} 
\bigg(\dfrac{(\norm{\ell(\cdot,\cdot, 0)}_\infty+1)^{q/r -1}}{(\eta \beta)^{q/r}}
\frac{\lip{\ell}{\varsigma \rho_n}^{q}}{\lambda_n^{q/r} n^{1/q}}\bigg)^{\frac{1}{(q-1)}}.
\end{equation}
On the other hand, if $q\leq r$, \eqref{eq:invpsib} yields
\begin{equation}
\label{e:4b50c1743}
\lip{\ell}{\varsigma \rho_n}(\widehat{\psi}_{\rho_n})^\natural
\left(\frac{{\tau}\lip{\ell}{\varsigma \rho_n}} {\lambda_n n^{1/q}}\right)
\leq\bigg(\frac{\tau}{\eta\beta}\bigg)^{1/(r-1)}
\bigg(\frac{\lip{\ell}{\varsigma \rho_n}^{r}}{\lambda_n n^{1/q}}\bigg)^{1/(r-1)}.
\end{equation}
Thus, altogether \eqref{eq:majorizeconscond} and
\eqref{e:4b50c1743} imply that there exists 
$\gamma\in\RPP$ such that, for every integer $n\geq\bar{n}$
\begin{equation}
\label{eq:majorizeconscond+}
\lip{\ell}{\varsigma \rho_n}(\widehat{\psi}_{\rho_n})^\natural
\left(\frac{{\tau}\lip{\ell}{\varsigma \rho_n}} {\lambda_n n^{1/q}}\right)
\leq \gamma \tau^{1/(m-1)}
\bigg(\frac{\lip{\ell}{\varsigma \rho_n}^{m}}{\lambda_n^{m/r} n^{1/q}}\bigg)^{1/(m-1)}.
\end{equation}
It therefore follows from \eqref{eZhd7rT9-06x} that the right-hand
side of \eqref{eq:majorizeconscond+} converges to zero and hence that
\eqref{eq:conscond2} is fulfilled. 
Likewise, \eqref{eZhd7rT9-06y} implies \eqref{eq:conscond2s}.
Altogether the statement follows from 
Theorem~\ref{thm:main}\ref{thm:maini}.

\ref{c:consJnormiii}: 
It follows from Remark~\ref{r:contrisk}\ref{rmk:lossloclip_vs_lpii}
that $\ell(\cdot,\cdot,0)$ is bounded and that, for every 
$\rho\in\RPP$, $\lip{\ell}{\rho}\leq
(p-1)\norm{b}_\infty+3 c p\max\{1,\rho^{\,p-1}\}$.
Set $(\forall n\in\NN)$ 
$\rho_n=\big((\norm{\ell(\cdot,\cdot, 0)}_\infty+1)/(\eta
\beta\lambda_n)\big)^{1/r}$. Then 
$(\exi\gamma\in\RPP)(\forall n\in\NN)$
$\lip{\ell}{\rho_n}\leq \gamma/\lambda_n^{(p-1)/r}$. 
Thus, the statement follows from \ref{thm:maini}.

\ref{c:consJnormii}: Fix ${\tau} \in \RPP$ and set
$(\forall n\in\NN)$ 
$\rho_n=\big((R(0)+1)/(\eta
\beta\lambda_n)\big)^{1/r}$. Then \eqref{eq:ucoJrhob} yields 
$(\forall\, n \in\NN)$ 
$\psi_0^\natural\left((R(0)+1)/\lambda_n\right)\leq\rho_n$. 
Since ${m/r} \geq 1$,  $1/(\lambda_n^{{m/r}}
n^{{1}/q})\to 0$ implies $1/(\lambda_n
n^{1/q})\to 0$. Moreover, since $\rho_n \to +\infty$, there exists 
$\bar{n}\in\NN\smallsetminus\{0\}$ such that,
for every integer $n\geq \bar{n}$, 
${\tau} / (\lambda_n n^{1/q})
\leq \eta\beta \rho_n^{r-1}$. Suppose that
$q>r$ and take an integer $n\geq\bar{n}$. 
Evaluating the maximum in \eqref{eq:invpsib}, we obtain
\begin{equation}
\label{e:afY34nd93nl17a}
(\widehat{\psi}_{\rho_n})^\natural
\left(\frac{{\tau}}{\lambda_n n^{1/q}}\right)\leq 2^{q}
\left(\frac{{\tau}\rho_n^{\,{q}-r}}{\eta\beta}
\frac{1}{\lambda_n n^{1/q}}\right)^{\frac{1}{q-1}}=2^{q}
{\tau}^{\frac{1}{q-1}} {\bigg(
\dfrac{(R(0)+1)^{q/r-1}}{(\eta \beta)^q/r} \frac{1}{\lambda_n^{q/r}
n^{1/q}}\bigg)^{\frac {1}{q-1}}}.
\end{equation}
On the other hand, 
if $q\leq r$, \eqref{eq:invpsib} yields
\begin{equation}
\label{e:afY34nd93nl17b}
(\widehat{\psi}_{\rho_n})^\natural
\left(\frac{{\tau}}
{\lambda_n n^{1/q}}\right)\leq\bigg(\frac{{\tau}}{\eta \beta}
\frac{1}{\lambda_n n^{1/q}} \bigg)^{1/(r-1)},
\end{equation}
Thus \eqref{eZhd7rT8-29b}, together with \eqref{e:afY34nd93nl17a} and 
\eqref{e:afY34nd93nl17b} imply that \eqref{eq:conscond_l1} is fulfilled. 
Likewise, the assumption 
$\log n/(\lambda_n^{{m/r}} n^{{1}/q})\to 0$ implies that
\eqref{eq:conscond_l13} holds. Altogether, the statement follows by 
Theorem~\ref{thm:main}\ref{thm:mainiii}.
\endproof

\appendix
\section{Appendix}
\subsection{Lipschitz continuity of convex functions}

\begin{proposition}
\label{prop:Lip2Subrad}
Let $\B$ be a real Banach space and let $\genf\colon\B\to\RPX$
be proper and convex. Then the following hold:
\vspace{-2ex}
\begin{enumerate}\setlength{\itemsep}{1pt} 
\item
\label{propi:Lip2Subrad} 
\cite[Proposition~1.11]{Phelps93}
Let $u_0\in \B$, and suppose that there exist a neighborhood
$\mathcal{U}$ of $u_0$ and $c\in\RP$ such that, 
$(\forall\, u\in \mathcal{U})$ $\abs{\genf(u)-\genf(u_0)} 
\leq c\norm{u-u_0}$.
Then $\partial \genf(u_0)\neq\emp$ and 
$\sup\norm{\partial \genf(u_0)}\leq c$ .
\item
\label{propii:Bound2Lip} 
\cite[Corollary 2.2.12]{Zali02}
Let $u_0\in \B$, and suppose that, for some
$(\rho,\delta)\in\RPP^2$, $\genf$ is bounded on
$u_0+B(\rho+\delta)$. Then $\genf$ is Lipschitz continuous relative
to $u_0+B(\rho)$ with constant
\begin{equation}
\frac{ 2 \rho+\delta}{\rho+\delta} \frac 1 \delta
\sup \genf(u_0+B(\rho+\delta)).
\end{equation}
\end{enumerate}
\end{proposition}

\begin{proposition}
\label{prop:LipGrowthp}
Let $\B$ be a real normed vector space, let $p\in [1,+\infty[$, let
$b\in\RP$, let $c\in\RPP$, and let $\genf\colon \B \to \RP$ be a 
convex function such that $\genf\leq c\norm{\cdot}^p+b$.
Then the following hold:
\vspace{-2ex}
\begin{enumerate}\setlength{\itemsep}{1pt} 
\item
\label{prop:LipGrowthpii}
Let $u\in \B$. Then $\partial \genf(u) \neq \emp$ and 
\begin{equation}
\label{eq:subdiffGrowthp}
\sup \norm{\partial \genf(u)}\leq
\begin{cases}
c &\text{if}\;\; p=1\\
3 c p\max\{1,\norm{u}^{p-1}\}+(p-1)b   &\text{if}\;\;
p>1.
\end{cases}
\end{equation}
\item
\label{prop:LipGrowthpi}
Let $\rho\in\RPP$. Then $\genf$ is 
Lipschitz continuous relative to $B(\rho)$ with constant 
\begin{equation}
\label{eq:LipGrowthp}
\begin{cases}
c &\text{if}\;\; p=1\\
 3 c p \max\{1, \rho^{\,p-1}\}+(p-1) b   &\text{if}\;\; p
> 1.
\end{cases}
\end{equation}
\end{enumerate}
\end{proposition}
\begin{proof}
\ref{prop:LipGrowthpii}:
Let $(\epsilon,\delta)\in\RPP^2$.
Since $\genf\leq c \norm{\cdot}^p+ b$, then, $\genf$ is bounded
on $u+B(\epsilon+\delta)$ and it follows from Proposition
\ref{prop:Lip2Subrad}\ref{propii:Bound2Lip} that $\genf$ is
Lipschitz continuous relative to $u+B(\epsilon)$ with
constant $(2 \epsilon+\delta)(\epsilon+\delta)^{-1}
\delta^{-1} \big( c (\norm{u}+\epsilon+\delta)^p +b
\big)$. Then 
Proposition~\ref{prop:Lip2Subrad}\ref{propi:Lip2Subrad} entails that 
$\partial \genf(u)\neq\emp$ and 
\begin{equation}
\label{eq:LipGrowthp2}\sup\norm{\partial \genf(u)}\leq\frac{
2\epsilon+\delta}{\epsilon+\delta} \frac 1 \delta \big(
c(\norm{u}+\epsilon+\delta)^p+ b \big).
\end{equation}
Letting $\epsilon\to 0^+$ in \eqref{eq:LipGrowthp2}, we get
\begin{equation}\label{eq:LipGrowthp3}
\sup \norm{\partial \genf(u)}\leq c \Big(
\frac{\norm{u}}{\delta}+1 \Big)
(\norm{u} +\delta)^{p-1}+\frac b \delta .
\end{equation}
If $p=1$, letting $\delta\to +\infty$ in \eqref{eq:LipGrowthp3}
yields $\sup \norm{\partial \genf(u)}\leq c$. 
Now, suppose that $p>1$ and set $s=\max \{\norm{u},1\}$. Then, 
since  $\norm{u}\leq s$, \eqref{eq:LipGrowthp3} implies that
\begin{equation}
\sup\norm{\partial \genf(u)}\leq c\Big(\frac s \delta+1\Big) 
s^{p-1}\Big(1+\frac \delta s\Big)^{p-1}\!+\frac b \delta
\leq c\Big(\frac s\delta+1 \Big) s^{p-1} e^{\delta(p-1)/s}\! 
+\frac b\delta,
\end{equation}
where we took into account that $(1+\delta/s)^{s/\delta}\leq e$. 
By choosing $\delta=s/(p-1)$, we get 
$\sup \norm{\partial \genf(u)}\leq3 c p s^{p-1}+(p-1) b/s$ and
\eqref{eq:LipGrowthp} follows since $1/s\leq1$.

\ref{prop:LipGrowthpi}:
Let $(u,v)\in B(\rho)^2$. 
It follows from \ref{prop:LipGrowthpii}  that $\partial \genf(u) \neq
\emp$ and $\partial \genf(v)\neq \emp$. Let  $u^*\in\partial \genf(u)$ and
$v^*\in\partial \genf(v)$. Then 
$\genf(v)-\genf(u)\geq\pair{v-u}{u^*}$ and $\genf(u)-\genf(v) \geq
\pair{u-v}{v^*}$.
Hence $\abs{\genf(u)-\genf(v)}\leq\max\{\norm{u^*},\norm{v^*}\}
\norm{u-v}$ and the statement follows by \ref{prop:LipGrowthpii}.
\end{proof}

\begin{proposition}
\label{rmk:subploss}
Let $\B$ be a real Banach space, let $\rho\in\RPP$, let 
$p\in\left]1,\pinf\right[$, let $b\in\RP$, let $c\in\RPP$, and set
$\genf=c \norm{\cdot}^p+b$. Then $\genf$ is Lipschitz continuous 
relative to $B(\rho)$ with constant $c p\rho^{\,p-1}$. 
\end{proposition}
\begin{proof}
Let $(u,v)\in\B^2$ and let $u^*\in J_{\B,p}(u)$. Then 
\eqref{e:JJJ} yields
$\norm{u}^p-\norm{v}^p\leq p\pair{u-v}{u^*}\leq p
\norm{u^*} \norm{v-u}=p \norm{u}^{p-1}
\norm{u-v}$. Swapping $u$ and $v$ yields
\begin{equation}
\label{eq:Lipforploss}
\big|\norm{u}^p-\norm{v}^p\big|\leq 
p\max\big\{\norm{u}^{p-1},\norm{v}^{p-1}\big\}\norm{u-v},
\end{equation}
and the claim follows.
\end{proof}

\subsection{Totally convex functions}
\label{subsec:totconv}

Let $\FF$ be a reflexive real Banach space and let
 $\genff\colon\FF\to\RX$ be a proper convex function.
Following \eqref{def:totconvmodi},
we denote by $\psi\colon \dom \genff \times \RR \to \RPX$
the modulus of total convexity of $\genff$ 
and, following \eqref{def:modtotconvbound}, for every $\rho \in
\RPP$ such that $B(\rho)\cap \dom \genff \neq \emp$, we denote
by $\psi_\rho\colon \RR \to \RPX$ the modulus of total 
convexity of $\genff$ on $B(\rho)$. 
$\genff$ is \emph{totally convex at $u\in\dom \genff$} if, 
for every $t\in\RPP$, $\psi(u,t)>0$.
Moreover, $\genff$ is \emph{totally convex
on bounded sets} if, for every $\rho\in\RPP$ such that
$B(\rho)\cap\dom \genff\neq\emp$, $\genff$ is totally convex on $B(\rho)$,
meaning that $\psi_{\rho}>0$ on $\RPP$.
Total convexity and standard variants of convexity are related as
follows:
\vspace{-2ex}
\begin{itemize}
\setlength{\itemsep}{1pt} 
\label{rZhd7rT9-01i}
\item
Suppose that $\genff$ is totally convex at every point of $\dom \genff$. 
Then $\genff$ is strictly convex.
\item
\label{rZhd7rT9-01ii}
Total convexity is closely related to uniform convexity 
\cite{Vlad78,Zalinescu83}. Indeed $\genff$ is
uniformly convex on $\FF$ if and only if, for every 
$t\in\RPP$, $\inf_{u\in\dom \genff} \psi(u,t)>0$ 
\cite[Theorem~3.5.10]{Zali02}. Alternatively, $\genff$ is 
uniformly convex on $\FF$ if and only if 
$(\forall t\in\RPP)$ $\inf_{\rho\in\RPP}\psi_\rho(t)>0$.
\item 
\label{rZhd7rT9-01iv}
In reflexive spaces, total convexity on bounded sets is
equivalent to uniform convexity on bounded sets 
\cite[Proposition~4.2]{ButIusZal03}. Yet, some results
will require pointwise total convexity, which
makes it the pertinent notion in our investigation.
\end{itemize}

\begin{remark}
\label{rmk:totconv}
Let $u_0$ and $u$ be in $\dom  \genff$. Then
\eqref{def:totconvmodi} implies that
\begin{equation}
\label{eq:totconvineq}
\genff(u)-\genff(u_0)\geq \genff^\prime(u_0;u-u_0)+\psi(u_0, \norm{u-u_0}).
\end{equation}
Moreover, if $u^*\in\partial  \genff(u_0)$, 
$\pair{u-u_0}{u^*}\leq  \genff^\prime(u_0; u-u_0)$ and
therefore
\begin{equation}
\label{eq:totconvineq2}
\genff(u)-\genff(u_0)\geq \pair{u-u_0}{u^*}+\psi(u_0,\norm{u-u_0}).
\end{equation}
Thus, $\partial \genff(u_0)\neq\emp$ $\Rightarrow$
$\psi(u_0,\norm{u-u_0})<+\infty$.
\end{remark}

The following proposition collects some properties of the classes
$\mathcal{A}_0$ and $\mathcal{A}_1$ introduced in \eqref{e:A0} and
\eqref{e:A1} that are used to study the modulus of total
convexity.

\begin{proposition}
\label{prop:psinatural}
Let $\phi\in\mathcal{A}_0$. Then the following hold:
\vspace{-2ex}
\begin{enumerate}
\setlength{\itemsep}{1pt} 
\item
\label{psi_dom}
$\dom\phi$ is an interval containing $0$.
\item
\label{psinat_dom}  
$\dom \phi^\natural=\left [ 0, \sup \phi(\RP) \right[$.
\item
\label{psinat_fulldom} 
Suppose that $\widehat{\phi}$ is increasing on $\RP$. Then 
$\dom\phi^\natural=\RP$ and $\phi$ is strictly increasing on 
$\dom\phi$.
\item
\label{psi_forcing} 
Suppose that $(t_n)_{n\in\NN}\in\RP^\NN$ satisfies 
$\phi(t_n)\to 0$. Then $t_n\to 0$.
\item
\label{psinat_increasing} 
$\phi^\natural$ is increasing on $\RP$ and $\lim_{s\to 0^+}
\phi^\natural(s)=0=\phi^\natural(0)$.
\item
\label{psi_psinat_eq} 
Let $(s,t)\in{\RP\times}\RPP$. Then $\phi^\natural(s)<t$
$\Leftrightarrow$ $s<\phi(t^-)$.
\item
\label{psinat_ina0} 
Suppose that $\phi\in \mathcal{A}_1$. Then 
$\mathrm{int}(\dom\phi)\neq\emp$, $\widehat{\phi}\in\mathcal{A}_0$, 
$\widehat{\phi}$ is right-continuous at $0$, and
$(\widehat{\phi})^\natural\in \mathcal{A}_0$.
\end{enumerate}
\end{proposition}
\begin{proof}
\ref{psi_dom}: This follows from \eqref{e:A0}.

\ref{psinat_dom}: For every $s\in\RP$, $[\phi\leq s] \subset
\dom \phi$.  Therefore, if $\dom\phi$ is bounded, $\phi^\natural$
is real-valued. Now, suppose that $\dom\phi=\RP$. Let $s\in\RP$ with
$s<\sup \phi(\RP)$. Then there exists $t_1\in\RP$ such 
that $s<\phi(t_1)$. Moreover, since $\phi$ is
increasing, $t\in [\phi\leq s] \Rightarrow\phi(t)\leq s<
\phi(t_1) \Rightarrow t\leq t_1$. Hence,
$\phi^\natural(s)=\sup [\phi\leq s]\leq t_1 <+\infty$. 
Therefore
$\left [0, \sup \phi(\RP)\right[ \subset \dom\, \phi^\natural$.
On the other hand, if $s\in\left[\sup \phi(\RP),\pinf\right[$,
then $[\phi\leq s]=\dom \phi$ and hence
$\phi^\natural(s)=+\infty$.

\ref{psinat_fulldom}: 
For every $t\in[1,\pinf[$, $\phi(t)\geq t\phi(1)>0$. Hence 
$\sup\phi(\RP)=\pinf$ and therefore \ref{psinat_dom} yields
$\dom\phi^\natural=\RP$. Let $t\in \dom \phi$ and $s\in \dom \phi$
with $t<s$. If $t>0$, then $0<\phi(t)=t \widehat{\phi}(t)
\leq t \widehat{\phi}(s)=(t/s)\phi(s)<\phi(s)$; 
otherwise, \eqref{e:A0} yields $\phi(t)=\phi(0)=0<\phi(s)$.

\ref{psi_forcing}: Suppose that there
exist $\varepsilon\in\RPP$ and a subsequence 
$(t_{k_n})_{n\in\NN}$ such that $(\forall n\in\NN)$
$t_{k_n}\geq \varepsilon$.
Then $\phi (t_{k_n})\geq\phi( \varepsilon)>0$ and hence
$\phi(t_n)\not\to 0$. 

\ref{psinat_increasing}: 
See \cite[Lemma~3.3.1(i)]{Zali02}.

\ref{psi_psinat_eq}: Suppose that $t\leq \phi^\natural(s)$. Then
for every $\delta\in\left]0,t\right[$ there exists 
$t^\prime\in\RP$ such that
$\phi(t^\prime)\leq s$ and  $t-\delta<t^\prime$, hence
$\phi(t-\delta)\leq \phi(t^\prime)\leq s$. Therefore $0<
\sup_{\delta\in\left]0,t\right[} \phi(t-\delta)=
\phi(t^-)\leq s$. Conversely, suppose that $t>\phi^\natural(s)$. Let
$t^\prime\in\left]\phi^\natural(s),t\right[$. Then 
\eqref{eq:psidag} gives $\phi(t^\prime)>s$,
and hence $\phi(t^-)>s$.

\ref{psinat_ina0}: 
By \eqref{e:A0} and \eqref{e:A1},
$\mathrm{int}(\dom\phi)\neq\emp$, $\widehat{\phi}\in\mathcal{A}_0$,
and $\widehat{\phi}$ is continuous at $0$. Let $s\in\RPP$. In view
of \ref{psinat_increasing}, to prove that
$(\widehat{\phi})^\natural\in \mathcal{A}_0$, it remains to show
that $(\widehat{\phi})^\natural(s)>0$. By continuity of
$\widehat{\phi}$ at $0$, $\menge{t\in\RP}{\widehat{\phi}(t)\leq s}$
is a neighborhood of $0$ and hence $(\widehat{\phi})^\natural(s)=
\sup\menge{t\in\RP}{\widehat{\phi}(t)\leq s}>0$.
\end{proof}

The properties of the modulus of total convexity are summarized
below.

\begin{proposition}
\label{p:8}
Let $\FF$ be a reflexive real Banach space,
let $ \genff\colon\FF\to\RX$ be a proper convex function the domain of 
which is not a singleton, let $\psi$ be the modulus of total
convexity of $\genff$, and let $u_0\in\dom  \genff$. Then
the following hold:
\vspace{-2ex}
\begin{enumerate}
\setlength{\itemsep}{1pt} 
\item
\label{genincreasing} 
Let $c\in\left]1,\pinf\right[$ and let
$t\in\RP$. Then $\psi(u_0,ct)\geq c\psi(u_0,t)$.
\item
\label{increasing} 
$\psi(u_0,\cdot)\colon\RR\to\RPX$ is increasing on $\RP$.
\item
\label{phi_eqdef}
Let $t\in\RP$. Then
\begin{equation}
\psi(u_0,t)=
\inf\menge{\genff(u)-\genff(u_0)-\genff^\prime(u_0;u-u_0)}
{u\in\dom \genff,\:\norm{u-u_0}\geq t}.
\end{equation}
\item
\label{strictincreasing} 
Suppose that $\genff$ is totally
convex at $u_0$. Then $\psi(u_0,\cdot)\in \mathcal{A}_0$ and
$\psi(u_0,\cdot)^{\!\widehat{\phantom{a}}}\in \mathcal{A}_0$.
\item
\label{phidomain} 
$\dom\psi(u_0,\cdot)$ is an interval containing $0$; moreover, 
if $\partial \genff(u_0)\neq\emp$, then
$\inte\dom\psi(u_0,\cdot)\neq\emp$.
\item
\label{phiinA1} 
Suppose that $\partial \genff(u_0)\neq\emp$. Then 
$\lim_{t\to 0^+}\psi(u_0,\cdot)^{\!\widehat{\phantom{a}}}(t)=0$. 
\item
\label{phiinA1'} 
Suppose that $\partial \genff(u_0)\neq\emp$ and that
$G$ is totally convex at $u_0$.
Then $\psi(u_0,\cdot)\in\mathcal{A}_1$.
\item
\label{p:8viii}
Let $\rho\in\RPP$ and suppose that $\genff$ is totally convex on 
$B(\rho)$. Then $\psi_\rho\in\mathcal{A}_0$ and 
$\widehat{\psi}_{\rho}\in\mathcal{A}_0$. Moreover,
if $B(\rho)\cap\dom\partial \genff\neq\emp$, 
then $\psi_{\rho}\in\mathcal{A}_1$.   
\item
\label{p:8ix}
Suppose that $u_0\in\Argmin_{\FF} \genff$ and that $\genff$ is totally
convex at $u_0$. Then $\genff$ is coercive.
\end{enumerate}
\end{proposition}
\begin{proof}
\ref{genincreasing}:
Suppose that $u\in\dom \genff$ satisfies $\norm{u-u_0}=ct$
and set $v=(1-c^{-1})u_0+c^{-1} u=u_0+c^{-1}(u-u_0)$. Then
$v\in\dom \genff$ and $\norm{v-u_0}=t$.
Therefore, since $\genff$ is convex and 
$\genff^\prime(u_0; \cdot)$ is positively homogeneous 
\cite[Proposition~17.2]{Livre1},
\begin{align*}
\psi(u_0, t) &\leq  \genff(v)-\genff(u_0)-\genff^\prime(u_0; v-u_0) \\
&\leq (1-c^{-1}) \genff(u_0)+c^{-1} \genff(u)-\genff(u_0)-c^{-1}
\genff^\prime(u_0; u-u_0)\\
&=c^{-1} \big(\genff(u)-\genff(u_0)-\genff^\prime(u_0; u-u_0) \big)\,.
\end{align*}
Hence $c\psi(u_0,t)\leq\psi(u_0,ct)$.

\ref{increasing}: Let $(s,t)\in\RPP^2$ be such that $t<s$, and 
set $c=s/t$. Then using \ref{genincreasing}, we have 
$\psi(u_0, t)\leq c^{-1} \psi(u_0, c t)\leq \psi(u_0, s)$.

\ref{phi_eqdef}: 
Suppose that $u\in\dom \genff$ satisfies $\norm{u-u_0}\geq t$ and 
set $s=\norm{u-u_0}$. Then by \ref{increasing} we have
$\psi(u_0, t)\leq\psi(u_0, s)\leq \genff(u)-\genff(u_0)-\genff^\prime(u_0;u-u_0)$.

\ref{strictincreasing}: Since $\psi(u_0,0)=0$, \ref{increasing}
yields $\psi(u_0,\cdot)\in\mathcal{A}_0$. Moreover, it follows
from  \ref{genincreasing} that
$\psi(u_0,\cdot)^{\!\widehat{\phantom{a}}}$ is increasing, hence 
$\psi(u_0,\cdot)^{\!\widehat{\phantom{a}}}\in\mathcal{A}_0$.

\ref{phidomain}: The first claim follows from the fact that
$\psi(u_0,\cdot)$ is increasing and $\psi(u_0,0)=0$. Next, since
$\dom \genff$ is not a singleton, there exists $u\in\dom \genff, u \neq
u_0$. Finally, Remark~\ref{rmk:totconv} asserts that 
$\partial \genff(u_0)\neq\emp$ $\Rightarrow$ 
$\psi(u_0,\norm{u-u_0})<+\infty$.

\ref{phiinA1}: Since \ref{genincreasing} asserts that 
$\psi(u_0,\cdot)^{\!\widehat{\phantom{a}}}$ is increasing, 
$\lim_{t\to 0^+}\psi(u_0,\cdot)^{\!\widehat{\phantom{a}}}(t)=
\inf_{t\in\RPP}\psi(u_0,\cdot)^{\!\widehat{\phantom{a}}}(t)$. 
Suppose that
$\inf_{t\in\RPP}\psi(u_0,\cdot)^{\!\widehat{\phantom{a}}}(t)>0$. 
Then there exists $\epsilon\in\RPP$ such that, for every 
$t\in\RPP$, $\psi(u_0,t)\geq\epsilon t$. 
Let $u\in\dom \genff\smallsetminus\{u_0\}$. 
For every $t\in\left]0,1\right]$, define $u_t=
u_0+t v$, where $v=u-u_0$. Then $\epsilon t \norm{v}
=\epsilon \norm{u_t-u_0}\leq\psi(u_0,\norm{u_t-u_0})\leq
\genff(u_0+t v)-\genff(u_0)-\genff^\prime(u_0;tv)$.
Hence, since $\genff^\prime(u_0; \cdot)$ is positively homogeneous, 
$\epsilon \norm{v}+ \genff^\prime(u_0; v)\leq (\genff(u_0+t v)-
\genff(u_0))/t$. Letting $t \to 0^+$ yields
$\epsilon\norm{v}+\genff^\prime(u_0;v)\leq \genff^\prime(u_0; v)$,
which contradicts the facts that 
$\genff^\prime(u_0;v)\in\RR$ and $\epsilon\norm{v}>0$. 

\ref{phiinA1'}--\ref{p:8viii}:
The claims follow from \ref{strictincreasing} and 
\ref{phiinA1}.

\ref{p:8ix}:
Since $0\in \partial \genff(u_0)$, \eqref{eq:totconvineq2} yields
$(\forall u\in\dom \genff)$ $\psi(u_0,\norm{u-u_0})\leq \genff(u)-\genff(u_0)$.
On the other hand, since $\genff$ is also totally convex at $u_0$,
\ref{strictincreasing}-\ref{phidomain}
imply that there exists $s\in\RPP$ such that $0<\psi(u_0,s)<+\infty$
and $(\forall t\in[s,\pinf[)$ $\psi(u_0, t)\geq t\psi(u_0, s)/s$. 
Therefore, for every $u\in\dom \genff$ such that
$\norm{u-u_0}\geq s$, we have $\genff(u)\geq
\genff(u_0)+\norm{u-u_0} \psi(u_0,s)/s$, which implies that $\genff$ is
coercive.
\end{proof}

\begin{remark}
Statements \ref{genincreasing}, \ref{increasing}, \ref{phi_eqdef},
and \ref{phidomain} are proved in \cite[Proposition~2.1]{ButRes06}
with the additional assumption that $\inte\dom \genff\neq\emp$,
and in \cite[Proposition~1.2.2]{ButIus2000} with the
additional assumption that $u_0$ is in the algebraic
interior of $\dom \genff$.
\end{remark}

\begin{example}
\label{ex:totconv}
Let $\FF$ be a uniformly convex real Banach space and let 
$\phi\in\mathcal{A}_0$ be real-valued, strictly increasing,
continuous, and such that $\lim_{t\to +\infty}\phi(t)=+\infty$.
Define $(\forall t\in\RR)$ $\varphi(t)=\int_0^{|t|}\phi(s)\ud s$.
Then \cite[Theorem~4.1(ii)]{Zalinescu83} and 
\cite[Proposition~4.2]{ButIusZal03} imply that 
$\genff=\varphi\circ\norm{\cdot}$ is totally convex on bounded sets 
(see also \cite[Theorem~6]{Vlad78}).
\end{example}

We now provide an example of computation of the 
modulus of total convexity on balls.

\begin{proposition}
\label{prop:modnormr}
Let $q\in [2,+\infty[$ and let $\FF$ be a uniformly convex real
Banach space with modulus of convexity of power type $q$. Let
$r\in\left]1, +\infty\right[$ and for every $\rho\in\RP$, denote by
$\psi_\rho$
the modulus of total convexity of $\norm{\cdot}^r$ on the ball
$B(\rho)$. Then there exists $\beta\in\RPP$ such that
\begin{equation}
\label{eq:ucoJrho}
(\forall\rho \in\RP)(\forall t\in\RP)\qquad \psi_\rho(t)\geq 
\begin{cases}
\beta t^r & \text{if}\;\; r\geq q \\
\dfrac{\beta{t}^{q}}{(\rho+t)^{q-r}} & \text{if}\;\;r<q.
\end{cases}
\end{equation}
Hence $\norm{\cdot}^r$ is totally convex on bounded sets and, 
if $r\geq q$, it is uniformly convex. Moreover, for every
$\rho\in\RP$ and every $s\in\RP$,
\begin{equation}
\label{eq:invpsi}
(\widehat{\psi}_\rho)^{\natural}(s)\leq
\begin{cases}
\bigg(\dfrac{s}{\beta}\bigg)^{1/(r-1)} & 
\text{if }  r\geq q \\[2ex]
2^{q} \rho \max\bigg\{ \bigg( \dfrac{s}{\beta \rho^{r-1}}
\bigg)^{1/(q-1)}, \bigg( \dfrac{s}{\beta \rho^{r-1}} \bigg)^{1/(r
-1)} \bigg\}& \text{if } r<q.
 \end{cases}
\end{equation}
\end{proposition}
\begin{proof}
Let $(u,v)\in\FF^2$. We derive from \cite[Theorem~1]{XuRo1991} that
\begin{equation}
(\forall u^*\in J_{\FF,r}(u))\quad
\norm{u+v}^r-\norm{u}^r\geq r
\pair{v}{u^*}+\vartheta_r(u,v),
\end{equation}
where 
\begin{equation*}
\vartheta_r(u,v)=r K_r \int_0^1 \frac{\max\{\norm{u+t v},
\norm{u}\}^r}{t} \delta_{\FF}\bigg( \frac{t \norm{v}}{ 2
\max\{\norm{u+t v}, \norm{u}\}} \bigg) \ud t
\end{equation*}
and $K_r\in\RPP$ is the constant defined according to
\cite[Lemma~3, Equation (2.13)]{XuRo1991}. Since
$\delta_\FF(\varepsilon)\geq c \varepsilon^q$ for some $c\in\RPP$,
then
\begin{equation}
\vartheta_r(u,v)\geq \frac{r K_r c}{2^q} \norm{v}^q\int_0^1
\max\{\norm{u+t v}, \norm{u}\}^{r-q} t^{q-1}  \ud t.
\end{equation}
Suppose first that $r\geq q$. 
Since, $\forall\, t\in[0,1]$, 
$\max\{\norm{u+t v}, \norm{u}\}\geq  t \norm{v}/2$, 
\begin{equation}
\label{e:8o376132}
\vartheta_r(u,v)\geq \frac{r K_r c}{2^q} \norm{v}^q\int_0^1
\frac{t^{r-q}}{2^{r-q}} \norm{v}^{r-q} t^{q-1}  \ud t=\frac{r
K_r c}{2^r} \norm{v}^r\int_0^1  t^{r-1}  \ud t=\frac{K_r
c}{2^r} \norm{v}^r.
\end{equation}
Now, suppose that $r<q$. Then since, for every $t\in[0,1]$,  
$\max\{\norm{u+t v}, \norm{u}\}
\leq \norm{u}+\norm{v}$,
\begin{equation}
\label{e:8o8b9k32}
\vartheta_r(u,v)\geq \frac{r K_r c}{2^q} \norm{v}^q\int_0^1
\frac{1}{\max\{\norm{u+t v}, \norm{v}\}^{q-r}} t^{q-1}  \ud t
\geq \frac{r K_r c}{ q 2^q}
\frac{\norm{v}^q}{(\norm{u}+\norm{v})^{q-r}}.
\end{equation}
Let  $\psi$ be the modulus of total convexity of $\norm{\cdot}^r$. 
Then it follows from \eqref{e:8o376132} and \eqref{e:8o8b9k32} 
that
\begin{equation}\label{eq:ucoJ}
(\forall\, u\in\FF)(\forall\, t\in\RP)\quad\psi(u,t)\geq 
\begin{cases}
\dfrac{K_r c}{2^r} t^r & \text{if } q\leq r \\[2ex]
\dfrac{r}{q} \dfrac{K_r c}{2^{q}} 
\dfrac{{t}^{q}}{(\norm{u}+{t})^{q-r}} & \text{if } q>r.
\end{cases}
\end{equation}
Let $\rho\in\RPP$ and set
$\beta=(r/\max\{q,r\})K_rc/2^{\max\{q,r\}}$.
Then we obtain \eqref{eq:ucoJrho} by 
taking the infimum over $u\in B(\rho)$ in \eqref{eq:ucoJ}.
Thus, if $r\geq q$, the modulus of total convexity is
independent from $\rho$, and hence $\norm{\cdot}^r$ is uniformly
convex on $\FF$. On the other hand, if $r<q$, we deduce
that $\norm{\cdot}^r$ is totally convex on bounded sets. 
Hence, 
\begin{equation}
\label{eq:hatpsi}
(\forall\, t\in\RP)\quad\widehat{\psi}_\rho(t)\geq \begin{cases}
\beta t^{r-1} & \text{if } r\geq q \\
\dfrac{{\beta t}^{q-1}}{(\rho+{t})^{q-r}} & \text{if } r<q.
\end{cases}
\end{equation}
A simple calculation shows that, if $r<q$,
\begin{equation}
\label{eq:hatpsipr}
(\forall\, t\in\RP)\quad\widehat{\psi}_\rho(t)\geq \nu_\rho(t),\quad
\text{where}\;\: \nu_\rho(t)=\dfrac{\beta\rho^{r-1}}{2^{q}} 
\min\big\{ (t/\rho)^{q-1}, (t/\rho)^{r-1}\big\}.
\end{equation}
The function $\nu_\rho$ is strictly increasing and continuous on 
$\RP$, thus $\nu_\rho^\natural=\nu_\rho^{-1}$. 
Since for arbitrary functions $\psi_1 \colon \RP \to \RP$ and
$\psi_2 \colon \RP \to \RP$ we have $\psi_1\geq \psi_2 \Rightarrow
\psi_1^\natural\leq \psi_2^\natural$, we obtain \eqref{eq:invpsi}.
\end{proof}

\begin{remark}\
\label{rmk:100614}
\begin{enumerate}
\setlength{\itemsep}{1pt} 
\vspace{-2ex}
\item
\label{rmk:100614i}
An inspection of the proof of Proposition~\ref{prop:modnormr}
reveals that the constant $\beta$ is explicitly available in
terms of $r$ and of a constant depending on the space $\FF$. 
In particular, it follows from
\cite[Equation~(2.13)]{XuRo1991} that, when $r\in\left]1,2\right]$,
\begin{equation} 
K_r\geq4 (2+\sqrt{3}) \min\{ r(r-1)/2, (r-1) \log (3/2),
1-(2/3)^{r-1}\}>14(1-(2/3)^{r-1}),
\end{equation}
and when $r \in\left]2,+\infty\right[$
\begin{equation}
K_r \geq 4(2+\sqrt{3}) \min\{ 1, (r-1) (2-\sqrt{3}),
1-(2/3)^{\frac{r}{2}}\}14(1-(2/3)^{\frac{r-1}{2}})\,.
\end{equation}
As an example, for the case $\FF=l^r(\KK)$ and
$\norm{\cdot}_r^r$, with $r\in\left]1,2\right]$, 
since $\FF$  has modulus of
convexity of power type $2$ with $c=(r-1)/8$ \cite{LindTzaf1979},
we have $\beta\geq(7/32)r(r-1)(1-(2/3)^{r-1})$.
\item
In \cite[Theorem~1]{Xu1991} and \cite[Lemma~2 p.~310]{Beauzamy1985}
the case $r=q$ is considered. It is proved
that $\norm{\cdot}^r_\FF$ is uniformly convex and that its modulus
of uniform convexity, say $\nu$, satisfies $\nu(t)\geq\beta t^r$,
for every $t\in\RPP$.
\end{enumerate}
\end{remark}

\subsection{Tikhonov-like regularization}\label{sec:varreg}

In this section we work with the following scenario.

\begin{assumption}
\label{SH}
$\FF$ is a reflexive real Banach space, $F\colon\FF\to\RX$ 
is bounded from below, $ G\colon\FF\to\RPX$, $\dom G$ is not a
singleton, and $\dom F\cap\dom  G\neq\emp$. The function
$\varepsilon\colon\RPP\to [0,1]$ satisfies
$\lim_{\lambda\to 0^+}\varepsilon(\lambda)=0$ and, for every 
$\lambda\in\RPP$,
$u_\lambda\in\Argmin_{\FF}^{\varepsilon(\lambda)}(F+\lambda G)$.
\end{assumption}

We study the behavior of the regularized problem
\begin{equation}
\label{genvarregprob}
\minimize{u\in\FF}{F(u)+\lambda G(u)}
\end{equation}
as $\lambda\to 0^+$ in connection with the limiting problem 
\begin{equation}
\minimize{u\in\FF}{F(u)}.
\end{equation}
We present results similar to those of \cite{Atto96}
under weaker assumptions and with approximate solutions of 
\eqref{genvarregprob}, as opposed to exact ones. 
In particular, Proposition~\ref{p:4} does not require the family
$(u_\lambda)_{\lambda\in\RPP}$ to be bounded or $F$ to have 
minimizers. Indeed, although these are common requirements in 
the inverse problems literature, where the convergence
of the minimizers $(u_\lambda)_{\lambda\in\RPP}$ is relevant, 
from the statistical learning
point of view this assumption is not always appropriate. In that
context, as discussed in the introduction, it is primarily the 
convergence of the values $(F(u_\lambda))_{\lambda\in\RPP}$ to 
$\inf F(\FF)$ which is of interest. On the other
hand, when $(u_\lambda)_{\lambda\in\RPP}$ is bounded and when 
additional convexity properties are imposed on $G$, we provide 
bounds and strong convergence results.

\begin{proposition}
\label{p:4} 
Suppose that Assumption~\ref{SH} holds. Then
the following hold:
\vspace{-2ex}
\begin{enumerate}
\setlength{\itemsep}{1pt} 
\item
\label{p:4ii} 
$\lim_{\lambda\to 0^+}\inf(F+\lambda  G)(\FF)=\inf F(\dom  G)$.
\item
\label{p:4i} 
$\lim_{\lambda\to 0^+} F(u_\lambda)=\inf F(\dom  G)$.
\item
\label{p:4iii} 
$\lim_{\lambda\to 0^+}\lambda  G(u_\lambda)=0$.
\end{enumerate}
\end{proposition}
\begin{proof}
\ref{p:4ii}:
Since $\dom F\cap\dom  G \neq \emp$, $\inf(F+\lambda G)(\FF)<\pinf$.
Let $u\in\dom  G$. Then
\begin{align}
\label{e:prop4}
(\forall\lambda\in\RPP)\quad
\nonumber\inf F(\dom  G)&\leq F(u_\lambda)\leq F(u_\lambda)+
\lambda G(u_\lambda)
\leq\inf(F+\lambda  G)(\FF)+\varepsilon(\lambda)\\
&\leq F(u)+\lambda  G(u)+\varepsilon(\lambda).
\end{align}
Hence, 
$\inf F(\dom  G)\leq\varliminf_{\lambda\to0^+}
\big(\inf(F+\lambda G)(\FF)+\varepsilon(\lambda) \big)
\leq \varlimsup_{\lambda\to 0^+}\big(\inf(F+\lambda  G)
(\FF)+\varepsilon(\lambda)\big)\leq F(u)$.
Therefore, $\lim_{\lambda\to 0^+}\big(\inf(F+\lambda  G)
(\FF)+\varepsilon(\lambda)\big)=\inf F(\dom  G)$, 
and the statement follows.

\ref{p:4i}: This follows from \ref{p:4ii} and \eqref{e:prop4}.

\ref{p:4iii}: 
By \ref{p:4ii} and \eqref{e:prop4} we have
$\lim_{\lambda\to 0^+}F(u_\lambda)+\lambda G(u_\lambda)
=\inf F(\dom G)$ which, together with \ref{p:4i}, yields 
the statement. 
\end{proof}

\begin{remark}
Assume that $\inf F (\FF)=\inf F(\dom  G)$. Then
Proposition~\ref{p:4} yields $\lim_{\lambda\to 0^+}
F(u_\lambda)=\inf F(\FF)$ and $\lim_{\lambda\to 0^+}
\inf(F+\lambda G)(\FF)=\inf F(\FF)$. In particular the condition 
$\inf F(\FF)=\inf F(\dom  G)$ is satisfied in each of the
following cases:
\begin{enumerate}
\setlength{\itemsep}{1pt} 
\item 
The lower semicontinuous envelopes of $F+\iota_{\dom G}$ and
$F$ coincide \cite[Theorem~2.6]{Atto96}.
\item 
$\overline{\dom G}\supset\dom F$ and $F$ is upper semicontinuous 
\cite[Proposition~11.1(i)]{Livre1}.
\item 
$\Argmin_\FF F\cap \dom  G\neq \emp$.
\end{enumerate}
\end{remark}

\begin{proposition}
\label{p:9} 
Suppose that Assumption~\ref{SH} holds and set 
$S=\Argmin_{\dom  G} F$.
Suppose that $F$ and $G$ are weakly lower semicontinuous, 
that $G$ is coercive, and that 
$\varepsilon(\lambda)/\lambda\to 0$ as $\lambda\to 0^+$. Then 
\begin{equation}
S \neq\emp\quad\Leftrightarrow\quad
(\exi t\in\RR)(\forall \lambda\in\RPP)\quad  G(u_\lambda)\leq t.
\end{equation}
Now suppose that $S\neq\emp$. Then the following hold:
\begin{enumerate}
\setlength{\itemsep}{1pt} 
\item 
\label{q34c87b}
$(u_{\lambda})_{\lambda\in\RPP}$ is bounded and 
there exists a vanishing sequence 
$(\lambda_n)_{n\in\mathbb{N}}$ in $\RPP$ such that 
$(u_{\lambda_n})_{n\in\NN}$ converges weakly.
\item 
\label{deffdag}  
Suppose that $u^\dag\in\FF$, that 
$(\lambda_n)_{n\in\NN}$ is a vanishing sequence in 
$\RPP$, and that $u_{\lambda_n}\weakly u^\dag$. 
Then $u^\dag\in\Argmin_{S}G$.
\item
\label{limJ}
$\lim_{\lambda\to 0^+}G(u_\lambda)=\inf G(S)$.
\item
\label{orderofinfI}
$\lim_{\lambda\to 0^+} \big(F(u_\lambda)
-\inf F (\dom  G)\big)/\lambda=0$. 
\item
\label{p:9v}
Suppose that $G$ is strictly quasiconvex 
\cite[Definition~10.25]{Livre1}. Then there exists
$u^\dag\in\FF$ such that $\Argmin_{S} G=\{u^\dag\}$ and 
$u_\lambda\weakly u^\dag$ as $\lambda\to 0^+$. 
\item
\label{p:9vi} 
Suppose that $ G$ is totally convex on bounded sets. Then 
$u_\lambda\to u^\dag=\argmin_{S}G$ as $\lambda\to 0^+$. 
\end{enumerate}
\end{proposition}
\begin{proof}
Assume that $S\neq\emp$ and let $u\in S$. For every
$\lambda\in\RPP$, $F(u_\lambda)+\lambda G(u_\lambda)\leq F(u)
+\lambda G(u)+\varepsilon(\lambda)$, 
so that $u_\lambda\in \dom  G$ and
\begin{equation}\label{eq:flambda}
G(u_\lambda)\leq \frac{F(u)-F(u_\lambda)}{\lambda} +
\frac{\varepsilon(\lambda)}{\lambda}+ G(u)\leq  G(u) +
\frac{\varepsilon(\lambda)}{\lambda}.
\end{equation}
Thus, since $(\varepsilon(\lambda)/\lambda)_{\lambda\in\RPP}$ is
bounded, so is $( G(u_\lambda))_{\lambda\in\RPP}$. Hence
$(u_\lambda)_{\lambda\in\RPP}$ is in some sublevel set of $ G$. 
Conversely, suppose that there exists $t\in\RPP$ such that
$\sup_{\lambda\in\RPP} G(u_\lambda)\leq t$. It follows from the
coercivity of $ G$ that $(u_\lambda)_{\lambda\in\RPP}$ is bounded.
Therefore, since $\FF$ is reflexive, there exist $u^\dag\in\FF$ and
a sequence $(\lambda_n)_{n\in\NN}$ in $\RPP$ such that 
$\lambda_n\to 0$ and $u_{\lambda_n}\weakly u^\dag$. 
In turn, we derive from 
the weak lower semicontinuity of $F$ and 
Proposition~\ref{p:4}\ref{p:4i} that
\begin{equation}\label{eq:ImindomJ}
F(u^\dag)\leq \varliminf F(u_{\lambda_n})=\lim
F(u_{\lambda_n})=\inf F(\dom  G).
\end{equation}
Moreover, since $ G$ is weakly lower semicontinuous,
\begin{equation}
G(u^\dag)\leq \varliminf  G(u_{\lambda_n})\leq \varlimsup
G({u_{\lambda_n}})\leq t.
\end{equation} 
Hence $u^\dag\in\dom  G$ and it follows from \eqref{eq:ImindomJ} 
that $u^\dag\in S$. 

\ref{q34c87b}:
This follows from the reflexivity of $\FF$ and the boundedness
of $(u_\lambda)_{\lambda\in\RPP}$.

\ref{deffdag}:
Arguing as above, we obtain that \eqref{eq:ImindomJ} holds.
Moreover, for every $u\in S$, 
it follows from \eqref{eq:flambda} that,
since $ G$ is weakly lower semicontinuous and
$\varepsilon(\lambda_n)/\lambda_n \to 0$, 
\begin{equation}
\label{eq:udaginargmin}
G(u^\dag)\leq \varliminf G(u_{\lambda_n})\leq\varlimsup
G({u_{\lambda_n}})\leq G(u)<+\infty.
\end{equation} 
Inequalities \eqref{eq:ImindomJ} and \eqref{eq:udaginargmin} imply
that $u^\dag\in S$ and that \ref{deffdag} holds.

\ref{limJ}: It follows from \eqref{eq:udaginargmin} and 
\ref{deffdag} that $ G(u_{\lambda_n})\to\inf G(S)$.

\ref{orderofinfI}: 
Let $\lambda\in\RPP$. Since $u_\lambda$ is an
$\varepsilon(\lambda)\verb0-0$minimizer of $F+\lambda  G$, for
every $u\in\dom  G$, we have
\begin{equation}\label{eq:341}
\frac{F(u_\lambda)-\inf F(\dom  G)}{\lambda}
+G(u_\lambda)\leq \frac{F(u)-\inf F(\dom G)}{\lambda}
+G(u)+\frac {\varepsilon(\lambda)}{\lambda}.
\end{equation}
In particular, taking $u=u^\dag$ in \eqref{eq:341} yields
\begin{equation}
\label{eq:flambdaiii}
\frac{F(u_\lambda)-\inf F(\dom G)}{\lambda}
+G(u_\lambda)\leq G(u^\dag)+
\frac{\varepsilon(\lambda)}{\lambda}.
\end{equation}
Since $\varepsilon(\lambda)/\lambda\to 0$, passing to the 
limit superior in \eqref{eq:flambdaiii} as $\lambda\to 0^+$,  and 
using \ref{deffdag} and \ref{limJ}, we get
\begin{equation}
\varlimsup_{\lambda\to 0^+} \frac{F(u_\lambda)-\inf F(\dom  G)}
{\lambda}+G(u^\dag)
\leq  G(u^\dag),
\end{equation}
which implies \ref{orderofinfI}, since
$F(u_\lambda)-\inf F(\dom  G)\geq 0$.

\ref{p:9v}: 
It follows from \ref{q34c87b} and \ref{deffdag} that
$\Argmin_{S} G\neq\emp$. Since $S$ is convex and $G$ is
strictly quasiconvex, $\Argmin_{S}G$ reduces to a singleton
$\{u^\dag\}$ and \ref{deffdag} yields 
$u_\lambda\weakly u^\dag$ as $\lambda\to 0^+$. 

\ref{p:9vi}:
Since $(u_\lambda)_{\lambda\in\RPP}$ is bounded, it follows from 
\cite[Proposition~3.6.5]{Zali02} (see also \cite{ButIusZal03}) 
that there exists $\phi\in\mathcal{A}_0$ such that
\begin{equation}
(\forall\lambda\in\RPP)\quad\phi
\Big(\frac{\norm{u_\lambda-u^\dag}}{2}\Big)\leq
\frac{G(u^\dag)+G(u_\lambda)}{2}-G\Big(
\frac{u_\lambda+u^\dag}{2}\Big).
\end{equation}
Hence, arguing as in \cite[Proof of Proposition~3.1(vi)]{Sico00} 
and using \ref{p:9v} and the weak lower semicontinuity of 
$G$, we obtain $u_\lambda\to u^\dag$ as $\lambda\to 0^+$.
\end{proof}

\begin{remark}
If $\Argmin_\FF F\cap\dom G\neq\emp$, then 
$S=\Argmin_{\dom G}F=\Argmin_\FF F\cap\dom G$ and
$\Argmin_{S}G=\Argmin_{\Argmin_\FF F}G$
(see \cite[Theorem~2.6]{Atto96} for related results).
\end{remark}

The following proposition provides an estimate of the growth of 
the function $\lambda\mapsto\norm{u_\lambda}$ as $\lambda\to 0^+$ 
when the condition $\Argmin_{\dom G}F\neq \emp$ is possibly not
satisfied. 

\begin{proposition}
\label{prop:flambda2}
Suppose that Assumption~\ref{SH} holds, that $G$ is convex with 
modulus of total convexity $\psi$, and that there exists 
$u \in \FF$ such that $\Argmin_\FF G\cap\,\dom\,F=\{u\}$. Then
\begin{equation}
\label{prop:flambda2i}
(\forall\lambda\in\RPP)
\quad\norm{u_\lambda-u}\leq
\psi(u,\cdot)^\natural \bigg(\frac{F(u)-\inf F(\dom  G)+
\varepsilon(\lambda)}{\lambda} \bigg).
\end{equation}
\end{proposition}
\begin{proof} 
Let $\lambda\in\RPP$. Since $F(u_\lambda)+\lambda  G(u_\lambda)\leq
F(u)+\lambda  G(u)+\varepsilon(\lambda)$, we have
\begin{equation}\label{eq:boundingJ}
 G(u_\lambda)-G(u)\leq \frac{F(u)-F(u_\lambda) +
\varepsilon(\lambda)}{\lambda}\leq\frac{F(u)-\inf F(\dom  G) +
\varepsilon(\lambda)}{\lambda}.
\end{equation}
Hence, recalling \eqref{eq:totconvineq2} and noting that 
$u\in\Argmin_\FF G$ $\Leftrightarrow$ $0\in\partial  G(u)$, 
we obtain $\psi(u,\norm{u_\lambda-u})\leq 
(F(u)-\inf F (\dom  G)+\varepsilon(\lambda))/\lambda$
and the claim follows. 
\end{proof}

\subsection{Concentration inequalities in Banach spaces}
\label{sec:concentration}

This section 
 provides the Banach space valued versions of the classical
Hoeffding inequality. The proof is similar to those
of \cite[Theorem~6.14 and Corollary~6.15]{SteChi2008}, which deal 
with the Hilbert space case (see also \cite{Yuri95}). A closely related result is 
\cite[Corollary 2.2]{Bosq2000}.

\begin{theorem}[Hoeffding's inequality]\label{thm:Hoeffding} 
Let $(\Omega, \mathfrak{A}, \PO)$ be a probability space and let 
$\B$ be a separable real Banach space of Rademacher type 
$q\in\left]1,2\right]$ with Rademacher constant $T_q$. Let 
$(\beta,\sigma)\in\RPP^2$, let $n\in\NN\smallsetminus\{0\}$, 
let $(\CU_i)_{1\leq i\leq n}$ be a family of independent random 
variables from $\Omega$ to $\B$ 
satisfying $\max_{1\leq i\leq n}\norm{\CU_i}\leq \beta$  
$\PO$-a.s., and let $\tau\in\RPP$. Then the following hold:
\begin{equation}
\PO \bigg[ \bigg\lVert \frac 1 n \sum_{i=1}^n (\CU_i-\EE_\PO \CU_i)
\bigg\rVert\geq   \frac{4 \beta T_q}{ n^{1-1/q}}+2 \beta
\sqrt{\frac{2 \tau}{n}}+\frac{4 \tau \beta}{3 n}    \bigg]\leq
e^{-\tau}.
\end{equation}
\end{theorem}
\begin{proof}
For every $i\in\{1, \dots, n\}$, set $V_i=\CU_i-\EE_\PO \CU_i$, so
that $\EE_\PO V_i=0$, $\norm{V_i}\leq 2 \beta$  $\QQ$-a.s., and
$\EE_\PO \norm{V_i}^q\leq (2 \beta)^q$.  Set $\sigma = 2 \beta$.
It follows from Jensen's inequality and \cite[Proposition~9.11]{LedTal91} that
\begin{equation}
\bigg( \EE_\PO \bigg\lVert \sum_{i=1}^n \CV_i \bigg\rVert
\bigg)^q\leq \EE_\PO \bigg\lVert \sum_{i=1}^n \CV_i
\bigg\rVert^q\leq (2T_q)^q \sum_{i=1}^n \EE_\PO 
\norm{\CV_i}^q\leq (2 T_q)^q n \sigma^q.
\end{equation}
Hence $\EE_\PO \big\lVert \sum_{i=1}^n \CV_i \big\rVert\leq 2 T_q
\sigma n^{1/q}$.
Now let $t\in\RP$. Then
\begin{equation}
\sum_{i=1}^n \EE_\PO \big( e^{t \norm{\CV_i}}-1-t
\norm{\CV_i}\big) 
=\sum_{i=1}^n 
\sum_{m=2}^{+\infty} \frac{t^m}{m !} \EE_\PO \norm{\CV_i}^{m-q}
\norm{\CV_i}^q  
\leq n  \big( e^{2 t\beta}-1-2 t \beta\big)
\end{equation}
and, using  
\cite[Theorem~6.13]{SteChi2008} (see also \cite[Theorem~3.3.1]{Yuri95}), we obtain that, for every 
$\varepsilon\in\RPP$,
\begin{equation}
\label{eq:bern1}
\PO \bigg[ \bigg\lVert \sum_{i=1}^n \CV_i \bigg\rVert\geq n
\varepsilon \bigg]\leq \exp \bigg(-t \varepsilon n+2 t \sigma T_q
n^{1/q}+ n \big( e^{2 t \beta}-1-2 t \beta
\big)  \bigg)\,.
\end{equation}
For every $\varepsilon\in\RPP$ such that $\varepsilon n-2 T_q
\sigma n^{1/q}\geq 0$, the right-hand side of
\eqref{eq:bern1} reaches its minimum at
\begin{equation}
\label{2cwrey6T17a}
\bar{t}=\frac{1}{2 \beta} \log (1+\alpha),\quad \text{where}\quad
\alpha=\big( \varepsilon n-2 T_q \sigma n^{1/q} \big)
/(2 n \beta)\,.
\end{equation}
Moreover, as in \cite[Theorem~6.14]{SteChi2008}, one gets
\begin{equation}\label{eq:tbar}
-\bar{t}\varepsilon n+\bar{t}(b_qn)^{1/q}\sigma+n
 \big(e^{\bar{t}\beta}-1-\bar{t}\beta\big)
\leq-\frac{3n}{2}\frac{\alpha^2}{\alpha+3}\,.
\end{equation}
Now set
\begin{equation}\label{eq:defepsilon}
\gamma=\frac{\tau}{3 n } \quad\text{and}\quad
\varepsilon=\frac{2 \tau \beta}{3 n} \big( \sqrt{6/\gamma+1}+ 1
\big)+\frac{2 T_q \sigma}{n^{1-1/q}}. 
\end{equation}
Then $\varepsilon n-2 T_q \sigma n^{1/q}>0$ and 
\eqref{2cwrey6T17a} yield
\begin{equation}
\alpha=\frac{ 3 \gamma n}{2 \tau \beta} \Big( \varepsilon-\frac{2
T_q \sigma}{n^{1-1/q}} \Big)=\gamma+\sqrt{\gamma^2+6 \gamma},
\end{equation}
so that $\alpha^2=2\gamma (\alpha+3)=2 \tau
(\alpha+3)/(3 n)$. Thus, 
\eqref{eq:bern1} and \eqref{eq:tbar} yield
$\QQ \big[ \big\lVert \sum_{i=1}^n \CV_i
\big\rVert/n \geq \varepsilon \big]\leq e^{-\tau}$.
From \eqref{eq:defepsilon}, substituting the expression of $\gamma$
into that of $\varepsilon$, we obtain
\begin{equation}
\varepsilon=\sqrt{\frac{8  \tau
\beta^{2}}{n}+\frac{ 4\beta^2 \tau^2}{9 n^2}}+ \frac{2 \tau \beta}{3
n}+\frac{2 T_q \sigma}{n^{1-1/q}} \leq\frac{4 \tau \beta}{3
n}+2 \beta \sqrt{\frac{2  \tau}{n}}+ \frac{2 T_q
\sigma}{ n^{1-1/q}},
\end{equation}
and the statement follows.
\end{proof}


\end{document}